\documentclass[12pt,a4paper]{amsart}
\usepackage{amssymb,eucal}

\calclayout

\newcommand{\+}{\protect\nobreakdash-}
\renewcommand{\:}{\colon}

\newcommand{\ot}{\otimes}
\newcommand{\rarrow}{\longrightarrow}

\newcommand{\ocn}{\odot}

\newcommand{\bu}{{\text{\smaller\smaller$\scriptstyle\bullet$}}}

\newcommand{\lrarrow}{\mskip.5\thinmuskip\relbar\joinrel\relbar\joinrel
 \rightarrow\mskip.5\thinmuskip\relax}

\DeclareMathOperator{\Ext}{Ext}
\DeclareMathOperator{\Hom}{Hom}
\DeclareMathOperator{\Spec}{Spec}
\DeclareMathOperator{\Spf}{Spf}

\DeclareMathOperator{\qHom}{\mathcal H \mskip-.3\thinmuskip
  \text{\rmfamily\mdseries\fontshape{ui}\selectfont om}}
\DeclareMathOperator{\Cohom}{\mathfrak{Cohom}}
\DeclareMathOperator{\fHom}{\mathfrak{Hom}}

\newcommand{\qc}{{\operatorname{\mathrm{-qc}}}}

\DeclareMathOperator{\Com}{\mathsf{Com}}
\DeclareMathOperator{\Hot}{\mathsf{Hot}}
\DeclareMathOperator{\Ac}{\mathsf{Ac}}
\DeclareMathOperator{\Fil}{\mathsf{Fil}}

\newcommand{\Ab}{\mathsf{Ab}}

\newcommand{\A}{\mathcal A}
\newcommand{\C}{\mathcal C}
\newcommand{\D}{\mathcal D}
\newcommand{\F}{\mathcal F}
\newcommand{\G}{\mathcal G}
\newcommand{\J}{\mathcal J}
\newcommand{\M}{\mathcal M}
\newcommand{\cO}{\mathcal O}
\newcommand{\cP}{\mathcal P}

\newcommand{\fF}{\mathfrak F}
\newcommand{\fJ}{\mathfrak J}
\newcommand{\fM}{\mathfrak M}
\renewcommand{\P}{\mathfrak P}
\newcommand{\Q}{\mathfrak Q}
\newcommand{\fR}{\mathfrak R}
\newcommand{\fS}{\mathfrak S}
\newcommand{\U}{\mathfrak U}
\newcommand{\V}{\mathfrak V}
\newcommand{\X}{\mathfrak X}

\newcommand{\modl}{{\operatorname{\mathsf{--mod}}}}
\newcommand{\comodl}{{\operatorname{\mathsf{--comod}}}}
\newcommand{\contra}{{\operatorname{\mathsf{--contra}}}}
\newcommand{\qcoh}{{\operatorname{\mathsf{--qcoh}}}}
\newcommand{\ctrh}{{\operatorname{\mathsf{--ctrh}}}}
\newcommand{\lcth}{{\operatorname{\mathsf{--lcth}}}}

\newcommand{\tors}{{\operatorname{\mathsf{-tors}}}}
\newcommand{\ctra}{{\operatorname{\mathsf{--ctra}}}}
\newcommand{\secmp}{{\operatorname{\mathsf{-secmp}}}}

\newcommand{\qs}{{\mathsf{qs}}}
\newcommand{\pco}{{\mathsf{pco}}}
\newcommand{\pctr}{{\mathsf{pctr}}}
\newcommand{\bco}{{\mathsf{bco}}}
\newcommand{\bctr}{{\mathsf{bctr}}}
\newcommand{\bb}{{\mathsf{b}}}
\newcommand{\abs}{{\mathsf{abs}}}
\newcommand{\inj}{{\mathsf{inj}}}
\newcommand{\proj}{{\mathsf{proj}}}
\newcommand{\fl}{{\mathsf{fl}}}
\newcommand{\vfl}{{\mathsf{vfl}}}
\newcommand{\cta}{{\mathsf{cta}}}
\renewcommand{\cot}{{\mathsf{cot}}}
\newcommand{\lct}{{\mathsf{lct}}}
\newcommand{\lin}{{\mathsf{lin}}}
\newcommand{\al}{{\mathsf{al}}}
\newcommand{\clp}{{\mathsf{clp}}}

\newcommand{\sop}{\mathsf{op}}
\newcommand{\id}{\mathrm{id}}

\newcommand{\sA}{\mathsf A}
\newcommand{\sB}{\mathsf B}
\newcommand{\sC}{\mathsf C}
\newcommand{\sD}{\mathsf D}
\newcommand{\sE}{\mathsf E}
\newcommand{\sF}{\mathsf F}
\newcommand{\sS}{\mathsf S}

\newcommand{\bB}{\mathbf B}
\newcommand{\bW}{\mathbf W}

\newcommand{\boL}{\mathbb L}

\newcommand{\boZ}{\mathbb Z}
\newcommand{\boQ}{\mathbb Q}

\newcommand{\Section}[1]{\bigskip\section{#1}\medskip}
\theoremstyle{plain}
\newtheorem{thm}{Theorem}[section]

\newtheorem{lem}[thm]{Lemma}
\newtheorem{prop}[thm]{Proposition}
\newtheorem{cor}[thm]{Corollary}
\theoremstyle{definition}
\newtheorem{ex}[thm]{Example}

\newtheorem{rem}[thm]{Remark}

\begin{document}

\author{Leonid Positselski}

\address{Institute of Mathematics, Czech Academy of Sciences \\
\v Zitn\'a~25, 115~67 Praha~1 \\ Czech Republic}

\email{positselski@math.cas.cz}

\title{Philosophy of contraherent cosheaves}

\begin{abstract}
 Contraherent cosheaves are module objects over algebraic varieties
defined by gluing using the colocalization functors.
 Contraherent cosheaves are designed to be used for globalizing
contramodules and contraderived categories for the purposes of Koszul
duality and semi-infinite algebraic geometry.
 One major technical problem associated with contraherent cosheaves
is that the colocalization functors, unlike the localizations,
are not exact.
 The reason is that, given a commutative ring homomorphism $R\rarrow S$
arising in connection with a typical covering in algebraic geometry,
the ring $S$ is usually a flat, but not a projective $R$\+module.
 We argue that the relevant difference between projective and flat
modules, from the standpoint of homological algebra, is not that big,
as manifested by the flat/projective and cotorsion periodicity theorems.
 The difference becomes even smaller if one observes that the ring $S$
is often a very flat $R$\+module.
\end{abstract}

\maketitle

\tableofcontents

\section*{Introduction}
\medskip

\setcounter{subsection}{-1}
\subsection{{}}
 The co-contra duality is a fundamental phenomenon.
 In the context of commutative algebra, it manifests itself as
the duality between the local cohomology and the local homology,
or between the adic torsion and
completion~\cite{Harr,Mat,Mat2,GM,DG,PSY,Pmgm,Pcta,PMat,Pdc}.
 In the context of the theory of coalgebras, contramodules were
introduced by Eilenberg and Moore~\cite[Section~III.5]{EM}.
 Various generalizations, including in particular the generalization
to topological rings~\cite{Pweak,PR,PS1,Pproperf,Pcoun,PS3},
were studied in numerous recent publications, surveyed in
the paper~\cite{Prev}.

 In the context of homological algebra of cochain complexes,
the derived co-contra correspondence manifests itself, in particular,
in the form of triangulated equivalences between the homotopy
categories of unbounded complexes of injective and projective modules.
 This was studied in the papers~\cite{Jor,Kra,IK,Pfp,Pps}, with
quasi-coherent sheaf versions worked out in
the dissertation~\cite{Mur} and the book~\cite{Psemten}.

 One major discovery of the derived nonhomogeneous Koszul duality
theory~\cite{Pkoszul,Prel,Pksurv,GrLe} is that Koszul duality happens
on the comodule and contramodule sides.
 So does semi-infinite homological algebra, where comodules
are sufficient if one is only interested in the semi-infinite
\emph{homology}, but contramodules are required if one wants to
consider the semi-infinite \emph{cohomology}~\cite{Psemi}.

\subsection{{}}
 In the context of algebraic geometry, one would be interested in
globalizing Koszul duality and semi-infinite homological algebra to
schemes and stacks.
 On the comodule side, this can be done in the conventional way,
using quasi-coherent sheaves for gluing modules and comodules from
affine pieces.
 \emph{Contraherent cosheaves} were invented in~\cite{Pcosh} for
the purpose of globalizing Koszul duality and semi-infinite algebra
on the contramodule side.

 For example, the exposition on $\D$\+$\Omega$ duality
in~\cite[Appendix~B]{Pkoszul} was worked out over algebraic varieties
on the coderived side, but the variety was assumed to be affine
on the contraderived side.
 To fill this gap, the contraherent cosheaves would be needed.
 Similarly, the semi-infinite algebraic geometry was only
developed on the comodule/coderived side in the book~\cite{Psemten}
(as one can see, e.~g., by comparing the quadrality diagram
in~\cite[Theorem~5.7.1]{Pcosh} with the duality
in~\cite[Theorem~7.16]{Psemten}).
 Accordingly, no SemiExt (i.~e., double-sided derived SemiHom) functor
was constructed in~\cite{Psemten}, but only the SemiTor (i.~e.,
double-sided derived semitensor product); cf.\ the affine case
in the book~\cite{Psemi}.

\subsection{{}}
 One can observe that comodules are defined using the tensor product
functors, while coderived categories (in one of the approaches) are
defined using the infinite coproduct functors.
 Both the tensor products and the infinite coproducts commute with
localizations.
 Furthermore, the localization of a torsion module is a torsion module.
 This is one way to explain why the quasi-coherent sheaves, which
are glued from modules using the localization functors, are suitable
for use with the comodules, torsion modules, and coderived categories.

 Dual-analogously, contramodules over coalgebras are defined using
the covariant Hom functors, while contraderived categories (in one of
the approaches) are defined using the infinite product functors.
 Both the covariant Hom functors and the infinite products commute with
the colocalizations.
 Furthermore, the colocalization of a derived complete module (also
known as a contramodule) over a commutative ring with a fixed finitely 
generated ideal is again a derived complete module.
 This suggests that the contraherent cosheaves, which are glued from
modules using the colocalization functors, should be suitable for use
with the contramodules, derived complete modules, and contraderived
categories.

\subsection{{}} \label{comodules-and-contramodules-introd-subsecn}
 One can also observe that the quasi-coherent sheaves \emph{are}
comodules.
 Let $X$ be a scheme or stack endowed with a covering $U\rarrow X$ by
an affine scheme $U$ such that the Cartesian product $U\times_XU$ is
again an affine scheme.
 In particular, this holds for quasi-compact semi-separated schemes $X$
and their Zariski coverings.
 Then the ring of functions $\C=\cO(U\times_XU)$ has a natural structure
of a \emph{coring} over the ring of functions $A=\cO(U)$.
 The category of quasi-coherent sheaves on $X$ is naturally equivalent
to the category of $\C$\+comodules~\cite[Section~2]{KR}, \cite{KR2},
\cite[Example~2.5]{Pflcc}.
 Hence one possible (if not necessarily the optimal) way to define
contraherent cosheaves on $X$ is to view $\C$\+contramodules as such
contraherent cosheaves.

\subsection{{}}
 The aim of this paper is twofold.
 We explain the \emph{why} and the \emph{how}.
 Why are the contraherent cosheaves relevant?
 How does one overcome the main technical problem associated with
the concept?
 These are the questions we elaborate on.

 Having briefly discussed the key aspects of the \emph{why} in this
introduction above, let us now comment on the \emph{how}.

\subsection{{}}
 Let $R\rarrow S$ be a commutative ring homomorphism interpreted
geometrically as corresponding to a covering or a part of a covering.
 In the simple context of the Zariski topology on schemes, one
can suppose that the related morphism of affine schemes
$\Spec S\rarrow\Spec R$ is an open immersion.
 To make things even simpler and more explicit, let us assume
that $\Spec S$ is a \emph{principal} affine open subscheme in
$\Spec R$; in other words, we have $S=R[r^{-1}]$, where $r\in R$
is an element.

 Then the \emph{localization functor} $R\modl\rarrow S\modl$ takes
an $R$\+module $M$ to the $S$\+module $M[r^{-1}]=R[r^{-1}]\ot_RM$.
 The \emph{colocalization functor} $R\modl\rarrow S\modl$ takes
an $R$\+module $P$ to the $S$\+module $\Hom_R(R[r^{-1}],P)$.

 The quasi-coherent sheaves are glued from their affine module pieces
using the localization functors.
 The contraherent cosheaves are glued from their affine module pieces
using the colocalization functors.

 One immediately observes that the localization functor
$R[r^{-1}]\ot_R{-}$ is \emph{exact}, while the colocalization functor
$\Hom_R(R[r^{-1}],{-})$ is \emph{not exact}.
 Indeed, the $R$\+module $R[r^{-1}]$ is flat, but usually \emph{not} 
projective.

 The main technical problems of the theory of contraherent cosheaves
arise from this fact.
 The first consequence is that the category of quasi-coherent sheaves
$X\qcoh$ on a scheme $X$ is \emph{abelian}, but the category of
contraherent cosheaves $X\ctrh$ is only \emph{exact} (in Quillen's
sense).
 The unpleasant technical complication of the \emph{nonlocality of
contraherence}~\cite[Example~3.2.1]{Pcosh} also arises from here.

\subsection{{}}
 In fact, as far as flat modules go, the flat $R$\+module
$R[r^{-1}]$ has rather simple nature.
 In particular, its projective dimension never exceeds~$1$.
 The theory of \emph{very flat modules}~\cite{Pcosh,PSl1}, as well
as our work on a related topic of \emph{strongly flat} and
\emph{quite flat} modules~\cite{PSl2,Pcoun,HPS} (going back to
the papers~\cite{Trl0,Trl,BS,BS2}), is built around this observation.
 The branch of contraherent cosheaf theory based on very flatness is
called the \emph{locally contraadjusted contraherent cosheaves}.

 The main result in this direction is the \emph{Very Flat Conjecture},
formulated in~\cite{Pcosh} and proved in~\cite{PSl1}.
 It claims that the commutative ring homomorphisms $R\rarrow S$
such that $S$ is a very flat $R$\+module are ubiquitous in algebraic
geometry.
 In fact, if $S$ is a finitely presented $R$\+algebra and $S$ is
a flat $R$\+module, then $S$ is a very flat $R$\+module.
 So there are lots of commutative ring homomorphisms with this
property beyond the localizations $S=R[r^{-1}]$.

\subsection{{}}
 Another branch of the contraherent cosheaf theory, not relying on
very flatness but instead based on the belief that arbitrary flat
modules can be dealt with successfully, is called the \emph{locally
cotorsion contraherent cosheaves}.

 The main results supporting the thesis that flat modules are not
that much more complicated than projective ones are called
the \emph{periodicity theorems} in homological algebra.
 The two most important ones for the contraherent cosheaf purposes
are the \emph{flat/projective periodicity theorem}~\cite{BG,Neem}
and the \emph{cotorsion periodicity theorem}~\cite{BCE}.
 A series of the present author's recent papers on
the topic~\cite{BHP,PS6,Pal,Pgen,Pflcc,Pacc,Plce,Pres,Pcor} was intended
to achieve a better understanding of the subject, with the applications
to contraherent cosheaves in mind.

 So far, the point seems to be that the flat modules are the direct
limits of (finitely generated) projective ones.
 The known results on the periodicity theorems suggest that
the passage to the direct limit closure is a relatively harmless
operation from the homological algebra standpoint, tending not
to complicate things too much.

 In particular, even though the projective dimensions of flat modules
can well be infinite in general, the passage from projective modules
to flat ones as the direct limits behaves in many ways similarly
to a passage to an ambient exact category of objects of finite
resolution dimension with respect to the original one.
 This is one of the main points which the exposition in this paper
is intended to emphasize.

\subsection{{}}
 In Sections~\ref{comodules-and-contramodules-secn}\+-%
\ref{coderived-and-contraderived-secn} we explain the motivation for
introducing contraherent cosheaves.
 The contraherent cosheaves are important devices for globalizing
contramodules and contraderived categories over nonaffine schemes.

 The required preliminary material for the definition of a contraherent
cosheaf is introduced in Section~\ref{preliminaries-secn}.
 The main definitions concerning contraherent cosheaves are spelled
out in Section~\ref{main-definition-secn}, and the main technical
problem arising in the context of these definitions is discussed in
this section.

 Sections~\ref{very-flat-secn}\+-\ref{flat-modules-secn} purport to
explain how this technical problem can be overcome.
 The discussion in the two sections centers on the two branches of
the theory of contraherent cosheaves, viz., the locally contraadjusted
contraherent cosheaves and the locally cotorsion contraherent cosheaves,
and aims to convince the reader that both the approaches are viable
and can be expected to work well.

 Finally, in Section~\ref{cotorsion-periodicity-secn} we sketch
an application of contraherent proof: a local, or rather \emph{colocal}
proof of the cotorsion periodicity theorem for quasi-coherent sheaves
on quasi-compact semi-separated schemes.

\subsection*{Acknowledgement}
  The author is supported by the GA\v CR project 23-05148S and
the Czech Academy of Sciences (RVO~67985840).

\Section{Why Contraherent Cosheaves? Comodules and Contramodules}
\label{comodules-and-contramodules-secn}

\subsection{Basics} \label{comodules-contramodules-basics-subsecn}
 Let $R$ be a commutative ring.
 A \emph{coalgebra} $\C$ over $R$ is an $R$\+module endowed with
two maps of \emph{comultiplication} $\mu\:\C\rarrow\C\ot_R\C$ and
\emph{counit} $\epsilon\:\C\rarrow R$, which must be $R$\+module
maps satisfying the usual coassociativity and counitality axioms.
 A \emph{left\/ $\C$\+comodule} $\M$ is an $R$\+module endowed with
a \emph{left coaction} map $\nu\:\M\rarrow\C\ot_R\M$, which must
be an $R$\+module map satisfying the usual coassociativity and
counitality axioms.

 A \emph{left\/ $\C$\+contramodule} $\P$ is an $R$\+module endowed
a \emph{left contraaction} map $\pi\:\Hom_R(\C,\P)\rarrow\P$,
which must be an $R$\+module map satisfying the contraassociativity
and contraunitality axioms.
 We refer to~\cite[Section~8]{Pksurv} or~\cite[Sections~1.1\+-1.2
and~2.5]{Prev} for the details.

 The category of left $\C$\+comodules $\C\comodl$ is abelian with
an exact forgetful functor $\C\comodl\rarrow R\modl$ \emph{if and
only if} $\C$ is a flat $R$\+module.
 The category of left $\C$\+contramodules $\C\contra$ is abelian with
an exact forgetful functor $\C\contra\rarrow R\modl$ \emph{if and
only if} $\C$ is a projective $R$\+module~\cite[Proposition~2.12]{Prev}.

 If $\C$ is a finitely generated projective $R$\+module, then
the categories of left $\C$\+comodules and left $\C$\+contramodules
are equivalent to each other and to the category of left modules
over the $R$\+algebra $\Hom_R(\C,R)$.
 This is a trivial, uninteresting case from the coalgebra theory
perspective.

 Both comodules and contramodules play a crucial role in the contexts
of nonhomogeneous Koszul duality~\cite{Pkoszul,Prel,Pksurv} and
semi-infinite homological algebra~\cite{Psemi,Prev}.

\subsection{Discussion} \label{contramodules-discussion-subsecn}
 How does one globalize the definitions from
Section~\ref{comodules-contramodules-basics-subsecn} to a scheme $X$
replacing a commutative ring~$R$\,?

 Let us start with a coalgebra before passing to comodules and
contramodules.
 One can say that a quasi-coherent coalgebra $\C$ over a scheme $X$ is
a quasi-coherent $\cO_X$\+module endowed with two maps of
\emph{comultiplication} $\mu\:\C\rarrow\C\ot_{\cO_X}\C$ and counit
$\epsilon\:\C\rarrow\cO_X$, which must of course be
$\cO_X$\+module maps.
 The coassociativity and counitality axioms can be written down
in this context in a quite obvious way.

 Similarly, a quasi-coherent left comodule $\M$ over
a quasi-coherent coalgebra $\C$ over $X$ is a quasi-coherent
$\cO_X$\+mod\-ule endowed with a \emph{left coaction} map
$\nu\:\M\rarrow\C\ot_{\cO_X}\M$, which must be an $\cO_X$\+module map
satisfying the obvious coassociativity and counitality equations.
 Then the category of quasi-coherent left $\C$\+comodules $\C\comodl$
is abelian with an exact forgetful functor $\C\comodl\rarrow X\qcoh$
if and only if $\C$ is a flat quasi-coherent sheaf.

 But what is a \emph{$\C$\+contramodule}?
 We would argue that an attempt to define the notion of
a ``quasi-coherent sheaf on $X$ with a $\C$\+contramodule structure''
runs into difficulties (unless we are in the trivial special case
when $\C$ is a locally free $\cO_X$\+module of finite rank).

 What does it mean that a quasi-coherent sheaf $\cP$ on $X$ is
a $\C$\+contramodule?
 Following the definition in
Section~\ref{comodules-contramodules-basics-subsecn}, one has to
consider some kind of ``Hom object'' from $\C$ to~$\cP$.

 The internal $\qHom$ sheaf from a coherent sheaf to a quasi-coherent
one (say, on a Noetherian scheme) is quasi-coherent, but it is not
interesting for us to assume $\C$ to be a coherent sheaf.
 The internal $\qHom$ sheaf from one quasi-coherent sheaf to another
one is not even quasi-coherent.

 One can apply the coherator functor~\cite[Section~B.12]{TT} to
the internal Hom sheaf $\qHom_{\cO_X}(\C,\cP)$, and obtain
the quasi-coherent internal Hom sheaf $\qHom_{X\qc}(\C,\cP)$.
 What happens if one defines a quasi-coherent $\C$\+contramodule $\cP$
as a quasi-coherent sheaf on $X$ endowed with a contraaction map
$\pi\:\qHom_{X\qc}(\C,\cP)\rarrow\cP$\,?
 Notice that this is \emph{not} the same thing as having a map
of sheaves of $\cO_X$\+modules $\qHom_{\cO_X}(\C,\cP)\rarrow\cP$,
as the coherator functor is adjoint to the inclusion of the category
of quasi-coherent sheaves into the category of sheaves of
$\cO_X$\+modules \emph{on the right} and not not on the left side.

 Still, it appears that one can spell out the contraassociativity and
contraunitality axioms in this context without running into
any problems.
 Indeed, such definition of a contramodule over a coalgebra makes sense
in any closed monoidal category; and the category of quasi-coherent
sheaves $X\qcoh$ is a closed monoidal category with respect to
the tensor product~$\ot_{\cO_X}$ and $\qHom_{X\qc}$ operations.

 But the quasi-coherent internal Hom operation is not local with
respect to restrictions to open subschemes.
 Suppose given a quasi-coherent contramodule $\cP$ over $\C$ in
the sense of the definition above.
 We would argue that there is \emph{no} way to restrict $\cP$ to
an open subscheme $U\subset X$, at least, staying on the level of
abelian or exact (rather than derived) categories.

 Indeed, let $j\:U\rarrow X$ denote the open immersion morphism.
 Given a contraaction map $\pi\:\qHom_{X\qc}(\C,\cP)\rarrow\cP$, one
can apply the inverse image functor~$j^*$ and obtain the induced
map $j^*\pi\:j^*\qHom_{X\qc}(\C,\cP)\rarrow j^*\cP$.
 How does one pass from that to a map $\qHom_{U\qc}(j^*\C,j^*\cP)
\rarrow j^*\cP$\,?
 The natural map between the quasi-coherent sheaves
$j^*\qHom_{X\qc}(\C,\cP)$ and $\qHom_{U\qc}(j^*\C,j^*\cP)$ goes in
the other direction: $j^*\qHom_{X\qc}(\C,\cP)\rarrow
\qHom_{U\qc}(j^*\C,j^*\cP)$.

 Notice that, in any event, the category of quasi-coherent
contramodules over $\C$ (as defined above) is abelian with an exact
forgetful functor to $X\qcoh$ \emph{if and only if} the functor
$\Hom_{X\qc}(\C,{-})\:X\qcoh\rarrow X\qcoh$ is exact.
 The latter condition is seldom satisfied.
 For example, if $\C\simeq\coprod_{\xi\in\Xi}\cO_X$ is isomorphic, as
an $\cO_X$\+module, to a coproduct of an infinite number of copies of
the structure sheaf $\cO_X$, then the functor $\Hom_{X\qc}(\C,{-})$
is isomorphic to the functor of infinite product of copies of
a quasi-coherent sheaf $\prod_{\xi\in\Xi}({-})\:X\qcoh\rarrow X\qcoh$.
 The latter functor is rarely exact.
 So one is not likely to obtain an abelian category of quasi-coherent
contramodules.

 All in all, someone may want to try to pick up this definition of
a ``quasi-coherent contramodule over a quasi-coherent coalgebra''
and build a theory of it.
 But from our perspective it appears to be less promising than
the notion of a \emph{contraherent cosheaf} on $X$ with
a $\C$\+contramodule structure.
 The latter can be defined using the notion of the functor
$\Cohom$ from a quasi-coherent sheaf to a contraherent cosheaf;
see~\cite[Section~2.4]{Pcosh} and Section~\ref{cohom-subsecn} below.
 Notice that some conditions must be imposed either on
a quasi-coherent sheaf $\F$, or on a contraherent cosheaf $\P$,
or on both, to make the construction of the contraherent cosheaf
$\Cohom_X(\F,\P)$ well-defined.

\Section{Why Contraherent Cosheaves? Torsion and Derived Complete
Modules}

\subsection{Basics} \label{torsion-completion-basics-subsecn}
 Let $R$ be a commutative ring and $I\subset R$ be an ideal.
 Then an $R$\+module $M$ is said to be \emph{$I$\+torsion} if
for each elements $m\in M$ and $s\in I$ there exists an integer
$n\ge1$ such that $s^nm=0$.
 When the ideal $I$ is finitely generated, this condition can be
rephrased by saying that for every element $m\in M$ there exists
an integer $n\ge1$ such that $I^nm=0$.

 The full subcategory of $I$\+torsion $R$\+modules $R\modl_{I\tors}$
is closed under direct sums, submodules, quotients, and extensions
(in other words, it is a hereditary torsion class) in the ambient
abelian category $R\modl$.
 In particular, the category $R\modl_{I\tors}$ is abelian and
the inclusion functor $R\modl_{I\tors}\rarrow R\modl$ is exact.
 The inclusion functor has a right adjoint functor $\Gamma_I\:
R\modl\rarrow R\modl_{I\tors}$ assigning to every $R$\+module $M$
its maximal $I$\+torsion submodule $\Gamma_I(M)\subset M$.

 Assume that the ideal $I\subset R$ is finitely generated.
 Then it makes sense to consider the $I$\+adic completion functor
$P\longmapsto\Lambda_I(P)=\varprojlim_{n\ge1}P/I^nP$ on the category of
$R$\+modules (cf.~\cite[Theorem~3.6]{Yek}).
 An $R$\+module $P$ is said to be \emph{$I$\+adically separated}
if the natural map $\lambda_{I,P}\:P\rarrow\Lambda_I(P)$ is
injective; and $P$ is said to be \emph{$I$\+adically complete} if
the map~$\lambda_{I,P}$ is surjective.

 The full subcategory of $I$\+adically separated and complete
$R$\+modules $R\modl_{I\secmp}\allowbreak\subset R\modl$ is usually
\emph{not} abelian~\cite[Example~2.7(1)]{Pcta}, and generally not
well-behaved from the homological algebra perspective.
 This is related to the fact that the functor~$\Lambda_I$, constructed
as a composition of a right exact and a left exact functor, is not
even exact in the middle~\cite[Example~3.20]{Yek}.
 Well behaved (abelian) replacements of the category $R\modl_{I\secmp}$
are certain slightly larger full subcategories in $R\modl$, i.~e.,
intermediate categories between $R\modl_{I\secmp}$ and $R\modl$.

 Two reasonably well-behaved replacements of the category of
$I$\+adically separated and complete $R$\+modules $R\modl_{I\secmp}$
are known.
 The more popular one is called the category of \emph{derived
$I$\+adically complete $R$\+modules in the sequential sense} in
the language of~\cite{Yek2} and the category of \emph{$I$\+contramodule
$R$\+modules} in the language of~\cite{Pmgm,Pcta,PSl1,Pdc}.
 The other one is called the category of \emph{derived $I$\+adically
complete $R$\+modules in the idealistic sense} in the language
of~\cite{Yek2} and the category of \emph{quotseparated
$I$\+contramodule $R$\+modules} in the language of~\cite{PSl1,Pdc}.

 An $R$\+module $P$ is said to be
an \emph{$I$\+contramodule}~\cite[Section~2]{Pmgm},
\cite[Sections~2, 5, 7, and~9]{Pcta}, \cite[Sections~0.2 and~1]{Pdc}
(or \emph{derived $I$\+complete} in the language
of~\cite[Section~3.4]{BhSch}, \cite[Section Tag~091N]{SP}) if,
for every element $s\in I$, one has $\Hom_R(R[s^{-1}],P)=0=
\Ext^1_R(R[s^{-1}],P)$.
 It suffices to check this condition for any given set of
generators $s_1$,~\dots, $s_m$ of the ideal~$I$
\,\cite[Theorem~5.1]{Pcta}.
 Notice that the projective dimension of the $R$\+module $R[s^{-1}]$
never exceeds~$1$ \,\cite[proof of Lemma~2.1]{Pcta}, so one need not
impose any higher Ext vanishing conditions.

 Every $I$\+contramodule $R$\+module is $I$\+adically
complete~\cite[Theorem~5.6]{Pcta}.
 But it need not be $I$\+adically separated, as a now-classical
counterexample~\cite[Example~2.5]{Sim}, \cite[Example~3.20]{Yek},
\cite[Section~A.1.1]{Psemi}, \cite[Section~1.5]{Prev},
\cite[Example~2.7(1)]{Pcta} shows.
 So an $I$\+adically separated and complete $R$\+module is the same
thing as an $I$\+adically separated $I$\+contramodule $R$\+module.

 The full subcategory of $I$\+contramodule $R$\+modules
$R\modl_{I\ctra}\subset R\modl$ is closed under direct products,
kernels, cokernels, and extensions in the ambient abelian
category $R\modl$.
 Hence the category $R\modl_{I\ctra}$ is abelian and the inclusion
functor $R\modl_{I\ctra}\rarrow R\modl$ is exact.
 This result can be obtained as a special case of the much more
general~\cite[Proposition~1.1]{GL} or~\cite[Theorem~1.2(a)]{Pcta}.

 An $I$\+contramodule $R$\+module is said to be
\emph{quotseparated}~\cite[Section~5.5]{PSl1},
\cite[Sections~0.1 and~1]{Pdc} if it is a quotient $R$\+module
of an $I$\+adically separated $I$\+contramodule.
 The full subcategory of quotseparated $I$\+contramodule $R$\+modules
$R\modl_{I\ctra}^\qs\subset R\modl_{I\ctra}\subset R\modl$ is closed
under direct products, kernels, and cokernels in the ambient abelian
categories $R\modl_{I\ctra}$ and $R\modl$.
 In particular, the category $R\modl_{I\ctra}^\qs$ is abelian, and
both the inclusion functors $R\modl_{I\ctra}^\qs\rarrow
R\modl_{I\ctra}$ and $R\modl_{I\ctra}^\qs\rarrow R\modl$ are exact.
 Every $I$\+contramodule $R$\+module is an extension of two
quotseparated $I$\+contramodule
$R$\+modules~\cite[Proposition~1.6]{Pdc}.

 The inclusion functor $R\modl_{I\ctra}\rarrow R\modl$ has a left
adjoint functor $\Delta_I\:R\modl\rarrow R\modl_{I\ctra}$
\,\cite[Proposition~2.1]{Pmgm}, \cite[Theorem~7.2]{Pcta}.
 The inclusion functor $R\modl_{I\ctra}^\qs\rarrow R\modl$ also has
a left adjoint functor, which can be computed as the $0$\+th left
derived functor $\boL_0\Lambda_I$ of the (non-right-exact) functor
$\Lambda_I\:R\modl\rarrow R\modl_{I\secmp}\subset R\modl_{I\ctra}^\qs$
\,\cite[Proposition~1.3]{Pdc}.
{\hbadness=1250\par}

 When the ideal $I\subset R$ is \emph{weakly proregular}, every
$I$\+contramodule $R$\+module is
quotseparated~\cite[Corollary~3.7 and Remark~3.8]{Pdc}.
 In particular, this holds for any ideal $I$ in a Noetherian
commutative ring~$R$ (see~\cite[Section~3]{Yek2} for a historical
discussion).
 So one has $\Delta_I=\boL_0\Lambda_I$ in this case. 

 One can easily see that any $I$\+torsion $R$\+module $M$ is
a module over the ring $\varprojlim_{n\ge1}R/I^n=\Lambda_I(R)$.
 Similarly, any $I$\+contramodule $R$\+module is a module
over the ring~$\Delta_I(R)$ \,\cite[last paragraph of
Example~5.2(3)]{Pper}.
 Any quotseparated $I$\+contramodule $R$\+module is a module
over the ring~$\Lambda_I(R)$ \,\cite[Proposition~1.5 and
the subsequent paragraph]{Pdc}.

\subsection{Discussion}
 Let $\X$ be a Noetherian formal scheme~\cite[Section~II.9]{HartAG}.
 How does one globalize the definitions of the categories
$R\modl_{I\tors}$ and $R\modl_{I\ctra}=R\modl_{I\ctra}^\qs$ from
Section~\ref{torsion-completion-basics-subsecn} to a formal scheme
$\X$ replacing a Noetherian commutative ring $R$ with an ideal~$I$\,?

 To simplify the discussion, let us assume that $\X$ is covered by
two affine Noetherian formal schemes $\U_1$ and $\U_2$
(with an affine intersection $\V=\U_1\cap\U_2$).
 As a further simplification, assume that $\V$ is a principal open
affine both in $\U_1$ and in~$\U_2$.

 This means that we are given two Noetherian commutative rings
$R_1$ and $R_2$, two ideals $I_\alpha\subset R_\alpha$ and two elements
$r_\alpha\in R_\alpha$, \ $\alpha=1$,~$2$.
 We are also given an isomorphism of topological commutative rings
\begin{equation} \label{intersection-completion-isomorphism}
 \varprojlim\nolimits_{n\ge1}(R_1[r_1^{-1}]/I_1^n[r_1^{-1}])\simeq
 \varprojlim\nolimits_{n\ge1}(R_2[r_2^{-1}]/I_2^n[r_2^{-1}])
\end{equation}
with the adic topologies.
 The affine formal schemes $\U_\alpha$ are the formal spectra of
the topological rings $\varprojlim_{n\ge1}R_\alpha/I_\alpha^n$ with
their adic topologies, $\U_\alpha=
\Spf\varprojlim_{n\ge1}R_\alpha/I_\alpha^n$,
and the intersection $\V=\U_1\cap\U_2$ is the formal spectrum of
the topological ring~\eqref{intersection-completion-isomorphism}.
 The whole formal scheme $\X$ is the union $\U_1\cup\U_2$.

 What should one mean by a torsion module over~$\X$, and what should
one mean by a contramodule over~$\X$\,?

 Let us start with the torsion modules.
 Given an $I_\alpha$\+torsion $R_\alpha$\+module $M_\alpha$,
the localization $M_\alpha[r_\alpha^{-1}]=
R_\alpha[r_\alpha^{-1}]\ot_{R_\alpha}M_\alpha$ is
an $I_\alpha[r_\alpha^{-1}]$\+torsion
$R_\alpha[r_\alpha^{-1}]$\+module.
 So it seems to be quite natural to define a \emph{quasi-coherent
torsion sheaf\/ $\M$ over\/~$\X$} as a pair $\M=(M_1,M_2)$ of
$I_\alpha$\+torsion $R_\alpha$\+modules $M_\alpha$ together with
an isomorphism
$$
 M_1[r_1^{-1}]\simeq M_2[r_2^{-1}]
$$
of discrete/torsion modules over the topological
ring~\eqref{intersection-completion-isomorphism}.
 This definition of a quasi-coherent torsion sheaf generalizes to
many ind-schemes~\cite[Section~7.11.4]{BD2}, \cite[Chapter~2]{Psemten};
the whole theory developed in the book~\cite{Psemten} is based on
the ind-scheme version of this definition.

 Passing to the contramodules, we start with the observation that
the functor $R_\alpha[r_\alpha^{-1}]\ot_{R_\alpha}{-}\,\:R\modl
\rarrow R[r_\alpha^{-1}]\modl$ does \emph{not} take
$I_\alpha$\+contramodules to $I_\alpha[r_\alpha^{-1}]$\+contramodules.

\begin{ex} \label{noncomplete-localization-example}
 Let $k$~be a field and $R=k[x,y]$ be the ring of polynomials in two
variables over~$k$.
 Put $I=(x)\subset R$ and $r=y\in R$.
 Then $P=\varprojlim_{n\ge1}R/x^nR$ is an $I$\+contramodule $R$\+module
(in fact, it is the \emph{free $I$\+contramodule $R$\+module with one
generator}).
 Still, the $R$\+module $P[y^{-1}]$ is \emph{not} an $I$\+contramodule
$R$\+module, and the $R[y^{-1}]$\+module $P[y^{-1}]$ is \emph{not}
an $I[y^{-1}]$\+contramodule $R[y^{-1}]$\+module.

 The point is that $P[y^{-1}]$ is $x$\+adically separated, but it is
\emph{not} $x$\+adically complete.
 In fact, one has
\begin{equation} \label{completion-localization-noncommutativity}
 P[y^{-1}]=\bigl(\varprojlim\nolimits_{n\ge1}k[x,y]/x^nk[x,y]\bigr) 
 [y^{-1}]\subsetneq
 \varprojlim\nolimits_{n\ge1}k[x,y,y^{-1}]/x^nk[x,y,y^{-1}].
\end{equation}
 For example, the expression $\sum_{n=0}^\infty x^n/y^n$ is
element of the right-hand side, but not of the left-hand side
of~\eqref{completion-localization-noncommutativity}.
 The right-hand side of~\eqref{completion-localization-noncommutativity}
is the $x$\+adic completion of the left-hand side.

 Since all contramodules are adically complete (though they need not
be adically separated)~\cite[Theorem~5.6]{Pcta}, it follows that
$P[y^{-1}]$ \emph{cannot} be either an $I$\+contramodule $R$\+module
or an $I[y^{-1}]$\+contramodule $R[y^{-1}]$\+module.
 (See~\cite{Yek1} for a weaker result which is also sufficient for
the purposes of this example.)
\end{ex}

 What should be meant by a \emph{quasi-coherent sheaf of contramodules
over\/~$\X$}\,?
 Suppose given an $I_1$\+contramodule $R_1$\+module $P_1$ and
$I_2$\+contramodule $R_2$\+module~$P_2$.
 So one can say that $P_1$ is a contramodule over $\U_1$ and
$P_2$ is a contramodule over $\U_2$.
 What could it mean that the localizations of $P_1$ and $P_2$ agree
over $\V=\U_1\cap\U_2$\,?

 Asking for an isomorphism $P_1[r_1^{-1}]\simeq P_2[r_2^{-1}]$ does
not seem to make sense.
 Isomorphic as objects \emph{of what category}?
 Example~\ref{noncomplete-localization-example} clearly shows that
$P_\alpha[r_\alpha^{-1}]$ need not be even modules over
the ring~\eqref{intersection-completion-isomorphism}.

 One could ask for an isomorphism
$$
 \Delta_{I_1[r_1^{-1}]}\bigl(P_1[r_1^{-1}]\bigr)\simeq
 \Delta_{I_2[r_2^{-1}]}\bigl(P_2[r_2^{-1}]\bigr)
$$
of contramodules over the topological
ring~\eqref{intersection-completion-isomorphism}.
 This condition is, at least, well-defined; but is it well-behaved?

 If one wants to prove that the resulting category of ``quasi-coherent
sheaves of contramodules on~$\X$\,'' is abelian, it would be convenient
to know that the gluing functors are exact.
 The functor $\Delta_I$ is certainly \emph{not} exact in most cases
(but only right exact, being a left adjoint).
 Is the functor
\begin{equation} \label{gluing-qcoh-of-contra-functor}
 P\longmapsto\Delta_{I[r^{-1}]}\bigl(P[r^{-1}]\bigr)\:
 R\modl_{I\ctra}\lrarrow R[r^{-1}]\modl_{I[r^{-1}]\ctra}
\end{equation}
exact for every Noetherian commutative ring $R$ with an ideal
$I\subset R$ and an element $r\in R$\,?
 It would be interesting to find out.

 What can one say about
the functor~\eqref{gluing-qcoh-of-contra-functor}\,?
 It is the functor of contraextension of scalars
$f^\sharp\:\fR\contra\rarrow\fS\contra$ \,\cite[Section~2.9]{Pproperf}
with respect to the natural morphism of topological rings
$f\:\fR\rarrow\fS$, where $\fR=\Delta_I(R)=\Lambda_I(R)$ and
$\fS=\Lambda_{I[r^{-1}]}(R[r^{-1}])$.
 But this does not yet seem to help much.

 If the functor~\eqref{gluing-qcoh-of-contra-functor} is \emph{not}
exact, then it is not even clear how to define a reasonable exact
category structure on the category of ``quasi-coherent sheaves
of contramodules over~$\X$\,''.
 Perhaps one should better restrict the gluing construction to
the full subcategories of $I_\alpha$\+contramodule $R_\alpha$\+modules
adjusted to~\eqref{gluing-qcoh-of-contra-functor}.

 It is well-known and easy to show that the functor $\Lambda_I$
takes short exact sequences of flat modules to short exact
sequences~\cite[Lemma~3.5]{PSY}.
 Consequently, the natural morphism $\boL_0\Lambda_I(F)\rarrow
\Lambda_I(F)$ is an isomorphism for all flat $R$\+modules~$F$
\,\cite[Proposition~3.6]{PSY}; so the functor
$\Delta_I=\boL_0\Lambda_I$ also preserves exactness of short
exact sequences of flat modules.

 This allows one to define an exact category of \emph{quasi-coherent
sheaves of flat contramodules on\/~$\X$} by restricting the gluing
construction above to $R_\alpha$\+flat $I_\alpha$\+contramodule
$R_\alpha$\+modules.
 This definition generalizes naturally to
ind-schemes~\cite[Section~7.11.3]{BD2}; the resulting category is
called the category of \emph{flat pro-quasi-coherent pro-sheaves}
in~\cite[Section~3]{Psemten}.
 But flat modules are only a tiny subcategory in the category of
all modules. {\hbadness=1450\par}

 From our perspective, a more promising approach is to glue
$I_\alpha$\+contramodule $R_\alpha$\+modules using
the \emph{colocalization} rather than the localization functors.
 One starts with the following helpful observation.

\begin{lem}
 Let $R$ be a commutative ring, $I\subset R$ be a (finitely generated)
ideal, and $r\in R$ be an element.
 Then, for any $I$\+contramodule $R$\+module $P$, the colocalization
$\Hom_R(R[r^{-1}],P)$ is an $I[r^{-1}]$\+contramodule
$R[r^{-1}]$\+module.
\end{lem}

\begin{proof}
 An $R[r^{-1}]$\+module is an $I[r^{-1}]$\+contramodule if and only if
its underlying $R$\+module is
an $I$\+contramodule~\cite[Theorem~5.1]{Pcta}; so it suffices to
show that $\Hom_R(R[r^{-1}],P)$ is an $I$\+contramodule $R$\+module.
 This is a particular case of~\cite[Lemma~6.1(b)]{Pcta}.
 Simply put, for any $R$\+module $M$, the $R$\+module $\Hom_R(M,P)$
can be obtained from $P$ using the operations of kernel and
infinite product; and the full subcategory $R\modl_{I\ctra}$ is
closed under kernels and infinite products in $R\modl$.
\end{proof}

 So one can define a \emph{contraherent cosheaf of
contramodules over\/~$\X$} as a pair $\P=(P_1,P_2)$ of
$I_\alpha$\+contramodule $R_\alpha$\+modules $P_\alpha$ together
with an isomorphism
$$
 \Hom_{R_1}(R_1[r_1^{-1}],P_1)\simeq\Hom_{R_2}(R_2[r_2^{-1}],P_2)
$$
of contramodules over
the ring~\eqref{intersection-completion-isomorphism} with
its defining ideal of the adic topology
(cf.\ the last paragraph of
Section~\ref{torsion-completion-basics-subsecn}).
 To make this definition well-behaved, one needs to assume adjustedness
to the colocalization, i.~e., impose the \emph{contraadjustedness}
condition~\cite[Section~1.1]{Pcosh}, \cite[Section~2]{Pcta}
$$
 \Ext^1_{R_1}(R_1[r_1^{-1}],P_1)=0=\Ext^1_{R_2}(R_2[r_2^{-1}],P_2).
$$

 The question can and should be asked as to why there are enough
$I_\alpha$\+contramodule $R_\alpha$\+modules satisfying
the contraadjustedness condition.
 Similar questions were addressed in~\cite[Sections~C.2\+-C.3
and~D.4\+-D.5]{Pcosh} and~\cite[Section~7]{PR};
see also~\cite[Section~5]{PSl1}.
 As the output of the definition above, one obtains
the \emph{exact category of contraherent cosheaves of contramodules
over\/~$\X$}.

 The reader should be \emph{warned} that the definition of
a ``contraherent cosheaf of contramodules on~$\X$\,'' as spelled out
above \emph{depends} on the choice of an affine open covering
$\X=\U_1\cup\U_2$.
 Actually, it depends on it for two reasons.

 Firstly, the (very weak) contraadjustedness condition as formulated
above depends on the covering.
 This can be resolved by making the contraadjustedness condition more
restrictive, or more specifically, assuming contraadjustedness with
respect to all elements $r\in R_\alpha$ rather than only $r=r_\alpha$.

 Secondly, and more unpleasantly, the contraherence property is
not local~\cite[Example~3.2.1]{Pcosh}.
 See the discussion of \emph{locally contraherent cosheaves}
in~\cite[Sections~3.1\+-3.2]{Pcosh} and
Remark~\ref{nonlocality-remark} below.
 So one needs to deal with the whole family of \emph{exact categories
of locally contraherent cosheaves of contramodules on\/~$\X$} with
respect to varying open coverings, establishing the fact that these
exact categories ``do not differ too much from each
other''~\cite[Lemma~4.6.1(b)]{Pcosh} and relying on it.

\Section{Why Contraherent Cosheaves? Coderived and Contraderived
Categories} \label{coderived-and-contraderived-secn}

\subsection{Basics} \label{co-contra-derived-basics-subsecn}
 Let $R$ be a ring.
 For the purposes of the discussion in
the next Section~\ref{co-contra-derived-discussion-subsecn},
the ring $R$ can be assumed to be commutative for simplicity,
though this assumption is not needed in the present section.

 There are two approaches to defining the \emph{coderived} and
the \emph{contraderived category} of $R$\+modules in the current
literature.
 The coderived and contraderived categories \emph{in the sense of
Positselski}~\cite{Psemi,Pkoszul,EP,Pweak,Pcosh,Pfp,Pps,PS2,Prel,Pedg}
have more elementary flavor, while the coderived and contraderived
categories \emph{in the sense of
Becker}~\cite{Jor,Kra,IK,Neem,Bec,Sto,PS4,PS5,PS7} have more
set-theoretic flavour.

 It is still an open problem whether the two approaches \emph{ever}
do not agree with each other within their common domain of
applicability~\cite[Examples~2.5(3) and~2.6(3)]{Pps}.
 We refer to~\cite[Remark~9.2]{PS4} and~\cite[Section~7]{Pksurv} for
a historical and philosophical discussion.
 \emph{Derived categories of the second kind} is a generic name for
the coderived and contraderived categories.

 Let $\Hot(R\modl)$ denote the homotopy category of unbounded
complexes of $R$\+modules.
 Any short exact sequence of complexes of $R$\+modules can be viewed
as a bicomplex with three rows, so a total complex (totalization)
can be attached to it.

 A complex of $R$\+modules is said to be \emph{coacyclic in the sense
of Positselski} if it belongs to the minimal full triangulated
subcategory of $\Hot(R\modl)$ containing the totalizations of short
exact sequences of $R$\+modules and \emph{closed under infinite direct
sums}.
 The full triangulated subcategory of Positselski-coacyclic complexes
is denoted by $\Ac^\pco(R\modl)\subset\Hot(R\modl)$, and the related
triangulated Verdier quotient category
$$
 \sD^\pco(R\modl)=\Hot(R\modl)/\Ac^\pco(R\modl)
$$
is called the \emph{coderived category of $R$\+modules in the sense
of Positselski}.

 A complex of $R$\+modules is said to be \emph{contraacyclic in
the sense of Positselski} if it belongs to the minimal full triangulated
subcategory of $\Hot(R\modl)$ containing the totalizations of short
exact sequences of $R$\+modules and \emph{closed under infinite direct
products}.
 The full triangulated subcategory of Positselski-contraacyclic
complexes is denoted by $\Ac^\pctr(R\modl)\subset\Hot(R\modl)$, and
the related triangulated Verdier quotient category
$$
 \sD^\pctr(R\modl)=\Hot(R\modl)/\Ac^\pctr(R\modl)
$$
is called the \emph{contraderived category of $R$\+modules in the sense
of Positselski}.

 The \emph{coderived category of $R$\+modules in the sense of Becker}
can be simply defined as the homotopy category of unbounded complexes
of injective $R$\+modules
\begin{equation} \label{becker-coderived-naive}
 \sD^\bco(R\modl)=\Hot(R\modl^\inj).
\end{equation}

 A more sophisticated approach is to define the \emph{coacyclic
complexes of $R$\+modules in the sense of Becker} as such complexes
$A^\bu\in\Hot(R\modl)$ that $\Hom_{\Hot(R\modl)}(A^\bu,J^\bu)
\allowbreak=0$ for all complexes of injective $R$\+modules
$J^\bu\in\Hot(R\modl^\inj)$.
 The full triangulated subcategory of Becker-coacyclic complexes
is denoted by $\Ac^\bco(R\modl)\subset\Hot(R\modl)$.
 Then the coderived category in the sense of Becker can be defined as
the related Verdier quotient category
\begin{equation} \label{becker-coderived-quotient}
 \sD^\bco(R\modl)=\Hot(R\modl)/\Ac^\bco(R\modl).
\end{equation}

 It is a known theorem that
the definitions~\eqref{becker-coderived-naive}
and~\eqref{becker-coderived-quotient} agree.
 In other words, the composition of functors
$$
 \Hot(R\modl^\inj)\lrarrow\Hot(R\modl)\lrarrow
 \Hot(R\modl)/\Ac^\bco(R\modl)
$$
is a triangulated equivalence~\cite[Proposition~1.3.6(2)]{Bec},
\cite[Theorem~2.13]{Neem2}, \cite[Corollary~5.13]{Kra2},
\cite[Corollary~9.5]{PS4}, \cite[Corollary~7.9]{PS5}.

 The \emph{contraderived category of $R$\+modules in the sense of
Becker} can be simply defined as the homotopy category of unbounded
complexes of projective $R$\+modules
\begin{equation} \label{becker-contraderived-naive}
 \sD^\bctr(R\modl)=\Hot(R\modl_\proj).
\end{equation}

 A more sophisticated approach is to define the \emph{contraacyclic
complexes of $R$\+modules in the sense of Becker} as such complexes
$B^\bu\in\Hot(R\modl)$ that $\Hom_{\Hot(R\modl)}(P^\bu,B^\bu)=0$
for all complexes of projective $R$\+modules
$P^\bu\in\Hot(R\modl_\proj)$.
 The full triangulated subcategory of Becker-contraacyclic complexes
is denoted by $\Ac^\bctr(R\modl)\subset\Hot(R\modl)$.
 Then the contraderived category in the sense of Becker can be defined
as the related Verdier quotient category {\hbadness=1050
\begin{equation} \label{becker-contraderived-quotient}
 \sD^\bctr(R\modl)=\Hot(R\modl)/\Ac^\bctr(R\modl).
\end{equation}

 It is} a known theorem that
the definitions~\eqref{becker-contraderived-naive}
and~\eqref{becker-contraderived-quotient} agree.
 In other words, the composition of functors
$$
 \Hot(R\modl_\proj)\lrarrow\Hot(R\modl)\lrarrow
 \Hot(R\modl)/\Ac^\bctr(R\modl)
$$
is a triangulated equivalence~\cite[Corollary~5.10]{Neem},
\cite[Proposition~1.3.6(1)]{Bec}, \cite[Corollary~7.4]{PS4},
\cite[Corollary~6.13]{PS5}.

 It is clear from~\cite[proof of Theorem~3.5(a)]{Pkoszul}
or~\cite[Lemma~9.1]{PS4} (see also~\cite[Lemma~7.8]{PS5}) that
any coacyclic complex in the sense of Positselski is also
coacyclic in the sense of Becker.
 Similarly, \cite[proof of Theorem~3.5(b)]{Pkoszul}
or~\cite[Lemma~7.1]{PS4} (see also~\cite[Lemma~6.12]{PS5}) tell
that any contraacyclic complex in the sense of Positselski is also
contraacyclic in the sense of Becker.
 The converse inclusions are open problems.

 The Positselski and Becker approaches are known to agree under
certain (somewhat restrictive) assumptions on a ring~$R$.
 See~\cite[Sections~3.6\+-3.8]{Pkoszul}, \cite[Proposition~2.3
and Theorem~2.4]{Pfp}, and~\cite[Section~4]{Pctrl}.
 The two approaches also agree for co/contramodules over a coalgebra
over a field~\cite[Section~4.4]{Pkoszul}.
 A discussion can be found in~\cite[Section~7.9]{Pksurv}.

 All (Becker or Positselski) coacyclic complexes of modules are
acyclic.
 All (Becker or Positselski) contraacyclic complexes of modules are
acyclic.
 So both the coderived and the contraderived category are larger
than the unbounded derived category (generally speaking), each in its
own way.
 Some elements of further philosophical discussion can be found
in~\cite[Section~0.2.7]{Psemi}.

 Both the coderived and the contraderived categories play a crucial
role in the contexts of nonhomogeneous Koszul
duality~\cite{Pkoszul,Prel,Pksurv} and semi-infinite homological
algebra~\cite{Psemi}.
 They are also of key importance for semi-infinite algebraic
geometry~\cite{Pfp,Psemten} and for the theory of curved DG\+modules,
such as matrix factorizations~\cite{PP2,Or,EP,Bec,BDFIK}.

\subsection{Discussion} \label{co-contra-derived-discussion-subsecn}
 How does one globalize the definitions from
Section~\ref{co-contra-derived-basics-subsecn} to a scheme $X$
replacing a commutative ring~$R$\,?
 What happens if one tries to plug the abelian category of
quasi-coherent sheaves $X\qcoh$ into these definitions instead of
the module category $R\modl$\,?

 The definition of the coderived category in the sense of Positselski
makes sense for any abelian (or even exact) category \emph{with exact
functors of infinite direct sum} in place of the module category
$R\modl$.
 Dually, the definition of the contraderived category in the sense
of Positselski makes sense for any abelian (or exact) category
\emph{with exact functors of infinite direct
product}~\cite[Sections~2.1 and~4.1]{Psemi}, \cite[Section~A.1]{Pcosh},
\cite[Appendix~A]{Pmgm}.
 See Section~\ref{exact-categories-fin-homol-dim-subsecn} below
for the details.

\begin{rem} \label{nonexact-co-products-discouraged}
 If the coproduct functors are not exact in an abelian category $\sA$,
then attempting to define the Positselski coderived category
$\sD^\pco(\sA)$ runs into the problem that the coproducts of short
exact sequences need not be exact, so some coacyclic complexes would
not be acyclic.
 Hence the functors of cohomology of complexes would not be defined
on the coderived category.

 Dually, if the product functors are not exact in an abelian category
$\sB$, then attempting to define the Positselski contraderived
category $\sD^\pctr(\sB)$ runs into the problem that some contraacyclic
complexes would not be acyclic.
 Therefore, the functors of cohomology of complexes would not be
defined on the contraderived category.
 To the present author, this appears as undesirable behavior.
\end{rem}

 Both the definitions of the coderived category in the sense of
Becker, \eqref{becker-coderived-naive}
and~\eqref{becker-coderived-quotient}, make sense for an abelian
or exact category $\sA$ plugged in instead of $R\modl$, provided
that $\sA$ has ``not too few'' injective objects in some vague sense.
 One can take this to mean having enough injective objects in
the sense of the usual formal definition.
 If one wants to prove the theorem that the definitions
\eqref{becker-coderived-naive} and~\eqref{becker-coderived-quotient}
agree, then it suffices to assume $\sA$ to be a Grothendieck abelian
category~\cite[Theorem~2.13]{Neem2}, \cite[Corollary~5.13]{Kra2},
\cite[Corollary~9.5]{PS4}, \cite[Corollary~7.9]{PS5}.

 Dually, both the definitions of the contraderived category in
the sense of Becker, \eqref{becker-contraderived-naive}
and~\eqref{becker-contraderived-quotient}, make sense for an abelian
or exact category $\sB$ plugged in instead of $R\modl$, provided
that $\sB$ has ``not too few'' projective objects in some vague sense.
 One can take this to mean having enough projective objects in
the sense of the usual formal definition.
 If one wants to prove the theorem that the definitions
\eqref{becker-contraderived-naive}
and~\eqref{becker-contraderived-quotient} agree, then it suffices to
assume $\sA$ to be a locally presentable abelian category with enough
projective objects~\cite[Corollary~7.4]{PS4},
\cite[Corollary~6.13]{PS5}.

 For any scheme $X$, the category $X\qcoh$ is a Grothendieck abelian
category.
 So both the constructions of the Positselski coderived category
$\sD^\pco(X\qcoh)$ and the Becker coderived category category
$\sD^\bco(X\qcoh)$ make perfect sense.
 We refer to~\cite[Remark~1.3]{EP} and~\cite[Section~A.2]{Psemten}
for a relevant discussion of \emph{locality of coacyclicity of
complexes} in both the contexts of Positselski and Becker coderived
categories.

 \cite[Appendix~A]{Psemten} is also relevant for a discussion of
a comparison between the Positselski and Becker approaches to
coderived categories of quasi-coherent sheaves.
 For Noetherian schemes, the Positselski and Becker coderived
categories of quasi-coherent sheaves agree~\cite[Lemma~1.7(b)]{EP},
\cite[Proposition~2.1]{Pfp}, \cite[Proposition~A.3]{Psemten}.

 Concerning the contraderived category of quasi-coherent sheaves,
it is clear that both the Positselski and the Becker definitions
do not work very well.
 Becker's definition runs into the problem that there are often
\emph{no nonzero projective objects} in the category $X\qcoh$;
so one would have $\sD^\bctr(X\qcoh)=0$.

 The infinite product functors are also usually \emph{not exact} in
$X\qcoh$, so Remark~\ref{nonexact-co-products-discouraged} applies.
 Someone may still want to try to compute the Positselski
contraderived category of quasi-coherent sheaves $\sD^\pctr(X\qcoh)$
in some examples, if only to find out whether it vanishes as well.

 One can modify the Becker definition of the contraderived category
for quasi-coherent sheaves by considering flat quasi-coherent sheaves
as a replacement of projective ones.
 This approach was taken in the dissertation~\cite{Mur}.
 Instead of the homotopy category of complexes of projective objects,
one has to consider the derived category of the exact category of flat
quasi-coherent sheaves (cf.\
Section~\ref{derived-of-flats-and-projectives-subsecn} below).
 This way one obtains a quasi-coherent version
of~\eqref{becker-contraderived-naive}.
 But what would be a quasi-coherent version
of~\eqref{becker-contraderived-quotient}\,?
 What are Becker-contraacyclic complexes of quasi-coherent sheaves?

 We suggest that the \emph{contraderived category of contraherent
cosheaves} $\sD^\pctr(X\ctrh)$ or $\sD^\bctr(X\ctrh)$ is a more
promising object of study than the contraderived category of
quasi-coherent sheaves.
 In all the three major versions of the exact category of
contraherent cosheaves considered in the preprint~\cite{Pcosh},
\begin{itemize}
\item the exact category of locally contraadjusted contraherent
cosheaves $X\ctrh$,
\item the exact category of locally cotorsion contraherent cosheaves
$X\ctrh^\lct$, and
\item the exact category of locally injective contraherent cosheaves
$X\ctrh^\lin$
\end{itemize}
(see~\cite[Section~2.2]{Pcosh} and
Sections~\ref{locally-contraadjusted-subsecn}\+-%
\ref{locally-cotorsion-subsecn} below), all the functors of infinite
product exist, and they take admissible short exact sequences to
admissible short exact sequences.
 So Positselski's contraderived category of contraherent cosheaves
appears to make perfect sense in all the three cases.

 Furthermore, assuming that the scheme $X$ is quasi-compact and
semi-separated, there are enough projective objects in all
the three exact categories $X\ctrh$, $X\ctrh^\lct$, and $X\ctrh^\lin$
\,\cite[Section~4.4 and paragraph before Corollary~4.2.8]{Pcosh}.

 In addition, on any locally Noetherian scheme $X$ there are
enough projective objects in the category of locally cotorsion
contraherent cosheaves $X\ctrh^\lct$.
 This is shown, and such projective objects are explicitly described,
in~\cite[Theorem~5.1.1]{Pcosh}.
 (This explicit description is based on the explicit description of
flat cotorsion modules over Noetherian commutative rings
in~\cite[Theorem in Section~2]{En2}, similarly to the explicit
description of injective quasi-coherent sheaves on locally Noetherian
schemes~\cite[Proposition~II.7.17]{HartRD} based on the explicit
description of injective modules over Noetherian commutative rings
in~\cite[Theorem~2.5, Proposition~2.7, and Proposition~3.1]{Mat0}.)
 There are also enough projective objects in the category of
locally contraadjusted contraherent cosheaves $X\ctrh$ on any
Noetherian scheme $X$ of finite Krull
dimension~\cite[Corollary~5.2.4]{Pcosh}.

 So the construction of the contraderived category in the sense
of Becker is also meaningfully applicable to the exact categories
of contraherent cosheaves, at least, under certain not-too-strong
assumptions.

 For nontrivial results involving the contraderived categories of
contraherent cosheaves, see the quadrality diagram of triangulated
equivalences in~\cite[Theorem~5.7.1]{Pcosh} (for any semi-separated
Noetherian scheme with a dualizing complex) or the triangulated
equivalence of~\cite[Theorem~5.8.1]{Pcosh} (for any Noetherian scheme
with a dualizing complex).

 One is often interested in quasi-coherent sheaves with an additional
structure, like an action of a quasi-coherent associative algebra.
 More generally, one can consider quasi-coherent sheaves on $X$ with
an action of a quasi-coherent quasi-algebra $\A$ over $X$, as
in~\cite[Section~2.3]{Pedg} (e.~g., the sheaf of rings of differential
operators $\D_X$ on a smooth algebraic variety $X$ over a field~$k$
is a quasi-coherent quasi-algebra over~$X$).

 One can define \emph{contraherent modules} $\fM$ over a quasi-coherent
quasi-algebra $\A$ over $X$ as contraherent cosheaves on $X$ endowed
with a morphism of contraherent cosheaves $\fM\rarrow\Cohom_X(\A,\fM)$
satisfying suitable associativity and unitality axioms.
 See~\cite[Section~2.4]{Pcosh} and Section~\ref{cohom-subsecn} below
for the definition of the $\Cohom$ functor (certain adjustedness
conditions need to be imposed either on $\A$, or on $\fM$, or on both
in order to make the $\Cohom$ well-defined).
 Then it is easy to see that the exact category $\A\ctrh$ of
contraherent $\A$\+modules has exact functors of infinite product.
 One may want to try to prove that it has enough projective objects,
under suitable assumptions on~$X$.

 For example, contraherent $\D_X$\+modules and their contraderived
category are relevant to a possible extension of the $\D$\+$\Omega$
duality theory of~\cite[Appendix~B]{Pkoszul} and~\cite{Prel} to
nonaffine schemes $X$ on the contramodule side.

\Section{Category-Theoretic and Module-Theoretic Preliminaries}
\label{preliminaries-secn}

 The notion of a \emph{cotorsion pair} (also known as
a \emph{cotorsion theory}) plays a key role in the theory of
contraherent cosheaves.
 In this section, we spell out the abstract definition in
category-theoretic context, as well as two main concrete examples:
the \emph{flat cotorsion pair} in the category of modules over
an associative ring, and the \emph{very flat cotorsion pair} in
the category of modules over a commutative ring.

\subsection{Cotorsion pairs in exact categories}
\label{cotorsion-pairs-subsecn}
 Let $\sE$ be an exact category (in the sense of Quillen); see
the survey~\cite{Bueh} for the background.
 A discussion of the Yoneda Ext groups in exact categories can be
found in the paper~\cite[Section~A.7]{Partin}.

 Let $\sF$ and $\sC\subset\sE$ be two classes of objects.
 The notation $\sF^{\perp_1}\subset\sE$ stands for the class of
all objects $X\in\sE$ such that $\Ext^1_\sE(F,X)=0$ for all $F\in\sF$.
 Dually, ${}^{\perp_1}\sC\subset\sE$ is the class of all objects
$Y\in\sE$ such that $\Ext^1_\sE(Y,C)=0$ for all $C\in\sC$.

 Similarly, $\sF^{\perp_{\ge1}}\subset\sE$ is the class of all objects
$X\in\sE$ such that $\Ext^n_\sE(F,X)=0$ for all $F\in\sF$ and $n\ge1$.
 Dually, ${}^{\perp_{\ge1}}\sC\subset\sE$ is the class of all objects
$Y\in\sE$ such that $\Ext^n_\sE(Y,C)=0$ for all $C\in\sC$ and $n\ge1$.

 A pair of classes of objects $\sF$, $\sC\subset\sE$ is called
a \emph{cotorsion pair} if $\sC=\sF^{\perp_1}$ and
$\sF={}^{\perp_1}\sC$.
 Informally, one thinks of a cotorsion pair as a pair of classes of
``somewhat projective objects'' and ``somewhat injective objects''
that fit together or complement each other in a suitable sense
(see the introduction to the paper~\cite{Pctrl} for a discussion).

 For example the pairs of classes (all objects, injective objects)
and (projective objects, all objects) are two trivial examples of
cotorsion pairs in an arbitrary exact category.
 The following definition tells what it means that ``there are
enough projective/injective objects'' in the context of
a cotorsion pair.

 A cotorsion pair $(\sF,\sC)$ in $\sE$ is said to be \emph{complete}
if for every object $E\in\sE$ there exist admissible short exact
sequences
\begin{gather}
 0\lrarrow C'\lrarrow F\lrarrow E\lrarrow0
 \label{spec-precover-sequence} \\
 0\lrarrow E\lrarrow C\lrarrow F'\lrarrow0
 \label{spec-preenvelope-sequence}
\end{gather}
in $\sE$ with $F$, $F'\in\sF$ and $C$, $C'\in\sC$.
 The sequence~\eqref{spec-precover-sequence} is known as
a \emph{special precover sequence}, and
the sequence~\eqref{spec-preenvelope-sequence} is called
a \emph{special preenvelope sequence}.
 Collectively, the sequences~(\ref{spec-precover-sequence}\+-%
\ref{spec-preenvelope-sequence}) are referred to as
\emph{approximation sequences}.

 When there are enough injective and projective objects in
an exact category $\sE$ (e.~g., $\sE=R\modl$ with the abelian
exact structure), the condition on a cotorsion pair in $\sE$ to be
complete is still nontrivial.
 But, at least, for any cotorsion pair $(\sF,\sC)$ in $\sE$, one can
say that all the projective objects of $\sE$ belong to $\sF$, and
all the injective objects of $\sE$ belong to~$\sC$.
 So every object of $\sE$ is an admissible quotient object of
an object from $\sF$ and an admissible subobject of an object
from~$\sC$.
 These weak forms of the completeness condition on a cotorsion pair
are not guaranteed in an abelian/exact category $\sE$ in general;
so they need to be assumed in the next definition.

 A class of objects $\sF$ in an exact category $\sE$ is said to be
\emph{generating} if for every object $E\in\sE$ there exists an object
$F\in\sF$ together with an admissible epimorphism
$F\twoheadrightarrow E$ in~$\sE$.
 Dually, a class of objects $\sC$ in an exact category $\sE$ is said
to be \emph{cogenerating} if for every object $E\in\sE$ there exists
an object $C\in\sC$ together with an admissible monomorphism
$E\rightarrowtail C$ in~$\sE$.

\begin{lem} \label{garcia-rozas}
 Let\/ $(\sF,\sC)$ be a cotorsion pair in an exact category\/~$\sE$.
 Assume that the class\/ $\sF$ is generating and the class\/ $\sC$ is
cogenerating in\/~$\sE$.
 Then the following four conditions are equivalent:
\begin{enumerate}
\item the class\/ $\sF$ is closed under kernels of admissible
epimorphisms in\/~$\sE$;
\item the class\/ $\sC$ is closed under cokernels of admissible
monomorphisms in\/~$\sE$;
\item $\Ext^2_\sE(F,C)=0$ for all $F\in\sF$ and $C\in\sC$;
\item $\Ext^n_\sE(F,C)=0$ for all $F\in\sF$, \ $C\in\sC$, and
$n\ge1$.
\end{enumerate}
\end{lem}

\begin{proof}
 This result, going back to~\cite[Theorem~1.2.10]{GR},
can be found in various formulations in~\cite[Lemma~5.24]{GT},
\cite[Lemma~4.25]{SaoSt}, \cite[Lemma~6.17]{Sto-ICRA},
and~\cite[Lemma~1.4]{PS4}.
 For a generalization, see~\cite[Lemma~6.1]{PS6}.
\end{proof}

 A cotorsion pair $(\sF,\sC)$ satisfying the equivalent conditions of
Lemma~\ref{garcia-rozas} is called \emph{hereditary}.

 For any class of objects $\sS\subset\sE$, the pair of classes
$(\sF,\sC)$ with $\sC=\sS^{\perp_1}$ and $\sF={}^{\perp_1}\sC$ is
a cotorsion pair in~$\sE$.
 The cotorsion pair $(\sF,\sC)=({}^{\perp_1}(\sS^{\perp_1}),\>
\sS^{\perp_1})$ is said to be \emph{generated by} the class of
objects $\sS\subset\sE$.

\subsection{Eklof--Trlifaj theorem} \label{eklof-trlifaj-subsecn}
 In this section, for the sake of simplicity of exposition we restrict
ourselves to the classical case of the module category $\sE=R\modl$
with the abelian exact structure.
 In fact, the main results hold in much greater generality.

 Let $R$ be an associative ring.
 Let $F$ be an $R$\+module and $\alpha$~be an ordinal.
 An \emph{$\alpha$\+indexed filtration} on $F$ is a family of submodules
$(F_\beta\subset F)_{0\le\beta\le\alpha}$ satisfying the following
conditions:
\begin{itemize}
\item $F_0=0$ and $F_\alpha=F$;
\item for all ordinals $0\le\gamma\le\beta\le\alpha$, one has
$F_\gamma\subset F_\beta$;
\item for all limit ordinals $0<\beta\le\alpha$, one has
$F_\beta=\bigcup_{\gamma<\beta}F_\gamma$.
\end{itemize}
 An $R$\+module $F$ endowed with an $\alpha$\+indexed filtration
$(F_\beta\subset F)_{0\le\beta\le\alpha}$ is said to be
\emph{filtered by} the quotient modules $F_{\beta+1}/F_\beta$,
\ $0\le\beta<\alpha$.

 Given a class of objects $\sS\subset R\modl$, an $R$\+module $F$
is said to be \emph{filtered by\/~$\sS$} if there exists
an ordinal~$\alpha$ and an $\alpha$\+indexed filtration
$(F_\beta\subset F)_{0\le\beta\le\alpha}$ on $F$ such that, for
every $0\le\beta<\alpha$, the $R$\+module $F_{\beta+1}/F_\beta$
is isomorphic to an $R$\+module from~$\sS$.
 In another language, the $R$\+module $F$ is said to be
a \emph{transfinitely iterated extension} (\emph{in the sense of
the direct limit}) of $R$\+modules from~$\sS$.

 The class of all $R$\+modules filtered by modules from~$\sS$
is denoted by $\Fil(\sS)\subset R\modl$.
 A class of $R$\+modules $\sF$ is said to be \emph{deconstructible}
if there exists a \emph{set} (rather than a proper class) of
$R$\+modules $\sS$ such that $\sF=\Fil(\sS)$.

 The following observation is known as the \emph{Eklof lemma}.

\begin{lem} \label{eklof-lemma}
 For any class of objects\/ $\sC\subset R\modl$, the class\/
${}^{\perp_1}\sC\subset R\modl$ is closed under transfinitely iterated
extensions (in the sense of the direct limit).
 So one has\/ $\Fil({}^{\perp_1}\sC)={}^{\perp_1}\sC$.
\end{lem}

\begin{proof}
 The classical references are~\cite[Lemma~1]{ET}
or~\cite[Lemma~6.2]{GT}.
 The result is actually much more general and, properly understood,
holds in any abelian category~\cite[Lemma~4.5]{PR},
\cite[Proposition~1.3]{PS4}, and even in any exact
category~\cite[Lemma~6.5]{PS6}.
\end{proof}

 Given a class of objects $\sF\subset R\modl$, let us denote by
$\sF^\oplus\subset R\modl$ the class of all direct summands of
objects from~$\sF$.
 The following important result is called the \emph{Eklof--Trlifaj
theorem}~\cite{ET}.

\begin{thm} \label{eklof-trlifaj}
 Let $\sS\subset R\modl$ be a \emph{set} (rather than a proper class)
of modules.  Then \par
\textup{(a)} the cotorsion pair $(\sF,\sC)=
({}^{\perp_1}(\sS^{\perp_1}),\>\sS^{\perp_1})$ generated by\/ $\sS$
in $R\modl$ is complete; \par
\textup{(b)} one has\/ $\sF=\Fil(\{{}_RR\}\cup\sS)^\oplus$, where
${}_RR$ is the free $R$\+module with one generator.
\end{thm}

\begin{proof}
 The classical references are~\cite[Theorems~2 and~10]{ET}
or~\cite[Theorem~6.11 and Corollary~6.14]{GT}.
 The result is actually more general and, properly stated, holds in
any \emph{efficient exact category}~\cite[Corollary~2.15]{SaoSt},
\cite[Theorem~5.16]{Sto-ICRA}, and also in any locally presentable
abelian category~\cite[Corollary~3.6 and Theorem~4.8]{PR},
\cite[Theorems~3.3 and~3.4]{PS4}.
 See~\cite[Section~1.3]{PR} or~\cite[Theorem~1.2]{BHP} for a discussion.
\end{proof}

\subsection{Two cotorsion pairs in module categories}
\label{two-cotorsion-pairs-in-modules-subsecn}
 Let $R$ be an associative ring.
 A left $R$\+module $C$ is said to be
\emph{cotorsion}~\cite[Section~2]{En2} if $\Ext^1_R(F,C)=0$ for all
flat left $R$\+modules~$F$.

 We will denote the full subcategory of flat $R$\+modules
by $R\modl_\fl\subset R\modl$ and the full subcategory of
cotorsion $R$\+modules by $R\modl^\cot\subset R\modl$.

 The following theorem is essentially one of the formulations of
the \emph{Flat Cover Conjecture}.
 Suggested by Enochs in the 1981 paper~\cite[page~196]{En}, it was
proved in the 2001 paper~\cite{BBE}.

\begin{thm} \label{flat-cotorsion-pair}
 For any associative ring $R$, the pair of classes (flat $R$\+modules,
cotorsion $R$\+modules) is a hereditary complete cotorsion pair
in $R\modl$.
\end{thm}

 The pair of classes from Theorem~\ref{flat-cotorsion-pair} is called
the \emph{flat cotorsion pair} in $R\modl$.

\begin{proof}
 The nontrivial part is the complete cotorsion pair assertion;
it is~\cite[Proposition~2]{BBE} or~\cite[Theorem~8.1]{GT}.
 The proof is based on the Eklof--Trlifaj theorem
(Theorem~\ref{eklof-trlifaj}), which is applicable in view of
Enochs' observation that the class of all flat $R$\+modules is
deconstructible~\cite[Lemma~1]{BBE}.
 The flat cotorsion pair is hereditary, since condition~(1) of
Lemma~\ref{garcia-rozas} is obviously satisfied for it.
\end{proof}

 Now let $R$ be a commutative ring.
 An $R$\+module $C$ is called \emph{contraadjusted} (i.~e.,
``adjusted for the contraherent cosheaf
theory'')~\cite[Section~1.1]{Pcosh} if, for every element $s\in R$,
one has $\Ext^1_R(R[s^{-1}],C)=0$.
 An $R$\+module $F$ is called \emph{very flat} if, for every
contraadujsted $R$\+module $C$, one has $\Ext^1_R(F,C)=0$.
 So, by the definition, the pair of classes (very flat $R$\+modules,
contraadjusted $R$\+modules) is the cotorsion pair generated by
the set of modules $\{R[s^{-1}]\mid s\in R\}$ in $R\modl$.

 We will denote the full subcategory of very flat $R$\+modules by
$R\modl_\vfl\subset R\modl$ and the full subcategory of contraadjusted
$R$\+modules by $R\modl^\cta\subset R\modl$.

\begin{thm} \label{very-flat-cotorsion-pair}
 Then $R$ be a commutative ring.  Then \par
\textup{(a)} The pair of classes (very flat $R$\+modules,
contraadjusted $R$\+modules) is a hereditary complete cotorsion
pair in $R\modl$. \par
\textup{(b)} An $R$\+module is very flat if and only if it is a direct
summand of an $R$\+module filtered by the $R$\+modules $R[s^{-1}]$,
\,$s\in R$. \par
\textup{(c)} Any quotient module of a contraadujsted $R$\+module is
contraadjusted. \par
\textup{(d)} Any very flat $R$\+module has projective dimension
at most\/~$1$.
\end{thm}

 The pair of classes from Theorem~\ref{very-flat-cotorsion-pair} is
called the \emph{very flat cotorsion pair} in $R\modl$.

\begin{proof}
 The complete cotorsion pair claim in part~(a)
is~\cite[Theorem~1.1.1]{Pcosh}, part~(b)
is~\cite[Corollary~1.1.4]{Pcosh}, part~(c)
is~\cite[second paragraph of Section~1.1]{Pcosh}, part~(d)
is~\cite[fourth paragraph of Section~1.1]{Pcosh}, and
the hereditary cotorsion pair claim is~\cite[fifth paragraph of
Section~1.1]{Pcosh}.
 Essentially, the complete cotorsion pair claim in part~(a), as well
as part~(b), are particular cases of the Eklof--Trlifaj theorem
(Theorem~\ref{eklof-trlifaj}(a\+-b) above) for
$\sS=\{R[s^{-1}]\mid s\in R\}\subset R\modl$.
 Part~(c) holds for the right-hand part of any cotorsion pair generated
by a class of modules of projective dimension at most~$1$,
as one can easily see.
 The point is that the projective dimension of the $R$\+module
$R[s^{-1}]$ never exceeds~$1$ \,\cite[proof of Lemma~2.1]{Pcta}.
 Part~(c) implies condition~(2) of Lemma~\ref{garcia-rozas}, proving
that the very flat cotorsion pair is hereditary.
 Part~(d) also follows from part~(c); or it can be deduced from
part~(b) and Lemma~\ref{eklof-lemma}.
\end{proof}

 It is obvious from the definition that any cotorsion $R$\+module
is contraadjusted (because the $R$\+modules $R[s^{-1}]$ are flat).
 It is clear from Theorem~\ref{very-flat-cotorsion-pair}(b) that
all projective $R$\+modules are very flat, and all very flat
$R$\+modules are flat.

\begin{ex} \label{open-immersion-very-flat}
 Let $R\rarrow S$ be a homomorphism of commutative rings such that
the induced morphism of the spectra $\Spec S\rarrow\Spec R$ is
an open immersion of affine schemes.
 Then the $R$\+module $S$ is very flat~\cite[Lemma~1.2.4]{Pcosh}.
\end{ex}

 A variety of concrete examples of approximation
sequences~(\ref{spec-precover-sequence}\+-%
\ref{spec-preenvelope-sequence}) for the flat cotorsion pair in
the category of abelian groups $\Ab=\boZ\modl$ can be found
in~\cite[Section~12]{Pcta}, and a description of cotorsion modules
over a commutative Noetherian ring of Krull dimension~$1$ is
obtained in~\cite[Corollary~13.12]{Pcta}.
 A characterization of contraadjusted abelian groups is discussed
in~\cite[end of Section~12]{Pcta}, following~\cite[Section~4]{Sl}
and~\cite[Example~5.2]{ST} (see also~\cite[end of Section~13]{Pcta}).
 Approximation sequences for the very flat cotorsion pair are more
elusive: \emph{no} explicit nontrivial examples for $R=\boZ$ are
known~\cite[Remark~12.15]{Pcta}.

\Section{Main Definition and Main Technical Problem}
\label{main-definition-secn}

\subsection{Quasi-coherent sheaves} \label{quasi-coherent-subsecn}
 Let $X$ be a scheme.
 The definition of a contraherent cosheaf on~$X$ in
the preprint~\cite{Pcosh} is modeled after the following description
of quasi-coherent sheaves on~$X$ \,\cite[Section~2]{EE}.

 A quasi-coherent sheaf $\M$ on $X$ is the same thing as the following
set of data:
\begin{itemize}
\item to every affine open subscheme $U\subset X$, an $\cO(U)$\+module
$\M(U)$ is assigned;
\item to every pair of embedded affine open subschemes $V\subset U
\subset X$, a homomorphism of $\cO(U)$\+modules $\M(U)\rarrow\M(V)$
is assigned
\end{itemize}
satisfying the following conditions:
\begin{itemize}
\item for every pair of embedded affine open subschemes $V\subset U
\subset X$, the homomorphism $\M(U)\rarrow\M(V)$ induces an isomorphism
of $\cO(V)$\+modules
$$
 \cO(V)\ot_{\cO(U)}\M(U)\simeq\M(V);
$$
\item for every triple of embedded affine open subschemes $W\subset V
\subset U\subset X$, the triangular diagram of homomorphisms of
$\cO(U)$\+modules
$$
 \M(U)\lrarrow\M(V)\lrarrow\M(W)
$$
is commutative.
\end{itemize}

 The definition of a contraherent cosheaf on $X$ is dual to
this definition of a quasi-coherent sheaf.
 But there are delicate details and caveats.

\subsection{Terminological remark} \label{terminological-remark-subsecn}
 From today's perspective, some terminological choices in the long
preprint~\cite{Pcosh} appear to be not optimal.
 Apparently, the significance of the colocalization functors
$\Hom_R(S,{-})$ and the related terminology in the context of
the contraherent cosheaf theory, emphasized in the abstract and
the introduction to the present paper, was not properly realized at
the time when the book manuscript~\cite{Pcosh} was written.
{\hbadness=1700\par}

 What appears today as the proper terminological solutions were
suggested in the recent preprint~\cite{Pal}.
 Perhaps it would be better to use the terms ``colocally
contraadjusted'' and ``colocally cotorsion'' for what are called
``locally contraadjusted'' and ``locally cotorsion contraherent
cosheaves'' both in~\cite{Pcosh} and in the present paper.

 The terminology of ``colocally projective'' and ``colocally flat
contraherent cosheaves'' used in~\cite[Sections~4.2\+-4.3]{Pcosh}
is particularly problematic.
 The word ``colocal'' in these terms was not used in the sense
of a reference to the colocalization functors.
 Rather, it meant something like ``opposite or complementary to local''.
 The term ``antilocal'' was suggested for this purpose in~\cite{Pal}.
{\hbadness=1800\par}

 So, later in this paper
(in Section~\ref{antilocal-contraherent-subsecn}), we will use the term
\emph{antilocal contraherent cosheaves} for what were called
``colocally projective contraherent cosheaves'' in~\cite{Pcosh}.
 In the same vein, what were called ``colocally flat contraherent
cosheaves'' in~\cite{Pcosh} should be called \emph{antilocally flat}.

\subsection{Locally contraadjusted contraherent cosheaves}
\label{locally-contraadjusted-subsecn}
 Let $X$ be a scheme.
 A \emph{contraherent cosheaf} $\P$ on $X$ \,\cite[Section~2.2]{Pcosh}
is the following set of data:
\begin{itemize}
\item to every affine open subscheme $U\subset X$,
an $\cO(U)$\+module $\P[U]$ is assigned;
\item to every pair of embedded affine open subschemes $V\subset U
\subset X$, a homomorphism of $\cO(U)$\+modules $\P[V]\rarrow\P[U]$
is assigned
\end{itemize}
satisfying the following conditions:
\begin{enumerate}
\renewcommand{\theenumi}{\roman{enumi}}
\item for every pair of embedded affine open subschemes $V\subset U
\subset X$, the homomorphism $\P[V]\rarrow\P[U]$ induces
an isomorphism of $\cO(V)$\+modules
$$
 \P[V]\simeq\Hom_{\cO(U)}(\cO(V),\P[U]);
$$
\item for every pair of embedded affine open subschemes $V\subset U
\subset X$, one has
$$
 \Ext_{\cO(U)}^1(\cO(V),\P[U])=0;
$$
\item for every triple of embedded affine open subschemes $W\subset V
\subset U\subset X$, the triangular diagram of homomorphisms of
$\cO(U)$\+modules
$$
 \P[W]\lrarrow\P[V]\lrarrow\P[U]
$$
is commutative.
\end{enumerate}

\begin{rem}
 No higher Ext vanishing condition ($\Ext^n=0$ for $n\ge2$) is
mentioned in~(ii), because these conditions hold automatically.
 In fact, for any open immersion of affine schemes $V\rarrow U$,
the $\cO(U)$\+module $\cO(V)$ has projective dimension at most~$1$.
 This assertion can be explained in several ways.

 On the one hand, the $\cO(U)$\+module $\cO(V)$ is very flat by
Example~\ref{open-immersion-very-flat}.
 All very flat modules over a commutative ring $R$ have projective
dimension at most~$1$ by Theorem~\ref{very-flat-cotorsion-pair}(d).

 On the other hand, any countably presented flat module over
an associative ring $R$ has projective dimension at most~$1$
\,\cite[Corollary~2.23]{GT}.
 The assertion that the $\cO(U)$\+module $\cO(V)$ is countably
presented can be also explained in various ways.

 One approach is to observe that the $R$\+module $R[s^{-1}]$ is
flat and countably presented for any element~$s$ in a commutative
ring~$R$.
 Then use the \v Cech coresolution as in~\cite[proof of
Lemma~1.2.4]{Pcosh} and the \emph{flat coherence}
lemma~\cite[Lemma~4.1]{Pflcc}, \cite[Corollary~10.12]{Pacc},
\cite[Corollary~4.7]{Plce}.

 Alternatively, countable presentability of the $\cO(U)$\+module
$\cO(V)$ follows immediately from finite presentability of
the $\cO(U)$\+algebra~$\cO(V)$.
 For the latter claim, we refer to~\cite[Definition~I.6.2.1,
Proposition~I.6.2.6(i), and Proposition~I.6.2.9]{EGA1}.
\end{rem}

\begin{rem} \label{contraadjustedness-condition-remark}
 Condition~(ii) can be equivalently restated by saying that
the $\cO(U)$\+module $\P[U]$ has to be contraadjusted
(in the sense of the definition in
Section~\ref{two-cotorsion-pairs-in-modules-subsecn}).
 Indeed, put $R=\cO(U)$; then, for every element $s\in R$,
the scheme $V=\Spec R[s^{-1}]$ is an affine open subscheme in $U$,
and consequently in~$X$.
 So~(ii) implies the contraadjustedness of~$\P[U]$.
 Conversely, for any affine open subscheme $V\subset U$,
the $\cO(U)$\+module $\cO(V)$ is very flat by
Example~\ref{open-immersion-very-flat}.
 Hence $\Ext^1_{\cO(U)}(\cO(V),C)=0$ for every contraadjusted
$\cO(U)$\+module~$C$.

 For this reason, the contraherent cosheaves on $X$ in the sense
of the definition above are called \emph{locally contraadjusted}.
 Condition~(ii) is called the \emph{contraadjustedness condition},
while condition~(i) is called the \emph{contraherence condition}.
\end{rem}

\begin{rem}
 It needs to be explained what the word ``cosheaf'' means in
the definition of a contraherent cosheaf given above.
 No ``cosheaf axiom'' appears in that definition; there are only
the contraherence and contraadjustedness axioms.

 The same question can (and should) be asked about the definition
or description of quasi-coherent sheaves from
Section~\ref{quasi-coherent-subsecn} (going back
to~\cite[Section~2]{EE}), which contains no sheaf axiom, but only
a quasi-coherence axiom.

 The answer to this question (for contraherent cosheaves) is
explained in~\cite[Sections~2.1\+-2.2]{Pcosh}.
 Essentially, the point is that it suffices to define a (co)sheaf
of $\cO_X$\+modules on any topology base $\bB$ of a ringed space
$(X,\cO_X)$, and it suffices to check the (co)sheaf axiom for coverings
of open subsets belonging to $\bB$ by other open subsets belonging
to~$\bB$ \,\cite[Section~0.3.2]{EGA1}, \cite[Theorem~2.1.2]{Pcosh}.
 In the situation at hand, one takes $\mathbf B$ to be the topology
base consisting of all affine open subschemes of a scheme~$X$.

 Furthermore, restricted to coverings of affine open subschemes by
afine open subschemes, the sheaf/cosheaf axiom follows from
the quasi-coherence/contraherence axiom.
 In the case of the cosheaves, one also needs to use
the contraadjustedness axiom~\cite[Theorem~2.2.1]{Pcosh}.
 So any contraherent cosheaf on a scheme $X$ in the sense of
the definition above can be uniquely extended a cosheaf of
$\cO_X$\+modules defined on all open subsets of $X$ and satisfying
the cosheaf axiom for all coverings of open subsets by other
open subsets.
\end{rem}

\begin{rem} \label{nonlocality-remark}
 So one can start with a cosheaf of $\cO_X$\+modules $\P$ on
a scheme~$X$, and consider the contraherence and contraadjustedness
axioms~(i\+-ii) as conditions imposed on a cosheaf~$\P$.
 This is the approach taken in the exposition in~\cite{Pcosh}.
 Then the question about a local or nonlocal nature of these axioms
arises.

 In this context, it turns out that the notion of quasi-coherent
sheaf is indeed local, but the notion of a contraherent cosheaf
is \emph{not}.
 Specifically, the contraadjustedness axiom~(ii) is still local, but
\emph{not} the contraherence axiom~(i).
 This is explained in~\cite[Sections~3.1\+-3.2]{Pcosh}; see
specifically~\cite[Example~3.2.1]{Pcosh}.

 In other words, given an open covering $X=U\cup V$ of a scheme $X$
by two open subschemes $U$ and $V$, any two quasi-coherent sheaves
on $U$ and $V$ that agree on the intersection $U\cap V$ can be glued
together into a quasi-coherent sheaf on~$X$.
 But two contraherent cosheaves on $U$ and $V$ that agree on $U\cap V$
can be only glued into a \emph{locally contraherent cosheaf} on $X$,
generally speaking.

 Locally contraherent cosheaves on a scheme $X$ are still cosheaves
of $\cO_X$\+modules, which can be viewed as well-defined on and
satisfying the cosheaf axiom for all the open subsets of~$X$.
 The category of $\bW$\+locally contraherent cosheaves $X\lcth_\bW$
depends on a chosen open covering $\bW$ of a scheme~$X$.
 It can be defined in the spirit of the definition of a contraherent
cosheaves by three conditions~(i\+-iii) above; the only difference is
that one considers the base $\bB$ of all affine open subschemes
$U$ in $X$ \emph{subordinate to} the open covering~$\bW$.
 This means that $U$ ranges over all the affine open subschemes of $X$
such that $U\subset W$ for some $W\in\bW$.

 The category of all locally contraherent cosheaves $X\lcth=\bigcup_\bW
X\lcth_\bW$ has a natural exact category structure.
 The full subcategories $X\ctrh$ and $X\lcth_\bW$ are closed under
extensions and kernels of admissible epimorphisms in $X\lcth$; so
they are also exact categories~\cite[Section~3.2]{Pcosh}.
 For a quasi-compact semi-separated scheme $X$, the categories $X\ctrh$
and $X\lcth_\bW$ are resolving subcategories in $X\lcth$.
 The resolution dimensions of all objects of $X\lcth_\bW$ with
respect to the resolving subcategory $X\ctrh$ are bounded by a constant
depending on the open covering~$\bW$ \,\cite[Lemma~4.6.1(b)]{Pcosh}.
\end{rem}

\subsection{Locally cotorsion contraherent cosheaves}
\label{locally-cotorsion-subsecn}
 A contraherent cosheaf $\P$ on a scheme $X$ is said to be
\emph{locally cotorsion} if the $\cO(U)$\+module $\P[U]$ is cotorsion
(in the sense of the definition in
Section~\ref{two-cotorsion-pairs-in-modules-subsecn})
for every affine open subscheme $U\subset X$.
 Being a locally cotorsion contraherent cosheaf is indeed a local
property: it suffices to check it for affine open subschemes
$U\subset X$ belonging to any fixed affine open covering of
a scheme~$X$ \,\cite[Lemma~1.3.6(a) and Section~2.2]{Pcosh}.

 The full subcategory of locally cotorsion contraherent cosheaves
$X\ctrh^\lct$ is closed under extensions and cokernels of
admissible monomorphisms in $X\ctrh$; so $X\ctrh^\lct$ is also
an exact category~\cite[Section~2.2]{Pcosh}.
 For a quasi-compact semi-separated scheme $X$, the full subcategory
$X\ctrh^\lct$ is coresolving (in fact, the right-hand side of
a hereditary complete cotorsion pair) in $X\ctrh$
\,\cite[Corollary~4.3.4(a)]{Pcosh}.
 See Section~\ref{exact-categories-of-contraherent-subsecn} below
for a further discussion.

 A contraherent cosheaf $\fJ$ on a scheme $X$ is said to be
\emph{locally injective} if the $\cO(U)$\+module $\fJ[U]$ is
injective for every affine open subscheme $U\subset X$.
 The formal properties of locally injective contrahent cosheaves
are similar to those of locally cotorsion ones (as per
the previous paragraphs).
 The exact category of locally injective contraherent cosheaves
is denoted by $X\ctrh^\lin\subset X\ctrh^\lct\subset X\ctrh$.

\subsection{The functor $\Cohom$} \label{cohom-subsecn}
 The functor $\Cohom_X$ from a quasi-coherent sheaf to a contraherent
cosheaf on a scheme $X$ is dual-analogous to the functor of
tensor product of quasi-coherent sheaves, $\ot_{\cO_X}\:X\qcoh
\times X\qcoh\rarrow X\qcoh$, in the same way as the Hom functor of
modules over a ring is dual-analogous to the tensor product functor
of modules.
 The notation $\Cohom$ comes from the interpretation of
quasi-coherent sheaves as comodules over a coring $\C$ over
a commutative ring $A$ and quasi-coherent sheaves as (approximately) 
contramodules over the same coring.
 See Section~\ref{comodules-and-contramodules-introd-subsecn};
cf.\ the definition of the Cohom functor for comodules and
contramodules in~\cite[Sections~2.5\+2.6]{Prev}
or~\cite[Sections~0.2.4 and~3.2]{Psemi}.

 Let $X$ be a scheme, $\M$ be a quasi-coherent sheaf on $X$, and
$\P$ be a contraherent cosheaf on~$X$.
 Then the contraherent cosheaf $\Cohom_X(\M,\P)$ is defined by
the rule
$$
 \Cohom_X(\M,\P)[U]=\Hom_{\cO(U)}(\M(U),\P[U])
$$
for all affine open subschemes $U\subset X$.
 One can easily check that the contraherence condition~(i) from
Section~\ref{locally-contraadjusted-subsecn} holds for
$\Cohom_X(\M,\P)$ for any quasi-coherent sheaf $\M$ and
contraherent cosheaf $\P$ on~$X$ \,\cite[Section~2.4]{Pcosh}.
 The problem is that the contraadjustedness condition~(ii) need
not hold in general.

\begin{ex}
 Here is an example of a commutative ring $R$, an $R$\+module $M$,
and a cotorsion $R$\+module $P$ such that the $R$\+module
$\Hom_R(M,P)$ is not even contraadjusted.
 The argument is based on the following observations.
 Let $f\:A\rarrow B$ be a homomorphism of commutative rings and $C$
be a $B$\+module.
 Then
\begin{enumerate}
\renewcommand{\theenumi}{\alph{enumi}}
\item if $C$ is a contraadjusted $B$\+module, then $C$ is also
a contraadjusted $A$\+mod\-ule~\cite[Lemma~1.2.2(a)]{Pcosh};
\item if $C$ is a cotorsion $B$\+module, then $C$ is also
a cotorsion $A$\+module~\cite[Lemma~1.3.4(a)]{Pcosh}.
\end{enumerate}

 Let $R=\boZ[x]$ be the ring of polynomials in one variable~$x$ with
integer coefficients.
 Pick a prime number~$p$, and consider the $\boZ$\+module of $p$\+adic
integers~$\boZ_p$.
 Then $\boZ\subset\boZ_p$ is a $\boZ$\+submodule.
 Consider the quotient $\boZ$\+module $\boZ_p/\boZ$, and pick
an embedding $i\:\boZ_p/\boZ\rarrow J$ of the $\boZ$\+module
$\boZ_p/\boZ$ into an injective $\boZ$\+module (i.~e., a divisible
abelian group)~$J$.
 Denote by $g\:\boZ_p\rarrow J$ the composition $\boZ_p
\twoheadrightarrow\boZ_p/\boZ\overset i\rightarrowtail J$.
 Put $P=\boZ_p\oplus J$, and let the generator $x\in\boZ[x]$ act
in $P$ by the operator
$\left(\begin{smallmatrix}0 & 0 \\ g & 0 \end{smallmatrix}\right)\:
\boZ_p\oplus J\rarrow\boZ_p\oplus J$.
 This defines a $\boZ[x]$\+module structure on~$P$.

 We claim that the $\boZ[x]$\+module $P$ is cotorsion.
 Indeed, there is a short exact sequence of $\boZ[x]$\+modules
$0\rarrow J\rarrow P\rarrow\boZ_p\rarrow0$, where $x$~acts by zero
in $J$ and in~$\boZ_p$.
 In other words, if $f\:\boZ[x]\rarrow\boZ$ is the ring homomorphism
for which $f(x)=0$, then the $\boZ[x]$\+module structures on $J$
and $\boZ_p$ are obtained from their $\boZ$\+module structures by
the restriction of scalars via~$f$.
 Now $J$ is an injective, hence cotorsion $\boZ$\+module; and
$\boZ_p=\Hom_\boZ(\boQ_p/\boZ_p,\>\boQ_p/\boZ_p)$ is
a cotorsion $\boZ$\+module by~\cite[Lemma~2.1]{En2},
\cite[Lemma~1.3.2(b) or~1.3.3(b)]{Pcosh},
\cite[Proposition~B.10.1]{Pweak}, or~\cite[Theorem~9.5
or Example~12.1]{Pcta}.
 According to assertion~(b) above, it follows that $J$ and $\boZ_p$ are
cotorsion $\boZ[x]$\+modules; hence so is their extension~$P$.

 Put $M=\boZ$, and let the generator $x\in\boZ[x]$ act in $M$ by
the zero map.
 This makes $M$ a $\boZ[x]$\+module.
 One easily computes $\Hom_{\boZ[x]}(M,P)\simeq\ker(x\:P\to P)=
\boZ\oplus J\subset\boZ_p\oplus J$.
 Now the restriction of scalars with respect to the ring homomorphism
$f\:\boZ\rarrow\boZ[x]$ takes the $\boZ[x]$\+module $\boZ\oplus J$ to
the $\boZ$\+module $\boZ\oplus J$.
 The $\boZ$\+module $\boZ$ is \emph{not} contraadjusted, as
$\Ext^1_\boZ(\boZ[1/p],\boZ)\simeq\boQ_p/(\boZ[1/p])\ne0$.
 Hence the $\boZ$\+module $\boZ\oplus J$ is not contraadjusted.
 According to assertion~(a) above, it follows that the $\boZ[x]$\+module
$\Hom_{\boZ[x]}(M,P)=\boZ\oplus J$ is not contraadjusted.
\end{ex}

 The contraherent cosheaf $\Cohom_X({-},{-})$ is well-defined in
the following cases~\cite[Section~2.4]{Pcosh}:
\begin{enumerate}
\item if $\F$ is a very flat quasi-coherent sheaf (as defined
in~\cite[Section~1.7]{Pcosh} and
Section~\ref{very-flat-sheaves-subsecn} below) and $\P$ is a locally
contraadjusted (i.~e., arbitrary) contraherent cosheaf, then
$\Cohom_X(\F,\P)$ is a locally contraadjusted contraherent cosheaf;
\item if $\F$ is a flat quasi-coherent sheaf and $\P$ is a locally
cotorsion contraherent cosheaf, then $\Cohom_X(\F,\P)$ is a locally
cotorsion contraherent cosheaf;
\item if $\M$ is an arbitrary quasi-coherent sheaf and $\fJ$ is
a locally injective contraherent cosheaf, then $\Cohom_X(\M,\fJ)$
is a locally cotorsion contraherent cosheaf;
\item if $\F$ is a flat quasi-coherent sheaf and $\fJ$ is
a locally injective contraherent cosheaf, then $\Cohom_X(\F,\fJ)$
is a locally injective contraherent cosheaf.
\end{enumerate}
 So $\Cohom_X\:X\qcoh^\sop\times X\ctrh\dasharrow X\ctrh$ is
a partially defined functor of two arguments.

\subsection{Discussion} \label{main-definition-discussion-subsecn}
 The category of quasi-coherent sheaves $X\qcoh$ on any scheme $X$ is
abelian.
 An inspection of the description of quasi-coherent sheaves in
Section~\ref{quasi-coherent-subsecn} leads one to the conclusion
that this fortunate state of affairs holds due to the fact that,
for any affine scheme $U$ and affine open subscheme $V\subset U$,
the tensor product functor $\cO(V)\ot_{\cO(U)}{-}\,\:\cO(U)\modl
\rarrow\cO(V)\modl$ is exact.
 In other words, it is important that the $\cO(U)$\+module $\cO(V)$
is flat.

 In the definition of a contraherent cosheaf in
Section~\ref{locally-contraadjusted-subsecn}, the localization
functor $\cO(V)\ot_{\cO(U)}{-}$ is replaced by the colocalization
functor $\Hom_{\cO(U)}(\cO(V),{-})\:\cO(U)\modl\allowbreak
\rarrow\cO(V)\modl$.
 The \emph{main technical problem} of the theory of contraherent
cosheaves is that the functor $\Hom_{\cO(U)}(\cO(V),{-})$ is not exact
as a functor between the abelian categories of modules.
 In other words, the $\cO(U)$\+module $\cO(V)$ is usually \emph{not}
projective.

 As a consequence, one needs to impose the contraadjustedness
condition~(ii) in the definition of a contraherent cosheaf
(see Remark~\ref{contraadjustedness-condition-remark}).
 This condition has no analogue in the definition of a quasi-coherent
sheaf.
 Imposing this condition makes the category of contraherent cosheaves
\emph{not abelian, but only exact} (in Quillen's sense).
 Furthermore, there are several versions of the condition, giving
rise to several versions of the contraherent cosheaf category.

 In particular, let $U$ be an affine scheme with the ring of
functions $R=\cO(U)$.
 Then the category of locally contraadjusted contraherent cosheaves
on $U$ (as defined in Section~\ref{locally-contraadjusted-subsecn})
is equivalent to the category of contraadjusted $R$\+modules, while
the categories of locally cotorsion and locally injective contraherent
cosheaves (as defined in Section~\ref{locally-cotorsion-subsecn}) are
equivalent to the categories of cotorsion $R$\+modules and injective
$R$\+modules, respectively:
\begin{align}
 U\ctrh &\simeq R\modl^\cta, \label{contraadjusted-on-affine} \\
 U\ctrh^\lct &\simeq R\modl^\cot, \label{cotorsion-on-affine} \\
 U\ctrh^\lin &\simeq R\modl^\inj. \label{injective-on-affine}
\end{align}
 These are exact, but not abelian categories.

 One may want to try dispensing with the contraadjustedness
condition altogether by simply dropping it from the definition of
a contraherent cosheaf, while keeping conditions~(i) and~(iii) only.
 The present author expects that the resulting additive category
will \emph{not} be abelian (on a nonaffine scheme), and generally not
well-behaved from the homological algebra point of view.
 But it might be worth trying.

 Various specific technical problems of the theory of contraherent
cosheaves usually arise as consequences of the main technical problem.
 One class of such specific technical problems is that various
functors are only partially defined.

 A functor taking values in an abelian category may fail to be exact,
but it is usually everywhere defined.
 For example, if $F\:\sD(\sA)\rarrow\sD(\sB)$ is a triangulated
functor between the derived categories of $\sA$ and $\sB$, then
one can compose $F$ with the inclusion functor $\sA\rarrow\sD(\sA)$
and the degree-zero cohomology functor $H^0\:\sD(\sB)\rarrow\sB$,
producing a (nonexact) functor $\sA\rarrow\sB$.

 In contrast, complexes in an exact category $\sE$ have \emph{no}
cohomology objects in $\sE$, generally speaking.
 Whenever a functor of abelian categories is not exact, a similar
functor of exact categories tends to be only partially defined.
 We have seen an example of this phenomenon in
Section~\ref{cohom-subsecn}, for the functor $\Cohom_X$.

 Another unpleasant specific technical problem of the theory of
contraherent cosheaves, also arising from nonexactness of
the colocalization functors, is the nonlocality of contraherence;
see Remark~\ref{nonlocality-remark}.

\subsection{Two branches of the theory}
 So, technical problems of the theory of contraherent cosheaves over
schemes arise from the fact that, for an affine open subscheme $V$
in an affine scheme $U$, the $\cO(U)$\+module $\cO(V)$ is usually only
flat, but not projective.
 What can be done to mitigate the problems?

 There are two approaches, hence two branches of the theory of
contraherent cosheaves.
 One point of view is that, as far as flat modules go,
the $\cO(U)$\+module $\cO(V)$ is not that bad.
 It belongs to a class of flat modules over commutative rings with
particularly simple homological behavior, viz., the \emph{very flat
modules}; see Section~\ref{two-cotorsion-pairs-in-modules-subsecn}
and specifically Example~\ref{open-immersion-very-flat}.
 Taking this point of view leads one to consider
the \emph{locally contraadjusted contraherent cosheaves},
see Section~\ref{locally-contraadjusted-subsecn}.

 The other point of view is that flat modules generally are
``not that bad'', or not that much different from projective modules.
 Adopting this approach means restricting oneself to
the \emph{locally cotorsion contraherent cosheaves}, as defined
in Section~\ref{locally-cotorsion-subsecn}.

 Each one of the two approaches comes with its own set of
technical problems.
 At the time of writing of the long preprint~\cite{Pcosh},
I~thought that locally cotorsion contraherent cosheaves were more
suited for Noetherian schemes of finite Krull dimension, while on
non-Noetherian schemes one had to consider locally contraadjusted
contraherent cosheaves.
 Ten years have passed, some work has been done, and now it seems
that both the locally contraadjusted and locally cotorsion
contraherent cosheaves behave reasonably well on all schemes;
certainly better than it was originally expected.
 The aim of the next
Sections~\ref{very-flat-secn}\+-\ref{flat-modules-secn} is to
explain why this is the case.

\Section{Very Flat Morphisms are Ubiquitous}
\label{very-flat-secn}

 The discussion of the $\Cohom$ functor in Section~\ref{cohom-subsecn}
(see item~(1)) illustrates the point that, when dealing with locally
contraadjusted contraherent cosheaves, one often needs one's modules
and sheaves to be very flat (and not just flat).
 An even more compelling manifestation of this phenomenon in
the context of inverse images of contraherent cosheaves will be
discussed in Section~\ref{inverse-images-of-contrah-subsecn} below.

 In this section we discuss the problem that very flat modules,
morphisms, and sheaves might be too few, and the solution in the form
of a theorem telling that they are actually quite common and numerous.

\subsection{Change of scalars and tensor product}
\label{change-of-scalars-subsecn}
 Let $f\:R\rarrow S$ be a homomorphism of commutative rings.
 Then, for any flat $R$\+module $F$, the $S$\+module $S\ot_RF$ is flat.
 Similarly, for any very flat $R$\+module $F$, the $S$\+module
$S\ot_RF$ is very flat~\cite[Lemma~1.2.2(b)]{Pcosh}.

 On the other hand, if $G$ is a flat $S$\+module and $f$~makes $S$
a flat $R$\+module, then $G$ is also a flat $R$\+module.
 The analogue of this result for very flat modules is more complicated.

 A ring homomorphism $f\:R\rarrow S$ is called \emph{very flat} if,
for every element $s\in S$, the $R$\+module $S[s^{-1}]$ is
very flat~\cite[Section~9]{PSl1}.

\begin{lem}
 The following two conditions on a commutative ring homomorphism
$f\:R\rarrow S$ are equivalent:
\begin{enumerate}
\item every very flat $S$\+module is also very flat as an $R$\+module;
\item the ring homomorphism $f$~is very flat.
\end{enumerate}
\end{lem}

\begin{proof}
 (1)\,$\Longrightarrow$\,(2) For any element $s\in S$, the $S$\+module
$S[s^{-1}]$ is very flat by the definition.
 Assuming~(1), it follows that $S[s^{-1}]$ has to be a very flat
$R$\+module.

 (2)\,$\Longrightarrow$\,(1) This is~\cite[Lemma~1.2.3(b)]{Pcosh}.
\end{proof}

 The condition that the ring $S$ itself is a very flat $R$\+module
is \emph{not} sufficient for a ring homomorphism $f\:R\rarrow S$
to be very flat.
 For a counterexample, see~\cite[Example~9.7]{PSl1}.

 Here is another manifestation of more complicated behavior of
very flat modules as compared to flat ones.
 Given two flat modules $F$ and $G$ over a commutative ring $R$,
the tensor product $F\ot_RG$ is also a flat $R$\+module.
 Similarly, if $F$ and $G$ are very flat $R$\+modules, then
$F\ot_RG$ is a very flat $R$\+module~\cite[Lemma~1.2.1(a)]{Pcosh}.

 However, let $R$ and $S$ be two commutative rings.
 Let $F$ be an $R$\+$S$\+bimodule and $G$ be an $S$\+module.
 In this context, if $F$ is a flat $R$\+module and $G$ is
a flat $S$\+module, then $F\ot_SG$ is a flat $R$\+module.
 The similar assertion holds for flat (bi)modules over arbitrary
associative rings.
 But the similar assertion for very flat modules is \emph{not} true.

\begin{ex}
 Consider the commutative $R$\+algebra $S$ from~\cite[Example~9.7]{PSl1}
mentioned above.
 So $S$ is a free $R$\+module, but there is an element $s\in S$
such that $S[s^{-1}]$ is \emph{not} a very flat $R$\+module.
 Consider the $R$\+$S$\+bimodule $F=S$ and the $S$\+module
$G=S[s^{-1}]$.
 Then the $R$\+module $F$ is very flat, and the $S$\+module $G$ is
very flat.
 But the $R$\+module $F\ot_SG=S[s^{-1}]$ is not very flat.
\end{ex}

\subsection{Very flat quasi-coherent sheaves and their direct images}
\label{very-flat-sheaves-subsecn}
 Let $X$ be a scheme.
 A quasi-coherent sheaf $\F$ on $X$ is said to be \emph{very flat}
if, for every affine open subscheme $U\subset X$, the $\cO(U)$\+module
$\F(U)$ is very flat.
 It suffices to check this condition for affine open subschemes
$U\subset X$ belonging to any given affine open covering of
the scheme~$X$ \,\cite[Lemma~1.2.6(a) and Section~1.7]{Pcosh}.
 So the very flatness property of quasi-coherent sheaves is local
(cf.~\cite[Example~2.5]{Pal}).

 Let $f\:Y\rarrow X$ be a morphism of schemes.
 Then, for any flat quasi-coherent sheaf $\F$ on $X$,
the inverse image $f^*\F$ on $Y$ of the quasi-coherent sheaf $\F$
on $X$ is also flat.
 Similarly, for any very flat quasi-coherent sheaf $\F$ on $X$,
the quasi-coherent sheaf $f^*\F$ on $Y$ is also very flat.
 This follows from the stability of the classes of flat and
very flat modules under the extension of scalars with respect to
homomorphisms of commutative rings, as per the discussion in
Section~\ref{change-of-scalars-subsecn}.

 Now let $f\:Y\rarrow X$ be an affine morphism of schemes
(i.~e., for every affine open subscheme $U\subset X$, the open
subscheme $f^{-1}(U)\subset Y$ is affine).
 The morphism~$f$ is said to be \emph{flat} if, for every
pair of affine open subschemes $U\subset X$ and $V\subset Y$
such that $f(V)\subset U$, the $\cO_X(U)$\+module $\cO_Y(V)$ is flat.
 This is equivalent to the $\cO_X(U)$\+module $\cO_Y(f^{-1}(U))$
being flat for all affine open subschemes $U\subset X$.
 If this is the case then, for every flat quasi-coherent sheaf
$\G$ on $Y$, the direct image $f_*\G$ on $X$ of the quasi-coherent
sheaf $\G$ on $Y$ is also flat.

 A morphism of schemes $Y\rarrow X$ is said to be \emph{very flat}
if, for every pair of affine open subschemes $U\subset X$ and
$V\subset Y$ such that $f(V)\subset U$, the $\cO_X(U)$\+module
$\cO_Y(V)$ is very flat~\cite[Section~1.7]{Pcosh}.
 For an affine morphism~$f$, very flatness of the $\cO_X(U)$\+modules
$\cO_Y(f^{-1}(U))$ is \emph{not} sufficient for the morphism~$f$ to be
very flat.
 For example, a morphism of affine schemes $f\:\Spec S\rarrow\Spec R$
is very flat if and only if the related morphism of rings $R\rarrow S$
is very flat in the sense of
Section~\ref{change-of-scalars-subsecn}, i.~e., the $R$\+modules
$S[s^{-1}]$ are very flat for all $s\in S$.

 Very flatness of the $R$\+module $S$ would imply very flatness of
the $R'$\+modules $S'=R'\ot_RS$ for all commutative ring
homomorphisms $R\rarrow R'$, including in particular such ring
homomorphisms corresponding to open immersions of affine schemes
$U\rarrow\Spec R$.
 So the $\cO_{\Spec R}(U)$\+modules $\cO_{\Spec S}(f^{-1}(U))$ would be 
very flat in this case.
 But this is \emph{not} sufficient for very flatness of the ring
homomorphism $R\rarrow S$ or the scheme morphism
$f\:\Spec S\rarrow\Spec R$, as explained in
Section~\ref{change-of-scalars-subsecn}.

 Let $f\:Y\rarrow X$ be a very flat affine morphism of schemes.
 Then, for every very flat quasi-coherent sheaf $\G$ on $Y$,
the quasi-coherent sheaf $f_*\G$ on $X$ is also very flat.
 This follows from the stability of the class of very flat modules
under the restriction of scalars with respect to very flat morphisms
of commutative rings, as per Section~\ref{change-of-scalars-subsecn}.

\subsection{Inverse images of contraherent cosheaves}
\label{inverse-images-of-contrah-subsecn}
 Let $f\:Y\rarrow X$ be a morphism of schemes.
 Given a contraherent cosheaf $\P$ on $X$, one would like to construct
a contraherent cosheaf $f^!\P$ on~$Y$ (in the notation of~\cite{Pcosh}).
 There are two problems arising in this connection.

 A minor issue is the nonlocality of contraherence (as per
Remark~\ref{nonlocality-remark}).
 Contraherent cosheaves on the scheme $X$ are defined in terms of
affine open subschemes $U\subset X$, and similarly contraherent
cosheaves on $Y$ are defined in terms of affine open subschemes
$V\subset Y$.
 A morphism of schemes~$f$ is called \emph{coaffine} if, for every
affine open subscheme $V\subset Y$, there exists an affine open
subscheme $U\subset X$ such that $f(V)\subset U$.
 This is \emph{not} always the case.
 Cf.\ the ``Jouanolou trick'' and its generalization by
Thomason~\cite[Proposition~4.3 or~4.4]{Wei}.

 The major issue is that the inverse image $f^!\P$ can be only
defined in the following cases (cf.\ the discussion of partially
defined functors between exact categories in
Section~\ref{main-definition-discussion-subsecn}):
\begin{enumerate}
\item if $\P$ is a locally contraadjusted (locally) contraherent
cosheaf and $f$~is a very flat morphism of schemes, then
$f^!\P$ is a locally contraadjusted locally contraherent cosheaf;
\item if $\P$ is a locally cotorsion (locally) contraherent
cosheaf and $f$~is a flat morphism of schemes, then
$f^!\P$ is a locally cotorsion locally contraherent cosheaf;
\item if $\fJ$ is a locally injective (locally) contraherent
cosheaf and $f$~is an arbitrary morphism of schemes, then $f^!\fJ$
is locally injective locally contraherent cosheaf.
\end{enumerate}

 To illustrate the nature of the major issue, let us consider
the case of a morphism of affine schemes $f\:\Spec S\rarrow\Spec R$.
 Then the contraherent cosheaves on $\Spec R$ are the same things
as $R$\+modules with suitable adjustedness properties, as in
formulas~(\ref{contraadjusted-on-affine}\+-\ref{injective-on-affine})
from Section~\ref{main-definition-discussion-subsecn}.

 The inverse image functor~$f^!$ assigns to an $R$\+module $P$
the $S$\+module $\Hom_R(S,P)$.
 If $P$ is a contraadjusted $R$\+module and $R\rarrow S$ is a very
flat homomorphism, then $\Hom_R(S,P)$ is a contraadjusted
$S$\+module~\cite[Lemma~1.2.3(a)]{Pcosh}.
 If $P$ is a cotorsion $R$\+module and $S$ is a flat $R$\+module,
then $\Hom_R(S,P)$ is a cotorsion
$S$\+module~\cite[Lemma~1.3.5(a)]{Pcosh}.
 If $J$ is an injective $R$\+module, then $\Hom_R(S,J)$ is
an injective $S$\+module~\cite[Lemma~1.3.4(b)]{Pcosh}.

 More generally, for any very flat coaffine morphism of schemes
$f\:Y\rarrow X$, there is a well-defined inverse image functor
$f^!\:X\ctrh\rarrow Y\ctrh$.
 For any flat coaffine morphism of schemes~$f$, there is
a well-defined functor $f^!\:X\ctrh^\lct\rarrow Y\ctrh^\lct$.
 For any coaffine morphism of schemes~$f$, there is
a well-define functor $f^!\:X\ctrh^\lin\rarrow Y\ctrh^\lin$
\,\cite[Section~2.3]{Pcosh}.

 For a noncoaffine morphism of schemes $f\:Y\rarrow X$, the inverse
image functors~$f^!$ in the respective cases~(1\+-3) take contraherent
cosheaves on $X$ to $\bW$\+locally contraherent cosheaves on $Y$,
where $\bW$ is the open covering of $Y$ by the preimages
$f^{-1}(U)\subset Y$ of the affine open subschemes $U\subset X$
\,\cite[Section~3.3]{Pcosh}.

\subsection{The problem}
 The discussion in
Sections~\ref{contramodules-discussion-subsecn} and
and~\ref{co-contra-derived-discussion-subsecn} illustrates the point
that the functor $\Cohom_X$ from quasi-coherent sheaves to contraherent
cosheaves is important.
 According to item~(1) in Section~\ref{cohom-subsecn},
for a locally contraadjusted contraherent cosheaf $\P$ on $X$,
the contraherent cosheaf $\Cohom_X(\F,\P)$ is well-defined whenever
$\F$ is a very flat quasi-coherent sheaf.
 Thus, for the theory of locally contraadjusted contraherent cosheaves
to be useful enough, one needs a plentiful supply of very flat
quasi-coherent sheaves.

 What does it mean ``to have a plentiful supply'' in this context?
 The result of~\cite[Lemma~4.1.1]{Pcosh} tells that, on
a quasi-compact semi-separated scheme $X$, every quasi-coherent
sheaf is a quotient sheaf of a very flat one.
 This is what one usually means by ``having enough flat objects''
(of a particular kind).
 Is that enough for our purposes, or what else is needed?

 It would be helpful to know that quasi-coherent sheaves arising
under various constructions are very flat, at, least, in the contexts
where the ones arising under comparable constructions are flat.
 In particular, the preservation of very flatness of quasi-coherent
sheaves by direct and inverse images is discussed in
Section~\ref{very-flat-sheaves-subsecn}.
 For the direct image functor $f_*\:Y\qcoh\rarrow X\qcoh$ with
respect to an affine morphism of schemes $f\:Y\rarrow X$ to preserve
very flatness, the morphism~$f$ should be very flat.
 Thus, we are interested in having an abundant supply of very flat
morphisms of schemes.

 The discussion of inverse images of contraherent cosheaves in
Section~\ref{inverse-images-of-contrah-subsecn} provides a more
compelling argument supporting the same point.
 According to item~(1), for an inverse image functor $f^!\:X\ctrh
\rarrow Y\lcth_\bW$ to be well-defined (for a suitable open covering
$\bW$ of the scheme~$Y$), the morphism~$f$ should be very flat.
 Surely the inverse image functors are of utmost importance.
 Thus, for the theory of locally contraadjusted (locally) contraherent
cosheaves to be substantial enough, one needs an abundant supply
of very flat morphisms of schemes.

\subsection{The solution} \label{very-flat-solution-subsecn}
 In the context of more geometric chapters of algebraic geometry,
flat morphisms of finite type between Noetherian schemes describe
deformations of algebraic varieties.
 Are all such morphisms very flat?
 This was conjectured in the February~2014 version of the long
prepring~\cite{Pcosh}; see~\cite[Conjecture~1.7.2]{Pcosh}.
 The assertion became known as the \emph{Very Flat Conjecture}.

 The Very Flat Conjecture was proved in the paper~\cite{PSl1}.
 The main result of that paper, \cite[Main Theorem~1.1]{PSl1},
claims the following.
 Let $R\rarrow S$ be a morphism of commutative rings making $S$
a finitely presented $R$\+algebra.
 Assume that $S$ is a flat $R$\+module.
 Then $S$ is a very flat $R$\+module.

 More generally, let $R\rarrow S$ be a morphism of commutative rings
making $S$ a finitely presented $R$\+algebra; and let $F$ be
a finitely presented $S$\+module.
 Assume that $F$ is a flat $R$\+module.
 Then $F$ is a very flat $R$\+module~\cite[Main Theorem~1.2]{PSl1}.

 The following corollary is easily deduced~\cite[Corollary~9.1]{PSl1}.
 Let $R$ be a commutative ring, and let $f\:R\rarrow S$ be a morphism
of commutative rings making $S$ a finitely presented, flat $R$\+algebra.
 Then $f$~is a very flat morphism of commutative rings (in the
sense of Section~\ref{change-of-scalars-subsecn}).

 So very flat morphisms of schemes are, at least, as common as flat
morphisms of finite presentation, which is to say, quite common.
 To give another example, for any (possibly infinite-dimensional)
schemes $X$ and $Y$ over a field~$k$, the natural projection
$X\times_k Y\rarrow X$ is a very flat morphism of schemes.
 This surprisingly nontrivial assertion (cf.\ an early geometric proof
of a particular case in~\cite[Corollary~1.7.15]{Pcosh}) is known to be
true in view of~\cite[Corollary~9.8]{PSl1}.

 Alongside with the class of very flat $R$\+modules, there is
the class of \emph{finitely very flat
$R$\+modules}~\cite[Section~1.3]{PSl1}, which is a subclass of very
flat modules having some technical advantages and disadvantages.
 For any commutative ring $R$, a finitely presented commutative
$R$\+algebra $S$, and a finitely presented $S$\+module $F$ that is
flat over $R$, the $R$\+module $F$ is actually finitely very
flat~\cite[Main Theorem~1.5]{PSl1}.

 The assertions of~\cite[Noetherian Main Lemma~1.4]{PSl1} and
\cite[Finitely Very Flat Main Lemma~1.6]{PSl1}, as well as their
generalizations in~\cite[Theorems~1.11 and~1.12]{PSl1}, provide
potentially useful criteria for (finite) very flatness of flat
modules over commutative rings and flat quasi-coherent sheaves.
 In particular~\cite[Theorem~1.11]{PSl1} tells that the property
of very flatness of flat quasi-coherent sheaves satisfies descent
with respect to surjective morphisms of finite type between
Noetherian schemes.

\subsection{Discussion} \label{very-flat-discussion-subsecn}
 What are the advantages of very flat modules/quasi-coherent sheaves
as compared to the flat ones?
 A closely related question is: What are the advantages of locally
contraadjusted contraherent cosheaves as compared to the locally
cotorsion ones?

 One answer to this question (cf.\ the discussion in
Section~\ref{flat-modules-secn} below) is that the class of all
very flat modules over commutative rings $R$ combines four properties:
\begin{enumerate}
\renewcommand{\theenumi}{\roman{enumi}}
\item all $R$\+modules in this class are flat and have uniformly bounded
finite projective dimension (in fact, projective dimension~$\le1$);
\item this class of $R$\+modules is preserved by transfinitely iterated
extensions (in the sense of Section~\ref{eklof-trlifaj-subsecn}),
direct summands, kernels of surjective morphisms, extensions of scalars,
and tensor products over~$R$;
\item all $R$\+modules are quotient modules of modules from this class;
\item for every affine open subscheme $V$ in an affine scheme $U$,
the $\cO(U)$\+module $\cO(V)$ belong to this class.
\end{enumerate}
 Property~(i) is Theorem~\ref{very-flat-cotorsion-pair}(d)
and~\cite[beginning of Section~1.1]{Pcosh}, property~(ii) was discussed
in Theorem~\ref{very-flat-cotorsion-pair}(a) and
Section~\ref{change-of-scalars-subsecn}, property~(iii) holds because
all projective $R$\+modules are very flat, and property~(iv) is
Example~\ref{open-immersion-very-flat}.
 From this point of view, the results of
Section~\ref{very-flat-solution-subsecn} (the Very Flat Conjecture)
come as a bonus.

 The class of flat modules satisfies~(ii), (iii), and~(iv), but it
need not satisfy~(i).
 Only over Noetherian commutative rings of finite Krull
dimension~\cite[Corollaire~II.3.2.7]{RG} and rings of
cardinality~$\le\aleph_n$, \,$n\in\boZ_{\ge0}$
\,\cite[Th\'eor\`eme~7.10]{GJ} (cf.~\cite[Corollaire~II.3.3.2]{RG})
flat modules have finite projective dimensions.
 Are there any other alternatives?

 For example, the class of all flat $R$\+modules of projective
dimension~$\le n$ has the properties~(i), (iii), and~(iv),
but is unlikely to be preserved by tensor products.
 Still, one can brute force all the four properties by considering
the class of all direct summands of $R$\+modules filtered by
countably presented flat $R$\+modules (in the sense of
Section~\ref{eklof-trlifaj-subsecn}).
 Such modules have projective dimensions at most~$1$
by~\cite[Corollary~2.23]{GT}, and their class is closed under
tensor products.
 Furthermore, all finitely presented $R$\+algebras are obviously
countably presented as $R$\+modules.

 A perhaps less ugly alternative to the class of very flat modules is
the class of \emph{quite flat} modules over commutative rings, which
was introduced in~\cite[Section~8]{PSl2} and studied in
the paper~\cite{HPS}.
 All the four properties~(i\+-iv) are satisfied for this class,
which is intermediate between the class of very flat modules and
the class of direct summands of modules filtered by countably presented
flat modules.
 For Noetherian commutative rings $R$, \,\cite[Theorem~2.4]{HPS} tells
that all (direct summands of modules filtered by) countably generated
flat $R$\+modules are quite flat.

\subsection{Conclusion}
 The class of very flat modules over commutative rings is
the minimal class closed under transfinitely iterated extensions,
direct summands, and containing the $R$\+modules $R[r^{-1}]$ for
all commutative rings $R$ and elements $r\in R$.
 These are weak versions of properties~(ii) and~(iv) from
Section~\ref{very-flat-discussion-subsecn}.
 It is also the minimal \emph{very local} class of modules in
the sense of~\cite[Section~2]{Pal} that is the left-hand class of
a cotorsion pair in $R\modl$ for every commutative ring~$R$.

 Having a large category of contraherent cosheaves is a mixed blessing.
 The larger is the right-hand class of a cotorsion pair, the smaller
is the left-hand class.
 The right-hand class of all contraadjusted modules is very large, and
its corresponding left-hand class of all very flat modules is
the smallest one as compared to possible alternatives.
 But one needs very flat modules, morphisms, and sheaves in order to
perform constructions with the locally contraadjusted contraherent
cosheaves.

 Still, very flat homomorphisms of commutative rings are surprisingly
ubiquitous, as (the proof of) the Very Flat Conjecture tells.
 That is what makes the theory of locally contraadjusted contraherent
cosheaves viable.

\Section{Flat Modules are Not That Much Different From Projectives}
\label{flat-modules-secn}

 Let $U$ be an affine scheme with the ring of functions $R=\cO(U)$.
 One would like to think of the exact category of contraherent
cosheaves on $U$ as being ``essentially'' equivalent to
the abelian category of $R$\+modules, but this not literally true.
 Actually, three versions of the exact category of contraherent
cosheaves on $U$ are described by
formulas~(\ref{contraadjusted-on-affine}\+-\ref{injective-on-affine})
from Section~\ref{main-definition-discussion-subsecn}.

 What about the derived categories?
 Do the exact categories $U\ctrh=R\modl^\cta$, \
$U\ctrh^\lct=R\modl^\cot$, and $U\ctrh^\lin=R\modl^\inj$ have the same
derived categories as the ambient abelian category $R\modl$\,?
 Furthermore, what about the contraderived categories?
 The aim of this section is to discuss these questions by putting
them in a wider context.
 Briefly, the answers are positive for the contraadjusted and
cotorsion modules, and negative for the injective ones.

\subsection{Exact categories of finite homological dimension}
\label{exact-categories-fin-homol-dim-subsecn}
 An exact category $\sE$ is said to have \emph{homological
dimension\/~$\le d$} (for an integer $d\ge-1$) if
$\Ext_\sE^{d+1}(A,B)=0$ for all objects $A$, $B\in\sE$.

 The \emph{coderived} and \emph{contraderived categories}
(\emph{in the sense of Positselski} and \emph{in the sense of Becker})
were discussed in Section~\ref{coderived-and-contraderived-secn}
in the context of the abelian categories of modules over a ring.
 Let us revisit these definitions in the context of exact categories.

 We denote by $\Hot(\sE)$ the homotopy category of unbounded
complexes in an additive category~$\sE$.
 The homotopy categories of bounded below, bounded above, and
two-sided bounded complexes are denoted by $\Hot^+(\sE)$,
\,$\Hot^-(\sE)$, and $\Hot^\bb(\sE)$, as usual.

 A discussion of the acyclic complexes and the derived categories
$\sD^\star(\sE)$ for an exact category $\sE$ (where $\star=\bb$, $+$,
$-$, or~$\varnothing$ is a conventional derived category symbol)
can be found in the papers~\cite{Neem0} and~\cite[Section~10]{Bueh}.
 We refer to~\cite[Section~2.1]{Psemi} for our idiosyncratic terminology
of \emph{exact} vs.\ \emph{acyclic} complexes in exact categories.
 This distinction is mostly not relevant to our tasks in this paper,
as the exact categories we are really interested in are
idempotent-complete.

 A complex in an exact category $\sE$ is said to be \emph{absolutely
acyclic}~\cite[Sections~3.3 and~4.2]{Pkoszul},
\cite[Section~A.1]{Pcosh}, \cite[Appendix~A]{Pmgm},
\cite[Section~7.6]{Pksurv} if it belongs to the minimal thick
subcategory of $\Hot(\sE)$ containing the totalizations of
(termwise admissible) short exact sequences of complexes in~$\sE$.
 The full subcategory of absolutely acyclic complexes is denoted
by $\Ac^\abs(\sE)\subset\Hot(\sE)$.
 The related triangulated Verdier quotient category is denoted by
$$
 \sD^\abs(\sE)=\Hot(\sE)/\Ac^\abs(\sE)
$$
and called the \emph{absolute derived category} of~$\sE$.

\begin{lem} \label{acyclic-absolutely-acyclic}
 Let\/ $\sE$ be an exact category.  Then \par
\textup{(a)} every absolutely acyclic complex in\/ $\sE$ is acyclic;
\par
\textup{(b)} if\/ $\sE$ has finite homological dimension, then every
acyclic complex in\/ $\sE$ is absolutely acyclic.
\end{lem}

\begin{proof}
 In part~(a), one proves that the class of all exact complexes in
an exact category is closed under extensions.
 Part~(b) is~\cite[Remark~2.1]{Psemi}.
\end{proof}

 An exact category $\sA$ is said to have \emph{exact coproducts} if
the infinite coproduct functors are everywhere defined in $\sA$ and
preserve admissible short exact sequences.
 Dually, an exact category $\sB$ is said to have \emph{exact products}
if all the infinite products exist in $\sB$ and preserve admissible
short exact sequences.

 A complex in an exact category $\sA$ is said to be \emph{coacyclic
in the sense of Positselski} if it belongs to the minimal triangulated
subcategory $\Ac^\pco(\sA)$ of the homotopy category $\Hot(\sA)$
containing all the absolutely acyclic complexes and closed under
coproducts.
 The related triangulated Verdier quotient category is denoted by
$$
 \sD^\pco(\sA)=\Hot(\sA)/\Ac^\pco(\sA)
$$
and called the \emph{coderived category of\/ $\sA$ in the sense of
Positselski}.

 Dually, a complex in an exact category $\sB$ is said to be
\emph{contraacyclic in the sense of Positselski} if it belongs to
the minimal triangulated subcategory $\Ac^\pctr(\sB)$ of the homotopy
category $\Hot(\sB)$ containing all the absolutely acyclic complexes
and closed under products.
 The related triangulated Verdier quotient category is denoted by
$$
 \sD^\pctr(\sB)=\Hot(\sB)/\Ac^\pctr(\sB)
$$
and called the \emph{contraderived category of\/ $\sB$ in the sense of
Positselski}.

\begin{lem} \label{Positselski-co-contra-acyclic-vs-acyclic}
\textup{(a)} In an exact category with exact coproducts, every
Positselski-coacyclic complex is acyclic. \par
\textup{(b)} In an exact category of finite homological dimension,
every acyclic complex is Positselski-coacyclic. \par
\textup{(c)} In an exact category with exact products, every
Positselski-contraacyclic complex is acyclic. \par
\textup{(d)} In an exact category of finite homological dimension,
every acyclic complex is Positselski-contraacyclic.
\end{lem}

\begin{proof}
 Follows from Lemma~\ref{acyclic-absolutely-acyclic}.
\end{proof}

 A complex $A^\bu$ in an exact category $\sA$ is said to be
\emph{coacyclic in the sense of Becker} if, for every complex of
injective objects $J^\bu$ in $\sA$, the complex of abelian groups
$\Hom_\sA(A^\bu,J^\bu)$ is acyclic.
 The full subcategory of Becker-coacyclic complexes is denoted by
$\Ac^\bco(\sA)\subset\Hot(\sA)$.
 The related triangulated Verdier quotient category is denoted by
$$
 \sD^\bco(\sA)=\Hot(\sA)/\Ac^\bco(\sA)
$$
and called the \emph{coderived category of\/ $\sA$ in the sense of
Becker}.

 An alternative definition of the coderived category in the sense of
Becker is simply the homotopy category of complexes of injective
objects, $\Hot(\sA^\inj)$.
 There is an obvious fully faithful triangulated functor
$\Hot(\sA^\inj)\rarrow\sD^\bco(\sA)$, constructed as the composition
$\Hot(\sA^\inj)\rarrow\Hot(\sA)\rarrow\sD^\bco(\sA)$.
 For Grothendieck abelian categories $\sA$, this functor is
a triangulated equivalence~\cite[Corollary~9.5]{PS4},
\cite[Corollary~7.9]{PS5}.

 A complex $B^\bu$ in an exact category $\sB$ is said to be
\emph{contraacyclic in the sense of Becker} if, for every complex of
projective objects $P^\bu$ in $\sB$, the complex of abelian groups
$\Hom_\sB(P^\bu,B^\bu)$ is acyclic.
 The full subcategory of Becker-contraacyclic complexes is denoted by
$\Ac^\bctr(\sB)\subset\Hot(\sB)$.
 The related triangulated Verdier quotient category is denoted by
$$
 \sD^\bctr(\sB)=\Hot(\sB)/\Ac^\bctr(\sB)
$$
and called the \emph{contraderived category of\/ $\sB$ in the sense of
Becker}.

 An alternative definition of the contraderived category in the sense of
Becker is simply the homotopy category of complexes of injective
objects, $\Hot(\sB_\proj)$.
 There is an obvious fully faithful triangulated functor
$\Hot(\sB_\proj)\rarrow\sD^\bctr(\sB)$, constructed as the composition
$\Hot(\sB_\proj)\rarrow\Hot(\sB)\rarrow\sD^\bctr(\sB)$.
 For locally presentable abelian categories $\sB$ with enough
projective objects, this functor is a triangulated
equivalence~\cite[Corollary~7.4]{PS4}, \cite[Corollary~6.13]{PS5}.

\begin{lem} \label{Positselski-trivial-implies-Becker-trivial}
\textup{(a)} Any Positselski-coacyclic complex in an exact category
is Becker-coacyclic. \par
\textup{(b)} Any Positselski-contraacyclic complex in an exact category
is Becker-contraacyclic. \hbadness=2150
\end{lem}

\begin{proof}
 Part~(a) is~\cite[proof of Theorem~3.5(a)]{Pkoszul},
\cite[Lemma~9.1]{PS4}, or~\cite[Theorem~5.5(a)]{Pedg}.
 Part~(b) is~\cite[proof of Theorem~3.5(b)]{Pkoszul},
\cite[Lemma~7.1]{PS4}, or~\cite[Theorem~5.5(b)]{Pedg}.
\end{proof}

\begin{lem}
\textup{(a)} In an abelian category with enough injective objects,
every Becker-coacyclic complex is acyclic. \par
\textup{(b)} In an exact category of finite homological dimension,
every acyclic complex is Becker-coacyclic. \par
\textup{(c)} In an abelian category with enough projective objects,
every Becker-contraacyclic complex is acyclic.
\par
\textup{(d)} In an exact category of finite homological dimension,
then every acyclic complex is Becker-contraacyclic.
\end{lem}

\begin{proof}
 Part~(b) follows from
Lemmas~\ref{Positselski-co-contra-acyclic-vs-acyclic}(b)
and~\ref{Positselski-trivial-implies-Becker-trivial}(a).
 Part~(d) follows from
Lemmas~\ref{Positselski-co-contra-acyclic-vs-acyclic}(d)
and~\ref{Positselski-trivial-implies-Becker-trivial}(b).
 Part~(a) is~\cite[Lemma~A.2]{Psemten}, and part~(c) is dual.
\end{proof}

\begin{prop} \label{enough-inj-proj-finite-homol-dim-prop}
\textup{(a)} Let\/ $\sA$ be an exact category of finite homological
dimension with enough injective objects.
 Then the clases of acyclic, absolutely acyclic, and Becker-coacyclic
complexes in\/ $\sA$ coincide.
 The composition of triangulated functors\/ $\Hot(\sA^\inj)\rarrow
\Hot(\sA)\rarrow\sD^\abs(\sA)=\sD^\bco(\sA)=\sD(\sA)$ is
a triangulated equivalence
$$
 \Hot(\sA^\inj)\simeq\sD^\abs(\sA)=\sD^\bco(\sA)=\sD(\sA).
$$ \par
\textup{(b)} Let\/ $\sB$ be an exact category of finite homological
dimension with enough projective objects.
 Then the classes of acyclic, absolutely acyclic, and
Becker-contraacyclic complexes in\/ $\sB$ coincide.
 The composition of triangulated functors\/ $\Hot(\sB_\proj)\rarrow
\Hot(\sB)\rarrow\sD^\abs(\sB)=\sD^\bctr(\sB)=\sD(\sB)$ is
a triangulated equivalence
$$
 \Hot(\sB_\proj)\simeq\sD^\abs(\sB)=\sD^\bctr(\sB)=\sD(\sB).
$$
\end{prop}

\begin{proof}
 One has $\Ac^\abs(\sA)=\Ac(\sA)$ and $\Ac^\abs(\sB)=\Ac(\sB)$ by
Lemma~\ref{acyclic-absolutely-acyclic}.
 Furthermore, $\Ac^\abs(\sA)\subset\Ac^\bco(\sA)$ and
$\Ac^\abs(\sB)\subset\Ac^\bctr(\sB)$ by
Lemma~\ref{Positselski-trivial-implies-Becker-trivial}.
 The argument with the totalization of a finite acyclic complex of
complexes similar to~\cite[proof of Theorem~3.6]{Pkoszul}
or~\cite[proof of Theorem~5.6]{Pedg} proves that the compositions
$\Hot(\sA^\inj)\rarrow\Hot(\sA)\rarrow\sD^\abs(\sA)$ and
$\Hot(\sB_\proj)\rarrow\Hot(\sB)\rarrow\sD^\abs(\sB)$ are
triangulated equivalences.
 Then it follows that $\Ac^\abs(\sA)=\Ac^\bco(\sA)$ and
$\Ac^\abs(\sB)=\Ac^\bctr(\sB)$.
 For further discussion, see~\cite[Theorem~7.8]{Pksurv}.
\end{proof}

\begin{prop}
\textup{(a)} Let\/ $\sA$ be an exact category with coproducts
and enough injective objects.
 Assume that countable coproducts of injective objects have finite
injective dimensions in\/~$\sA$.
 Then the clases of Positselski-coacyclic and Becker-coacyclic
complexes in\/ $\sA$ coincide.
 The composition of triangulated functors\/ $\Hot(\sA^\inj)\rarrow
\Hot(\sA)\rarrow\sD^\pco(\sA)=\sD^\bco(\sA)$ is
a triangulated equivalence
$$
 \Hot(\sA^\inj)\simeq\sD^\pco(\sA)=\sD^\bco(\sA).
$$ \par
\textup{(b)} Let\/ $\sB$ be an exact category with products
and enough projective objects.
 Assume that countable products of projective objects have finite
projective dimensions in\/~$\sB$.
 The composition of triangulated functors\/ $\Hot(\sB_\proj)\rarrow
\Hot(\sB)\rarrow\sD^\pctr(\sB)=\sD^\bctr(\sB)$ is
a triangulated equivalence
$$
 \Hot(\sB_\proj)\simeq\sD^\pco(\sB)=\sD^\bctr(\sB).
$$
\end{prop}

\begin{proof}
 Existence of coproducts and enough injective objects implies exactness
of coproducts in an exact category, and dually, existence of products
and enough projective objects implies exactness of
products~\cite[Remark~5.2]{Pedg}.
 The assertions of the proposition are provable similarly
to~\cite[Theorems~3.7\+-3.8]{Pkoszul}.
 For a generalization, see~\cite[Theorem~5.10]{Pedg}.
 See also~\cite[Theorem~7.9]{Pksurv}.
\end{proof}

 Conclusions: as a general rule, all absolutely acyclic, coacyclic,
or contraacyclic complexes tend to be acyclic.
 However, \emph{we do not know} how to prove that all Becker
co/contraacyclic complexes in an exact category are acyclic without
making some restrictive assumptions, or what kind of restrictive
assumptions would be optimal for the task.
 This is not a problem for abelian categories.

 On the other hand, in any exact category of finite homological
dimension, all acyclic complexes are absolutely acyclic.
 Hence they are also coacyclic and contraacyclic both in the senses
of Positselski and Becker.

\subsection{Exact subcategories of finite (co)resolution dimension}
\label{finite-coresolution-dimension-subsecn}
 Let $\sE$ be an exact category.
 A full subcategory $\sF\subset\sE$ is said to be \emph{resolving} if
\begin{itemize}
\item the full subcategory $\sF$ is closed under extensions and kernels
of admissible epimorphisms in~$\sE$;
\item the full subcategory $\sF$ is generating in $\sE$ (in the sense
of Section~\ref{cotorsion-pairs-subsecn}).
\end{itemize}

 Dually, a full subcategory $\sC\subset\sE$ is said to be
\emph{coresolving} if
\begin{itemize}
\item the full subcategory $\sC$ is closed under extensions and
cokernels of admissible monomorphisms in~$\sE$;
\item the full subcategory $\sC$ is cogenerating in $\sE$ (in the sense
of Section~\ref{cotorsion-pairs-subsecn}).
\end{itemize}

 For example, the left-hand classes of hereditary cotorsion pairs are
resolving, and the right-hand classes are coresolving.
 Furthermore, any (co)resolving subcategory is closed under extensions;
so it inherits an exact category structure from the ambient exact
category.

 The classical concepts of projective, injective, and flat dimensions
generalize naturally to (co)resolving subcategories.
 Let us assume the exact category $\sE$ to be, at least, weakly
idempotent-complete (in the sense of~\cite[Section~7]{Bueh}).

 Given a resolving subcategory $\sF$ in an exact category $\sE$, one
says that an object $E\in\sE$ has \emph{resolution dimension~$\le d$}
with respect to $\sF$ if there exists an exact sequence 
$$
 0\lrarrow F_d\lrarrow F_{d-1}\lrarrow\dotsb\lrarrow F_0\lrarrow E
 \lrarrow 0
$$
in $\sE$ with $F_i\in\sF$ for all $0\le i\le d$.
 The \emph{coresolution dimension} with respect to a coresolving
subcategory is defined dually.

 The (co)resolution dimension of an object with respect to
a (co)resolving subcategory does not depend on the choice of
a (co)resolution, in the same sense as the projective/injective/flat
dimension of a module does not depend on the choice of a (co)resolution.
 This result can be found in~\cite[Proposition~2.3]{Sto1}
or~\cite[Corollary~A.5.2]{Pcosh}; see also~\cite[Proposition~6.2]{Pedg}.

\begin{prop} \label{one-sided-bounded-resolving-subcategory-prop}
 Let\/ $\sE$ be an exact category, $\sF\subset\sE$ be a resolving
subcategory, and\/ $\sC\subset\sE$ be a coresolving subcategory.
 Then \par
\textup{(a)} the inclusion of exact categories\/ $\sC\rarrow\sE$
induces a triangulated equivalence of bounded below derived
categories\/ $\sD^+(\sC)\rarrow\sD^+(\sE)$; \par
\textup{(b)} the inclusion of exact categories\/ $\sF\rarrow\sE$
induces a triangulated equivalence of bounded above derived
categories\/ $\sD^-(\sF)\rarrow\sD^-(\sE)$.
\end{prop}

\begin{proof}
 This well-known result goes back to~\cite[Lemma~I.4.6]{HartRD}.
 See~\cite[Proposition~13.2.2]{KS} or~\cite[Proposition~A.3.1(a)]{Pcosh}
for some further details.
\end{proof}

\begin{prop} \label{co-contra-resolving-subcategory-prop}
\textup{(a)} Let\/ $\sA$ be an exact category with exact coproducts
and\/ $\sC\subset\sA$ be a coresolving subcategory closed under
coproducts.
 Then the inclusion of exact categories\/ $\sC\rarrow\sA$ induces
a triangulated equivalence of the Positselski-coderived categories\/
$\sD^\pco(\sC)\rarrow\sD^\pco(\sA)$. \par
\textup{(b)} Let\/ $\sB$ be an exact category with exact products
and\/ $\sF\subset\sB$ be a resolving subcategory closed under products.
 Then the inclusion of exact categories\/ $\sF\rarrow\sB$ induces
a triangulated equivalence of the Positselski-contraderived categories\/
$\sD^\pctr(\sF)\rarrow\sD^\pctr(\sB)$.
\end{prop}

\begin{proof}
 This result goes back to~\cite[Proposition and Remark~1.5]{EP}.
 Part~(b) in its stated form can be found
in~\cite[Proposition~A.3.1(b)]{Pcosh}.
 For a generalization, see~\cite[Theorem~7.11]{Pedg}.
\end{proof}

 There is not much we can say concerning the comparison of
the co/contraderived categories in the sense of Becker for an exact
category and its (co)resolving subcategory.
 All we can offer is the following lemma.

\begin{lem} \label{proj-inj-objects-in-co-resolving-lemma}
\textup{(a)} Let\/ $\sA$ be an exact category and\/ $\sC\subset\sA$
be a coresolving subcategory closed under direct summands.
 Then the classes of injective objects in\/ $\sA$ and\/ $\sC$
coincide, $\sC^\inj=\sA^\inj$.
 If there are enough injective objects in\/ $\sA$, then there are also
enough injective objects in\/~$\sC$. \par
\textup{(b)} Let\/ $\sB$ be an exact category and\/ $\sF\subset\sB$
be a resolving subcategory closed under direct summands.
 Then the classes of projective objects in\/ $\sB$ and\/ $\sF$
coincide, $\sF_\proj=\sB_\proj$.
 If there are enough projective objects in\/ $\sB$, then there are also
enough projective objects in\/~$\sF$.  \qed
\end{lem}

 The two propositions and lemma above exhaust our list of results
about the derived/coderived/contraderived categories for
\emph{arbitrary} (co)resolving subcategories.
 The assertions below assume \emph{finite (co)resolution dimension}.

\begin{prop} \label{finite-co-resol-dim-conventional-derived}
 Let\/ $\sE$ be an exact category, $\sF\subset\sE$ be a resolving
subcategory, and\/ $\sC\subset\sE$ be a coresolving subcategory.
 Assume that there exists a finite integer $d\ge0$ such that all
objects of\/ $\sE$ have resolution dimensions\/~$\le d$ with respect
to\/ $\sF$ and coresolution dimensions\/~$\le d$ with respect
to\/~$\sC$.
 Then\/ \par
\textup{(a)} for every conventional or absolute derived category
symbol\/ $\star=\bb$, $+$, $-$, $\varnothing$, or\/~$\abs$,
the inclusion of exact categories\/ $\sC\rarrow\sE$ induces
a triangulated equivalence of the respective derived categories\/
$\sD^\star(\sC)\rarrow\sD^\star(\sE)$; \par
\textup{(b)} for every conventional or absolute derived category
symbol\/ $\star=\bb$, $+$, $-$, $\varnothing$, or\/~$\abs$,
the inclusion of exact categories\/ $\sF\rarrow\sE$ induces
a triangulated equivalence of derived categories\/
$\sD^\star(\sF)\rarrow\sD^\star(\sE)$.
\end{prop}

\begin{proof}
 Some cases are covered by
Proposition~\ref{one-sided-bounded-resolving-subcategory-prop} and
do not need the bounded (co)resolution dimension assumption.
 Otherwise, the case $\star=\varnothing$ can be found
in~\cite[Proposition~5.14]{Sto1}.
 The case $\star=\abs$ goes back, essentially,
to~\cite[Theorem~7.2.2(a)]{Psemi}; see also~\cite[Theorem~1.4]{EP}.
 In the form stated above, all cases of part~(b) can be found
in~\cite[Proposition~A.5.6]{Pcosh}.
 For a generalization, see~\cite[Theorem~6.6]{Pedg}.
\end{proof}

\begin{ex} \label{thematic-counterex}
 With the finite (co)resolution dimension conditions dropped,
the assertions of
Proposition~\ref{finite-co-resol-dim-conventional-derived}
certainly do \emph{not} stay true.
 It suffices to consider the abelian category $\sE=R\modl$
for a ring $R$ of infinite homological dimension, the resolving
subcategory of projective modules $\sF=R\modl_\proj$,
the coresolving subcategory of injective modules $\sC=R\modl^\inj$,
and $\star=\bb$ or~$\varnothing$.

 Indeed, no $R$\+module of infinite projective dimension belongs
to the essential image of the functor $\Hot^\bb(R\modl_\proj)=
\sD^\bb(R\modl_\proj)\rarrow\sD^\bb(R\modl)$, and no $R$\+module
of infinite injective dimension belongs to the essential image of
the functor $\Hot^\bb(R\modl^\inj)=\sD^\bb(R\modl^\inj)\rarrow
\sD^\bb(R\modl)$.
 On the other hand, the thematic example of an unbounded complex of
projective-injective $R$\+modules representing nonzero objects
in $\Hot(R\modl_\proj)$ and $\Hot(R\modl^\inj)$, but vanishing
as an object of $\sD(R\modl)$, can be found in~\cite[formula~(2)
in Section~7.4]{Pksurv}.
\end{ex}

 The following proposition contains a full list of cases for
the coderived and contraderived categories in the sense of
Positselski in the context of (co)resolving subcategories with
finite (co)resolution dimension.

\begin{prop} \label{finite-co-resol-dim-Positselski-second-kind}
 Let\/ $\sE$ be an exact category, $\sF\subset\sE$ be a resolving
subcategory, and\/ $\sC\subset\sE$ be a coresolving subcategory.
 Assume that there exists a finite integer $d\ge0$ such that all
objects of\/ $\sE$ have resolution dimensions\/~$\le d$ with respect
to\/ $\sF$ and coresolution dimensions\/~$\le d$ with respect
to\/~$\sC$.
 Then\/ \par
\textup{(a)} assuming that the exact category\/ $\sE$ has exact
products and the full subcategory $\sC\subset\sE$ is closed under
products, the inclusion of exact categories\/ $\sC\rarrow\sE$ induces
a triangulated equivalence of the Positselski-contraderived
categories\/ $\sD^\pctr(\sC)\rarrow\sD^\pctr(\sE)$; \par
\textup{(b)} assuming that the exact category\/ $\sE$ has exact
coproducts and the full subcategory $\sC\subset\sE$ is closed under
coproducts, the inclusion of exact categories\/ $\sC\rarrow\sE$ induces
a triangulated equivalence of the Positselski-coderived categories\/
$\sD^\pco(\sC)\rarrow\sD^\pco(\sE)$; \par
\textup{(c)} assuming that the exact category\/ $\sE$ has exact
products and the full subcategory $\sF\subset\sE$ is closed under
products, the inclusion of exact categories\/ $\sF\rarrow\sE$ induces
a triangulated equivalence of the Positselski-contraderived
categories\/ $\sD^\pctr(\sF)\rarrow\sD^\pctr(\sE)$; \par
\textup{(d)} assuming that the exact category\/ $\sE$ has exact
coproducts and the full subcategory $\sF\subset\sE$ is closed under
coproducts, the inclusion of exact categories\/ $\sF\rarrow\sE$ induces
a triangulated equivalence of the Positselski-coderived categories\/
$\sD^\pco(\sF)\rarrow\sD^\pco(\sE)$.
\end{prop}

\begin{proof}
 Parts~(b) and~(c) are particular cases of
Proposition~\ref{co-contra-resolving-subcategory-prop}.
 Part~(d) goes back to~\cite[Theorem~7.2.2(a)]{Psemi}
and~\cite[Theorem~1.4]{EP}.
 In the form stated above, it can be found
in~\cite[Proposition~A.5.6]{Pcosh}.
 For a generalization, see~\cite[Theorem~6.6]{Pedg}.
\end{proof}

 Conclusions: as a general rule, the passage to a (co)resolving
subcategory with uniformly bounded (co)resolution dimension tends to
not change the derived, coderived, and contraderived categories.
 This has been established for the conventional derived categories
and the co/contraderived categories in the sense of Positselski.
 For the co/contraderived categories in the sense of Becker,
the formulations and proofs of such assertions remain to be worked out.
 When the condition of finite (co)resolution dimension is dropped,
the situation becomes more varied.

\subsection{Derived categories of flat and projective modules}
\label{derived-of-flats-and-projectives-subsecn}
 Let $R$ be an associative ring.
 Consider the exact category of flat $R$\+modules $\sE=R\modl_\fl$
(with the exact structure inherited from the abelian exact structure
of the ambient abelian category $R\modl$).
 Then the full subcategory of projective $R$\+modules $\sF=R\modl_\proj
\subset R\modl_\fl$ is a resolving subcategory in $R\modl_\fl$.

 The resolution dimensions of objects of $R\modl_\fl$ with respect to
$R\modl_\proj$ (i.~e., the projective dimensions of flat $R$\+modules)
are finite when $R$ is a Noetherian commutative ring of finite
Krull dimension~\cite[Corollaire~II.3.2.7]{RG}, when the ring $R$
has cardinality~$\le\aleph_n$, \,$n\in\boZ_{\ge0}$
\,\cite[Th\'eor\`eme~7.10]{GJ} (cf.~\cite[Corollaire~II.3.3.2]{RG}),
and for some noncommutative rings with dualizing complexes~\cite[first
assertion of Proposition~1.5]{CFH}, \cite[Proposition~4.3]{Pfp}.
 But, generally speaking, the projective dimensions of flat modules
can be infinite.
 Nevertheless, the following theorem holds.

\begin{thm} \label{projectives-in-flats-resolving-derived}
 For any ring $R$, the inclusion of exact categories
$R\modl_\proj\rarrow R\modl_\fl$ induces an equivalence of their
unbounded derived categories,
$$
 \Hot(R\modl_\proj)=\sD(R\modl_\proj)\simeq\sD(R\modl_\fl).
$$
\end{thm}

\begin{proof}
 This is~\cite[Proposition~8.1 and
Theorem~8.6(i)\,$\Leftrightarrow$\,(iii)]{Neem}.
 Notice that, by the definition, the acyclic complexes in the exact
category $R\modl_\fl$ are precisely all the acyclic complexes of
flat $R$\+modules \emph{with flat $R$\+modules of cocycles} (rather
than arbitrary acyclic complexes of flat $R$\+modules).
\end{proof}

 Thus, even though the assumption of
Proposition~\ref{finite-co-resol-dim-conventional-derived}
does \emph{not} hold for the resolving subcategory $\sF=R\modl_\proj$
in the exact category $\sE=R\modl_\fl$, the conclusion of
Proposition~\ref{finite-co-resol-dim-conventional-derived}(b) still
holds true in this case for the derived category symbols $\star=-$
(by Proposition~\ref{one-sided-bounded-resolving-subcategory-prop}(b))
and, much more nontrivially, $\star=\varnothing$ (by
Theorem~\ref{projectives-in-flats-resolving-derived}).
 In this sense, one can say that the resolving subcategory
$R\modl_\proj$ in the exact category $R\modl_\fl$ behaves ``as if
the resolution dimensions were finite''.

 For a generalization of
Theorem~\ref{projectives-in-flats-resolving-derived}
to graded-flat and graded-projective (curved) DG\+modules over
a (curved) DG\+ring, see~\cite[Theorem~4.12]{PS7}.

\subsection{Becker-contraderived category of flat modules}
\label{contraderived-of-flats-subsecn}
 The following theorem is largely a restatement of
Theorem~\ref{projectives-in-flats-resolving-derived}.
 Notice that, by Lemma~\ref{proj-inj-objects-in-co-resolving-lemma}(b),
the projective objects of the exact category $R\modl_\fl$ are
precisely all the projective $R$\+modules, $(R\modl_\fl)_\proj=
R\modl_\proj$.
 There are enough projective objects in $R\modl_\fl$.

\begin{thm} \label{contraderived-category-of-flats}
 For any ring $R$, the class of all Becker-contraacyclic complexes
in the exact category\/ $\sB=R\modl_\fl$ coincides with the class of
all acyclic complexes in this exact category.
 The composition of triangulated functors\/ $\Hot(R\modl_\proj)\rarrow
\Hot(R\modl_\fl)\rarrow\sD^\bctr(R\modl_\fl)=\sD(R\modl_\fl)$ is
a triangulated equivalence
$$
 \Hot(R\modl_\proj)\simeq\sD^\bctr(R\modl_\fl)=\sD(R\modl_\fl).
$$
\end{thm}

\begin{proof}
 The first assertion is
precisely~\cite[Theorem~8.6(i)\,$\Leftrightarrow$\,(iii)]{Neem}.
 The second assertion is~\cite[Proposition~8.1]{Neem}.
\end{proof}

 Thus, even though the assumptions of
Proposition~\ref{enough-inj-proj-finite-homol-dim-prop}(b)
do \emph{not} hold for the exact category $\sB=R\modl_\fl$,
the conclusion of this proposition still partly holds true for
this exact category.
 In this sense, one can say that the exact category $R\modl_\fl$
behaves ``as if its homological dimension were finite''
(see also Section~\ref{coderived-of-flats-subsecn} below).

\subsection{Derived categories of cotorsion and all modules}
\label{derived-of-cotorsion-and-all-subsecn}
 For any ring $R$, the full subcategory of cotorsion $R$\+modules
$\sC=R\modl^\cot$ is a coresolving subcategory in the abelian
category of $R$\+modules $\sE=R\modl$.
 By the \emph{cotorsion dimension} of an $R$\+module one means its
coresolution dimension with respect to the coresolving subcategory
$R\modl^\cot\subset R\modl$.

 Let $\sE$ be an exact category, $A$, $B\in\sE$ be two objects,
and $d\ge-1$ be an integer.
 One says that the object $A$ has \emph{projective dimension\/~$\le d$}
in $\sE$ if $\Ext_\sE^{d+1}(A,E)=0$ for all objects $E\in\sE$.
 Dually, one says that the object $B$ has \emph{injective
dimension\/~$\le d$} in $\sE$ if $\Ext_\sE^{d+1}(E,B)=0$ for all
objects $E\in\sE$.

 We start with a straightforward lemma and corollary.
 
\begin{lem} \label{projective-and-coresolution-dimensions}
 Let $\sE$ be an exact category, $(\sF,\sC)$ be a hereditary
cotorsion pair in\/ $\sE$, and $d\ge0$ be an integer.
 Then all objects of\/ $\sF$ have projective dimensions\/~$\le d$
in\/ $\sE$ if and only if all objects of\/ $\sE$ have coresolution
dimensions\/~$\le d$ with respect to\/~$\sC$. \qed
\end{lem}

\begin{cor} \label{projdims-of-flats-and-cotorsion-dimensions}
 For any ring $R$ and integer $d\ge0$, the following two conditions
are equivalent:
\begin{enumerate}
\item the projective dimensions of all flat $R$\+modules do not
exceed~$d$;
\item the cotorsion dimensions of all $R$\+modules do not exceed~$d$.
\qed
\end{enumerate}
\end{cor}

 So the cotorsion dimensions of modules can be infinite, generally
speaking.
 Nevertheless, the following theorem holds.

\begin{thm} \label{cotorsion-in-all-coresolving-derived}
 For any ring $R$, the inclusion of exact categories $R\modl^\cot
\rarrow R\modl$ induces an equivalence of their unbounded derived
categories,
$$
 \sD(R\modl^\cot)\simeq\sD(R\modl).
$$
\end{thm}

\begin{proof}
 This is essentially~\cite[first assertion of Theorem~5.1(2)]{BCE}.
 See~\cite[Proposition~9.6]{PS6} for additional details.
\end{proof}

 Thus, even though the assumption of
Proposition~\ref{finite-co-resol-dim-conventional-derived}
does \emph{not} hold for the coresolving subcategory $\sC=R\modl^\cot$
in the exact category $\sE=R\modl$, the conclusion of
Proposition~\ref{finite-co-resol-dim-conventional-derived}(a) still
holds true in this case for the derived category symbols $\star=+$
(by Proposition~\ref{one-sided-bounded-resolving-subcategory-prop}(a))
and, much more nontrivially, $\star=\varnothing$ (by
Theorem~\ref{cotorsion-in-all-coresolving-derived}).
 In this sense, one can say that the coresolving subcategory
$R\modl^\cot$ in the abelian category $R\modl$ behaves ``as if
the coresolution dimensions were finite''.

 For a version of Theorem~\ref{cotorsion-in-all-coresolving-derived}
for (curved) DG\+modules over a (curved) DG\+ring,
see~\cite[Theorem~3.7]{PS7}.

\subsection{Becker-coderived category of flat modules}
\label{coderived-of-flats-subsecn}
 An $R$\+module is said to be \emph{flat cotorsion} if it is
simultaneously flat and cotorsion.
 The full subcategory of flat cotorsion $R$\+modules is denoted by
$R\modl_\fl^\cot=R\modl_\fl\cap R\modl^\cot\subset R\modl$.

 The flat cotorsion $R$\+modules are precisely all the injective
objects of the exact category of flat $R$\+modules $R\modl_\fl$,
and there are enough such injective objects.
 The flat cotorsion $R$\+modules are also precisely all the projective
objects of the exact category of cotorsion $R$\+modules $R\modl^\cot$,
and there are enough such projective objects.
 These assertions follow from the fact that the flat cotorsion pair
in $R\modl$ is complete (by Theorem~\ref{flat-cotorsion-pair});
see~\cite[Lemma~2.1]{Pgen}.

\begin{thm} \label{coderived-category-of-flats}
 For any ring $R$, the class of all Becker-coacyclic complexes in
the exact category\/ $\sA=R\modl_\fl$ coincides with the class of
all acyclic complexes in this exact category.
 The composition of triangulated functors\/ $\Hot(R\modl_\fl^\cot)
\rarrow\Hot(R\modl_\fl)\rarrow\sD^\bco(R\modl_\fl)=\sD(R\modl_\fl)$
is a triangulated equivalence
$$
 \Hot(R\modl_\fl^\cot)\simeq\sD^\bco(R\modl_\fl)=\sD(R\modl_\fl).
$$
\end{thm}

\begin{proof}
 The inclusion $\Ac(R\modl_\fl)\subset\Ac^\bco(R\modl_\fl)$
means that, for any acyclic complex of flat $R$\+modules $F^\bu$
with flat modules of cocycles, and for any complex of flat
cotorsion $R$\+modules $Q^\bu$, the complex of abelian groups
$\Hom_R(F^\bu,Q^\bu)$ is acyclic.
 More generally, this assertion holds for any complex of (not
necessarily flat) cotorsion $R$\+modules~$Q^\bu$.
 This is the result of~\cite[Theorem~5.3]{BCE}.

 It remains to show that for every complex of flat $R$\+modules
$G^\bu$ there exists a complex of flat cotorsion $R$\+modules
$Q^\bu$ together with a closed morphism of complexes of $R$\+modules
$G^\bu\rarrow Q^\bu$ whose cone is an acyclic complex of flat
$R$\+modules with flat $R$\+modules of cocycles.
 Indeed, by~\cite[Definition~3.3 and Corollary~4.10]{Gil}, for
any complex of $R$\+modules $G^\bu$, there exists a short exact
sequence of complexes of $R$\+modules $0\rarrow G^\bu\rarrow
Q^\bu\rarrow F^\bu\rarrow0$, where $Q^\bu$ is a complex of
cotorsion $R$\+modules and $F^\bu$ is an acyclic complex of flat
$R$\+modules with flat $R$\+modules of cocycles.
 Now if $G^\bu$ is a complex of flat $R$\+modules, then $Q^\bu$
is a complex of flat cotorsion $R$\+modules.
 Since $F^\bu\in\Ac(R\modl_\fl)$ and the totalization of the short
exact sequence $0\rarrow G^\bu\rarrow Q^\bu\rarrow F^\bu\rarrow0$
belongs to $\Ac^\abs(R\modl_\fl)\subset\Ac(R\modl_\fl)$ (see
Lemma~\ref{acyclic-absolutely-acyclic}(a)), the cone of the morphism
$G^\bu\rarrow Q^\bu$ also belongs to the triangulated subcategory
$\Ac(R\modl_\fl)\subset\Hot(R\modl_\fl)$.
\end{proof}

 Thus, even though the assumptions of
Proposition~\ref{enough-inj-proj-finite-homol-dim-prop}(a)
do \emph{not} hold for the exact category $\sB=R\modl_\fl$,
the conclusion of this proposition still partly holds true for
this exact category.
 In this sense, one can say once again that the exact category
$R\modl_\fl$ behaves ``as if its homological dimension were finite''
(cf.\ Section~\ref{contraderived-of-flats-subsecn}).

 For a version of Theorem~\ref{coderived-category-of-flats} for
(curved) DG\+modules over a (curved) DG\+ring,
see~\cite[Theorem~4.13]{PS7}.

\subsection{Becker-contraderived categories of cotorsion and
all modules} \label{contraderived-of-cotorsion-subsecn}
 As mentioned in the beginning of
Section~\ref{coderived-of-flats-subsecn}, the flat cotorsion
$R$\+modules are the projective objects of the exact category of
cotorsion $R$\+modules $R\modl^\cot$.
 For any exact category $\sE$, we denote by $\Com(\sE)$ the category
of all cochain complexes in~$\sE$.

\begin{thm} \label{cotorsion-in-all-coresolving-contraderived-thm}
 For any ring $R$, the class of all Becker-contraacyclic complexes
in the exact category\/ $R\modl^\cot$ coincides with the class
of all complexes of cotorsion $R$\+modules that are
Becker-contraacyclic as complexes in the abelian category $R\modl$,
$$
 \Ac^\bctr(R\modl^\cot)=\Com(R\modl^\cot)\cap\Ac^\bctr(R\modl).
$$
 The composition of triangulated functors\/ $\Hot(R\modl_\fl^\cot)
\rarrow\Hot(R\modl^\cot)\rarrow\sD^\bctr(R\modl^\cot)$ is
a triangulated equivalence
$$
 \Hot(R\modl_\fl^\cot)\simeq\sD^\bctr(R\modl^\cot).
$$
\end{thm}

\begin{proof}
 It follows from Theorems~\ref{contraderived-category-of-flats}
and~\ref{coderived-category-of-flats} that, for every complex of
projective $R$\+modules $P^\bu$ there exists a complex of flat
cotorsion $R$\+modules $Q^\bu$ together with a morphism of complexes
of $R$\+modules $P^\bu\rarrow Q^\bu$ whose cone is acyclic in
the exact category $R\modl_\fl$.
 Conversely, for every complex of flat cotorsion $R$\+modules $Q^\bu$
there exists a complex of projective $R$\+modules $P^\bu$ together
with a morphism of complexes of $R$\+modules $P^\bu\rarrow Q^\bu$
with the same property.
 Furthermore, by~\cite[Theorem~5.3]{BCE}, the complex
$\Hom_R(F^\bu,C^\bu)$ is acyclic for all complexes
$F^\bu\in\Ac(R\modl_\fl)$ and $C^\bu\in\Com(R\modl^\cot)$.
 Hence the complex $\Hom_R(P^\bu,C^\bu)$ is acyclic if and only if
the complex $\Hom_R(Q^\bu,C^\bu)$ is.
 This proves the first assertion of the theorem.

 To prove the second assertion, notice that the class of all complexes
of flat modules $\Com(R\modl_\fl)$ is deconstructible in the abelian
category of all complexes of modules $\Com(R\modl)$
by~\cite[Proposition~4.3]{Sto0}.
 There are enough projective objects in the Grothendieck abelian
category $\Com(R\modl)$, and they belong to $\Com(R\modl_\fl)$.
 By a suitable version of the Eklof--Trlifaj theorem (cf.\
Theorem~\ref{eklof-trlifaj}), there exists a complete cotorsion pair
$(\sF,\sC)$ in $\Com(R\modl)$ with $\sF=\Com(R\modl_\fl)$.

 In particular, all the contractible two-term complexes of flat
modules $\dotsb\rarrow0\rarrow F\overset{\id}\rarrow F\rarrow0\
\rarrow\dotsb$, \,$F\in R\modl_\fl$, belong to $\sF$, and it follows
that all the objects of $\sC$ are complexes of cotorsion $R$\+modules
(cf.~\cite[fourth paragraph of the proof of Theorem~6.6]{PS4}).
 Furthermore, all the objects $C^\bu\in\sC$ are Becker-contraacyclic
as complexes in $R\modl^\cot$, since the complex $\Hom_R(G^\bu,C^\bu)$
is acyclic for all $G^\bu\in\Com(R\modl_\fl)$, and in particular for
all $G^\bu\in\Com(R\modl_\fl^\cot)$, by~\cite[Lemma~5.1]{PS4}.

 Now let $D^\bu\in\Com(R\modl^\cot)$ be a complex of cotorsion
$R$\+modules, and let $0\rarrow C^\bu\rarrow G^\bu\rarrow D^\bu
\rarrow0$ be a special precover sequence in $\Com(R\modl)$ with
$C^\bu\in\sC$ and $G^\bu\in\sF$.
 Then $C^\bu$ is a complex of cotorsion $R$\+modules, hence
$G^\bu$ is a complex of flat cotorsion $R$\+modules.
 Finally, we have $C^\bu\in\Ac^\bctr(R\modl^\cot)$, while
the totalization of the short exact sequence $0\rarrow C^\bu\rarrow
G^\bu\rarrow D^\bu\rarrow0$ belongs to $\Ac^\abs(R\modl^\cot)
\subset\Ac^\bctr(R\modl^\cot)$.
 Hence the cone of the morphism $G^\bu\rarrow D^\bu$ belongs to
$\Ac^\bctr(R\modl^\cot)$.
 This proves the second assertion of the theorem.
\end{proof}

\begin{lem} \label{homotopy-of-flat-cotorsion-as-contraderived-of-all}
 For any ring $R$, the composition of triangulated functors
$\Hot(R\modl_\fl^\cot)\allowbreak\rarrow\Hot(R\modl)\rarrow
\sD^\bctr(R\modl)$ is a triangulated equivalence \hfuzz=6pt
$$
 \Hot(R\modl_\fl^\cot)\simeq\sD^\bctr(R\modl).
$$
\end{lem}

\begin{proof}
 The composition of triangulated functors $\Hot(R\modl_\proj)
\rarrow\Hot(R\modl)\allowbreak\rarrow\sD^\bctr(R\modl)$ is
a triangulated equivalence
\begin{equation} \label{contraderived-of-all-modules}
 \Hot(R\modl_\proj)\simeq\sD^\bctr(R\modl)
\end{equation}
by~\cite[Corollary~5.10]{Neem}, \cite[Proposition~1.3.6(1)]{Bec},
or~\cite[Corollary~7.4]{PS4}.
 Comparing the results of Theorems~\ref{contraderived-category-of-flats}
and~\ref{coderived-category-of-flats}, one comes to a triangulated
equivalence
\begin{equation} \label{co-contra-derived-of-flats}
 \Hot(R\modl_\proj)\simeq\Hot(R\modl_\fl^\cot).
\end{equation}
 Finally, the two equivalences~\eqref{contraderived-of-all-modules}
and~\eqref{co-contra-derived-of-flats} form a commutative diagram
with the desired triangulated functor $\Hot(R\modl_\fl^\cot)\rarrow
\sD^\bctr(R\modl)$, because all acyclic complexes in the exact
category $R\modl_\fl$ are Becker-contraacyclic in $R\modl_\fl$, hence
also in $R\modl$, by Theorem~\ref{contraderived-category-of-flats}.
\end{proof}

\begin{cor} \label{cotorsion-in-all-coresolving-contraderived-cor}
 For any ring $R$, the inclusion of exact categories
$R\modl^\cot\rarrow R\modl$ induces a well-defined triangulated
functor between the Verdier quotient categories\/
$\sD^\bctr(R\modl^\cot)\rarrow\sD^\bctr(R\modl)$.
 The resulting functor is a triangulated equivalence of
Becker-contraderived categories
$$
 \sD^\bctr(R\modl^\cot)\simeq\sD^\bctr(R\modl).
$$
\end{cor}

\begin{proof}
 The first assertion holds, because the inclusion
$\Ac^\bctr(R\modl^\cot)\subset\Ac^\bctr\allowbreak(R\modl)$
follows from the first assertion of
Theorem~\ref{cotorsion-in-all-coresolving-contraderived-thm}.
 To deduce the second assertion, compare the second assertion
of Theorem~\ref{cotorsion-in-all-coresolving-contraderived-thm}
with Lemma~\ref{homotopy-of-flat-cotorsion-as-contraderived-of-all}.
\end{proof}

\subsection{Periodicity theorems} \label{periodicity-subsecn}
 The most nontrivial aspects of the results of
Sections~\ref{derived-of-flats-and-projectives-subsecn}\+-%
\ref{contraderived-of-cotorsion-subsecn} are based on the class of
phenomena called the \emph{periodicity theorems}, or more specifically
the \emph{flat and projective periodicity} and the \emph{cotorsion
periodicity}.

 Let $L$ be an $R$\+module.
 An $R$\+module $M$ is said to be \emph{$L$\+periodic} if there is
a short exact sequence of $R$\+modules $0\rarrow M\rarrow L\rarrow M
\rarrow0$.
 Periodicity theorems are results claiming that if a module $L$
belongs to a certain class, a module $M$ belongs to a certain class,
and the exact sequence $0\rarrow M\rarrow L\rarrow M\rarrow0$ belongs
to a certain class of exact sequences, then the module $M$ actually
has to belong to a certain (more narrow) class of modules.

 The \emph{flat and projective periodicity theorem} claims that if $L$
is a projective $R$\+module and $M$ is a flat $R$\+module, then $M$
is actually a projective $R$\+module~\cite[Theorem~2.5]{BG}.
 The \emph{cotorsion periodicity theorem} claims that if $L$ is
a cotorsion $R$\+module, then $M$ is also a cotorsion
$R$\+module~\cite[Theorem~1.2(2) or Proposition~4.8(2)]{BCE}.

 Periodicity theorems are closely related to results about modules of
cocycles in acyclic complexes.
 A general discussion of this connection can be found
in~\cite[proof of Proposition~7.6]{CH},
\cite[Propositions~1 and~2]{EFI}, and~\cite[Proposition~2.4]{BCE};
see~\cite[Proposition~1.1]{Pgen} and~\cite[Proposition~8.4]{PS6}
for generalizations.
 In particular, in any acyclic complex of projective modules with
flat modules of cocycles, the modules of cocycles are actually
projective~\cite[Remark~2.15]{Neem}, while in any acyclic complex
of cotorsion modules, the modules of cocycles are also
cotorsion~\cite[Theorem~5.1(2)]{BCE}.

 A general discussion of known periodicity theorems with a long list
of such theorems and connections between them can be found in
the introduction to the paper~\cite{BHP}.
 In particular, beyond the flat/projective periodicity theorem and
cotorsion periodicity theorem, there is also
the fp\+injective/injective periodicity theorem~\cite[Theorem~5.4
and Corollary~5.5]{Sto} and the fp\+projective periodicity
theorem~\cite[Example~4.3]{SarSt}, \cite[Theorems~0.7\+-0.8
or~4.1\+-4.2]{BHP}.
 Subsequent generalizations can be found in
the preprint~\cite[Theorems~0, A, and~B]{Pgen}.

\subsection{Flat things as direct limits}
\label{flats-as-direct-limits-subsecn}
 The classical \emph{Govorov--Lazard theorem} tells that flat
$R$\+modules are precisely the direct limits of projective $R$\+modules,
or equivalently, of finitely generated free $R$\+modules~\cite{Gov,Laz},
\cite[Corollary~2.22]{GT}.
 This result can be generalized very widely if one agrees to replace
``finitely generated free/projective modules'' by ``countably presented
flat modules''.
 As stated in the abstract to the paper~\cite{Pres}, a general
principle seems to be that ``anything flat is a direct limit of
countably presented flats''.

 In particular, all at most countable diagrams of flat modules, such
as, e.~g., complexes of flat modules, are direct limits of
diagrams/complexes of countably presented flat
modules~\cite[Corollaries~10.3\+-10.4]{Pacc}.
 All acyclic complexes of flat modules with flat modules of
cocycles are direct limits of such complexes of countably presented
flat modules~\cite[Corollary~10.14]{Pacc}.
 All graded-flat (curved) DG\+modules are direct limits of countably
presented graded-flat (curved) DG\+modules~\cite[Proposition~4.7]{PS7}.

 Over a countably coherent ring (e.~g., over a coherent ring or
a countably Noetherian ring), all acyclic complexes of flat modules
(with arbitrary modules of cocycles) are direct limits of such
complexes of countably presented flat modules~\cite[Theorem~4.2]{Pres}.
 Furthermore, for any given integer $m\ge0$, all modules of flat
dimension~$m$ over such a ring are direct limits of countably presented
modules of flat dimension at most~$m$ \,\cite[Corollary~5.2]{Pres}.

 All flat quasi-coherent sheaves on countably quasi-compact,
countably quasi-separated schemes (e.~g., on the usual quasi-compact
quasi-separated schemes) are direct limits of locally countably
presented flat quasi-coherent sheaves~\cite[Theorems~2.4 and~3.5]{PS6}.
 All complexes of flat quasi-coherent sheaves on such schemes are
direct limits of complexes of locally countably presented flat
quasi-coherent sheaves~\cite[Theorem~4.1]{PS6}.
 All acyclic complexes of flat quasi-coherent sheaves with flat
sheaves of cocycles on such schemes are direct limits of complexes
of locally countably presented flat quasi-coherent sheaves from
the same class~\cite[Theorem~4.2]{PS6}.
 The proofs of these assertions are based on category-theoretic
results of~\cite[Proposition~3.1]{CR},
\cite[Pseudopullback Theorem~2.2]{RR}, and the paper~\cite{Pacc}
(going back to~\cite[Theorem~3.8, Corollary~3.9, and
Remark~3.11(II)]{Ulm}).

 All homotopy flat complexes of flat modules are direct limits of
bounded complexes of finitely generated projective
modules~\cite[Theorem~1.1]{CH}.
 It follows that all homotopy flat complexes of flat quasi-coherent
sheaves on countably quasi-compact, semi-separated schemes (e.~g.,
on the usual quasi-compact semi-separated schemes) are direct limits
of homotopy flat complexes of locally countably presented flat
quasi-coherent sheaves~\cite[Theorem~4.5]{PS6}.

 Furthermore, all flat quasi-coherent sheaves over certain stacks
are direct limits of locally countably presented
ones~\cite[Theorem~3.1]{Pflcc}, and so are flat pro-quasi-coherent
pro-sheaves over certain ind-schemes~\cite[Theorem~10.1]{Pflcc}.
 All acyclic complex of flat quasi-coherent sheaves with flat sheaves
of cocycles over such stacks are direct limits of complexes of
locally countably presented flat quasi-coherent sheaves with flat
sheaves of cocycles, and the same applies to such
ind-schemes~\cite[Corollaries~4.5 and~11.4]{Pflcc}.

 Given a commutative ring $R$, all (say, coassociative and counital;
or coassociative, cocommutative, and counital)  coalgebras over $R$
with flat underlying $R$\+modules are direct limits  of coalgebras
from the same class with flat and countably presented underlying
$R$\+modules~\cite[Theorem~3.1 and Remark~3.2]{Pcor}.

 The importance of countably presented flat modules lies in
the fact that they have projective dimensions at most~$1$
\,\cite[Corollary~2.23]{GT}, \cite[Corollary~2.4]{Pres}.
 The proof of \emph{quasi-coherent cotorsion periodicity}
in~\cite[Theorem~9.2]{PS6} is based on this fact.
 Similarly one obtains cotorsion periodicity theorems for
quasi-coherent sheaves over stacks and pro-quasi-coherent pro-sheaves
over ind-schemes~\cite[Theorems~5.4 and~12.3]{Pflcc}, as well as
certain versions of the ``derived $=$ coderived category''  theorem
for flat quasi-coherent sheaves and flat pro-quasi-coherent
pro-sheaves~\cite[Theorems~6.5 and~13.2]{Pflcc}.

 \emph{Flat contramodules over topological rings} is another name for
flat pro-quasi-coherent pro-sheaves over ind-affine ind-schemes
(under suitable assumptions; see~\cite[Example~3.8(2)]{Psemten}).
 The exposition in~\cite{Pflcc} is written in the language of comodules
and contramodules.
 In fact, the comparison between the quasi-coherent sheaves over
stacks and the comodules over corings is relatively unproblematic;
but (nonflat) pro-quasi-coherent pro-sheaves form a badly behaved
full subcategory in the abelian category of
contramodules~\cite[Examples~3.1 and~3.8(1)]{Psemten}.
 So the cotorsion periodicity theorem for ind-affine
ind-schemes~\cite[Theorem~12.3]{Pflcc} should be stated in the language
of contramodules rather than pro-quasi-coherent pro-sheaves
(and it is done this way in~\cite{Pflcc}).
 We do \emph{not} know what, if anything, should a quasi-coherent
cotorsion periodicity theorem for non-ind-affine ind-schemes say.
{\hbadness=1400\par}

\subsection{Periodicity theorems~II} \label{periodicity-II-subsecn}
 The aim of this section is to discuss the nature of periodicity
theorems as it appears in our present understanding.

 It appears that both the flat/projective and cotorsion periodicity
theorems stem from the fact that all flat modules are direct limits
of (finitely generated) projective ones.
 Similarly, both the fp\+projective and fp\+injective/injective
periodicity theorems stem from the fact that all modules are direct
limits of finitely presented ones.

 Concerning the cotorsion and fp\+injective/injective periodicities,
this connection is expressed in the form of a general category-theoretic
theorem in~\cite[Theorem~8.1 and Corollary~8.2]{PS6}.
 A common generalization of the cotorsion and fp\+injective/injective
periodicity theorems based on these results from~\cite{PS6} is
worked out in~\cite[Theorem~0(b) in Section~0.1 and Theorem~B in
Section~0.3]{Pgen}.
 The quasi-coherent cotorsion periodicity on quasi-compact
semi-separated schemes~\cite[Theorem~9.2]{PS6}, as well as
generalizations to stacks and ind-schemes~\cite[Theorems~5.4
and~12.3]{Pflcc}, are corollaries of~\cite[Corollary~8.2]{PS6} (cf.\
the last paragraph of Section~\ref{flats-as-direct-limits-subsecn}).

 Concerning the flat/projective and fp\+projective periodicities,
a common generalization of both, expressing their interpretation as
properties of the direct limit closure, can be found
in~\cite[Theorem~0(a) in Section~0.1 and Theorem~A in
Section~0.2]{Pgen}.

 A general concept of the direct limit completion of an exact category
is worked out in the preprint~\cite{Plce}.
 Some related periodicity theorems can be found in~\cite[Theorems~7.1,
7.6, and~8.3]{Plce}.

\subsection{Becker-contraderived category of very flat modules}
\label{contraderived-of-very-flats-subsecn}
 We start with a very flat version of
Theorem~\ref{projectives-in-flats-resolving-derived} before
passing to a very flat version of
Theorem~\ref{contraderived-category-of-flats}.

 For any commutative ring $R$, the full subcategory of projective
$R$\+modules $\sF=R\modl_\proj$ is a resolving subcategory in
the exact category of very flat $R$\+modules $\sE=R\modl_\vfl$.
 Moreover, the projective objects of the exact category $R\modl_\vfl$
are precisely all the projective $R$\+modules, $(R\modl_\vfl)_\proj=
R\modl$, and there are enough such projective objects.

\begin{cor} \label{projectives-in-very-flats-resolving-derived}
 For any commutative ring $R$ and any conventional derived category
symbol\/ $\star=\bb$, $+$, $-$, or\/~$\varnothing$, the inclusion of
exact categories $R\modl_\proj\rarrow R\modl_\vfl$ induces
an equivalence of their derived categories
$$
 \Hot^\star(R\modl_\proj)=\sD^\star(R\modl_\proj)\simeq
 \sD^\star(R\modl_\vfl).
$$
\end{cor}

\begin{proof}
 Projective dimensions of very flat modules do not exceed~$1$ by
Theorem~\ref{very-flat-cotorsion-pair}(d).
 Hence the assertion of the theorem is a particular case of
Proposition~\ref{finite-co-resol-dim-conventional-derived}(b).
\end{proof}

\begin{cor} \label{contraderived-category-of-very-flats}
 For any commutative ring $R$, the classes of all absolutely acyclic
complexes, Becker-contraacyclic complexes, and acyclic complexes in
the exact category $\sB=R\modl_\vfl$ coincide.
 The composition of triangulated functors\/ $\Hot(R\modl_\proj)\rarrow
\Hot(R\modl_\vfl)\rarrow\sD^\abs(R\modl_\vfl)=\sD^\bctr(R\modl_\vfl)=
\sD(R\modl_\vfl)$ is a triangulated equivalence \hbadness=1300
$$
 \Hot(R\modl_\proj)\simeq\sD^\abs(R\modl_\vfl)=
 \sD^\bctr(R\modl_\vfl)=\sD(R\modl_\vfl).
$$
\end{cor}

\begin{proof}
 The homological dimension of the exact category $R\modl_\vfl$
does not exceed~$1$, so
Proposition~\ref{enough-inj-proj-finite-homol-dim-prop}(b)
is applicable.
\end{proof}

\subsection{Derived categories of contraadjusted and all modules}
\label{derived-of-contraadjusted-and-all-subsecn}
 Let $R$ be a commutative ring.
 Then the full subcategory of contraadjusted $R$\+modules
$\sC=R\modl^\cta$ is a coresolving subcategory in the abelian
category of $R$\+modules $\sE=R\modl$.

\begin{cor} \label{contraadjusted-in-all-coresolving-derived}
 For any commutative ring $R$ and any conventional or absolute
derived category symbol\/ $\star=\bb$, $+$, $-$, $\varnothing$,
or\/~$\abs$, the inclusion of exact categories $R\modl^\cta\rarrow
R\modl$ induces an equivalence of their derived categories
$$
 \sD^\star(R\modl^\cta)\simeq\sD^\star(R\modl).
$$
\end{cor}

\begin{proof}
 Contraadjusted dimensions of $R$\+modules do not exceed~$1$ by
Theorem~\ref{very-flat-cotorsion-pair}(c).
 Hence the assertion of the theorem is a particular case of
Proposition~\ref{finite-co-resol-dim-conventional-derived}(a).
\end{proof}

\subsection{Coderived category of very flat modules}
\label{coderived-of-very-flats-subsecn}
 Let $R$ be a commutative ring.
 An $R$\+module is said to be \emph{very flat contraadjusted} if it is
simultaneously very flat and contraadjusted.
 The full subcategory of very flat contraadjusted $R$\+modules is
denoted by $R\modl_\vfl^\cta=R\modl_\vfl\cap R\modl^\cta\subset R\modl$.

 The very flat contraadjusted $R$\+modules are precisely all
the injective objects of the exact category of very flat $R$\+modules
$R\modl_\vfl$, and there are enough such injective objects.
 The very flat contraadjusted $R$\+modules are also precisely all
the the projective objects of the exact category of contraadjusted
$R$\+modules $R\modl^\cta$, and there are enough such projective
objects.
 These assertions follow from the fact that the very flat cotorsion
pair in $R\modl$ is complete (by
Theorem~\ref{very-flat-cotorsion-pair}(a));
see~\cite[Lemma~2.1]{Pgen}.
 
\begin{cor} \label{coderived-category-of-very-flats}
 For any commutative ring $R$, the classes of all absolutely acyclic
complexes, Positselski-coacyclic complexes, Becker-coacyclic complexes,
and acyclic complexes in the exact category\/ $\sA=R\modl_\vfl$
coincide.
 The composition of triangulated functors\/ $\Hot(R\modl_\vfl^\cta)
\rarrow\Hot(R\modl_\vfl)\rarrow\sD^\abs(R\modl_\vfl)=
\sD^\pco(R\modl_\vfl)=\sD^\bco(R\modl_\vfl)=\sD(R\modl_\vfl)$
is a triangulated equivalence
\begin{multline*}
 \Hot(R\modl_\vfl^\cta) \\ \simeq\sD^\abs(R\modl_\vfl)=
 \sD^\pco(R\modl_\vfl)=\sD^\bco(R\modl_\vfl)=\sD(R\modl_\vfl).
\end{multline*}
\end{cor}

\begin{proof}
 The homological dimension of the exact category $R\modl_\vfl$
does not exceed~$1$, so
Proposition~\ref{enough-inj-proj-finite-homol-dim-prop}(a)
is applicable and proves all the assertions except the one
concerning Positselski-coacyclic complexes.
 To deduce the latter one, observe that the inclusions
$\Ac^\abs(\sE)\subset\Ac^\pco(\sE)\subset\Ac^\bco(\sE)$ hold for
any exact category $\sE$ by the definition and by
Lemma~\ref{Positselski-trivial-implies-Becker-trivial}(a).
\end{proof}

\subsection{Becker-contraderived categories of contraadjusted and
all modules} \label{contraderived-of-contraadjusted-and-all-subsecn}
 We start with the result about Positselski-contraderived
categories before passing to Becker-contraderived ones.

\begin{cor} \label{contraadjusted-in-all-coresolving-pos-contrader-cor}
 For any commutative ring $R$, the class of all
Positselski-contraacyclic complexes in the exact category
$R\modl^\cta$ coincides with the class of all complexes of
contraadjusted $R$\+modules that are Positselski-contraacyclic
as complexes in the abelian category $R\modl$, \hbadness=1500
$$
 \Ac^\pctr(R\modl^\cta)=\Com(R\modl^\cta)\cap\Ac^\pctr(R\modl).
$$
 The inclusion of exact categories $R\modl^\cta\rarrow R\modl$
induces a triangulated equivalence of the Positselski-contraderived
categories
$$
 \sD^\pctr(R\modl^\cta)\simeq\sD^\pctr(R\modl).
$$
\end{cor}

\begin{proof}
 The inclusion $\Ac^\pctr(R\modl^\cta)\subset
\Com(R\modl^\cta)\cap\Ac^\pctr(R\modl)$ holds by the definition,
and it follows that the triangulated functor between
the Verdier quotient categories $\sD^\pctr(R\modl^\cta)\rarrow
\sD^\pctr(R\modl)$ induced by the inclusion of exact categories
$R\modl^\cta\rarrow R\modl$ is well-defined.
 Since the contraadjusted dimensions of $R$\+modules never exceed~$1$,
\,Proposition~\ref{finite-co-resol-dim-Positselski-second-kind}(a)
tells that this triangulated functor is an equivalence, proving
the second assertion of the theorem.
 Then the first assertion of the theorem follows from the second one.
\end{proof}

\begin{lem} \label{dw-contraadjusted-are-dg-contraadjusted}
 Let $R$ be a commutative ring.
 Then, for any acyclic complex of very flat $R$\+modules $F^\bu$ with
very flat $R$\+modules of cocycles, and any complex of contraadjusted
$R$\+modules $C^\bu$, the complex of morphisms $\Hom_R(F^\bu,C^\bu)$
is acyclic.
\end{lem}

\begin{proof}
 Choose a homotopy injective complex of injective $R$\+modules $J^\bu$
together with a quasi-isomorphism of complexes of $R$\+modules
$C^\bu\rarrow J^\bu$.
 Then the complex $\Hom_R(F^\bu,J^\bu)$ is acyclic, since $F^\bu$ is
an acyclic complex of $R$\+modules.
 Denote the cone of the morphism of complexes $C^\bu\rarrow J^\bu$
by~$D^\bu$.
 Then $D^\bu$ is an acyclic complex of contraadjusted $R$\+modules.
 By Theorem~\ref{very-flat-cotorsion-pair}(c), the $R$\+modules of
cocycles of the complex $D^\bu$ are contraadjusted.
 Now \cite[Lemma~3.9]{Gil} tells that the complex $\Hom_R(F^\bu,D^\bu)$
is acyclic.
 Thus the complex $\Hom_R(F^\bu,C^\bu)$ is acyclic as well.
\end{proof}

 We recall from Section~\ref{coderived-of-very-flats-subsecn}
that the very flat contraadjusted $R$\+modules are the projective
objects of the exact category of contraadjusted $R$\+modules
$R\modl^\cta$.

\begin{thm} \label{contraadjusted-in-all-coresolving-beck-contrader-thm}
 For any commutative ring $R$, the class of all Becker-contraacyclic
complexes in the exact category\/ $R\modl^\cta$ coincides with
the class of all complexes of contraadjusted $R$\+modules that are
Becker-contraacyclic as complexes in the abelian category $R\modl$,
$$
 \Ac^\bctr(R\modl^\cta)=\Com(R\modl^\cta)\cap\Ac^\bctr(R\modl).
$$
 The composition of triangulated functors\/ $\Hot(R\modl_\vfl^\cta)
\rarrow\Hot(R\modl^\cta)\rarrow\sD^\bctr(R\modl^\cta)$ is
a triangulated equivalence
$$
 \Hot(R\modl_\vfl^\cta)\simeq\sD^\bctr(R\modl^\cta).
$$
\end{thm}

\begin{proof}
 The argument is similar to the proof of
Theorem~\ref{cotorsion-in-all-coresolving-contraderived-thm},
using Corollaries~\ref{contraderived-category-of-very-flats}
and~\ref{coderived-category-of-very-flats} instead of
Theorems~\ref{contraderived-category-of-flats}
and~\ref{coderived-category-of-flats}, respectively.
 Instead of~\cite[Theorem~5.3]{BCE}, use
Lemma~\ref{dw-contraadjusted-are-dg-contraadjusted}.
\end{proof}

\begin{lem} \label{homotopy-of-vfl-cta-as-contraderived-of-all}
 For any commutative ring $R$, the composition of triangulated functors
$\Hot(R\modl_\vfl^\cta)\rarrow\Hot(R\modl)\rarrow
\sD^\bctr(R\modl)$ is a triangulated equivalence
$$
 \Hot(R\modl_\vfl^\cta)\simeq\sD^\bctr(R\modl).
$$
\end{lem}

\begin{proof}
 Similar to the proof of
Lemma~\ref{homotopy-of-flat-cotorsion-as-contraderived-of-all}.
 Once again, Corollaries~\ref{contraderived-category-of-very-flats}
and~\ref{coderived-category-of-very-flats} need to be used instead of
Theorems~\ref{contraderived-category-of-flats}
and~\ref{coderived-category-of-flats}.
\end{proof}

\begin{cor} \label{contraadjusted-in-all-coresolving-beck-contrader-cor}
 For any commutative ring $R$, the inclusion of exact categories
$R\modl^\cta\rarrow R\modl$ induces a well-defined triangulated
functor between the Verdier quotient categories\/
$\sD^\bctr(R\modl^\cta)\rarrow\sD^\bctr(R\modl)$.
 The resulting functor is a triangulated equivalence of
Becker-contraderived categories
$$
 \sD^\bctr(R\modl^\cta)\simeq\sD^\bctr(R\modl).
$$
\end{cor}

\begin{proof}
 Similar to the proof of
Corollary~\ref{cotorsion-in-all-coresolving-contraderived-cor}.
 Theorem~\ref{contraadjusted-in-all-coresolving-beck-contrader-thm}
and Lemma~\ref{homotopy-of-vfl-cta-as-contraderived-of-all}
need to be used instead of
Theorem~\ref{cotorsion-in-all-coresolving-contraderived-thm}
and Lemma~\ref{homotopy-of-flat-cotorsion-as-contraderived-of-all}.
\end{proof}

\subsection{Summary~I}
 Let us collect here explicit answers to the questions posed in
the second introductory paragraph of
Section~\ref{flat-modules-secn}.
 Let $U$ be an affine scheme and $R=\cO(U)$ a commutative ring.

 The inclusion of exact categories $U\ctrh=R\modl^\cta\rarrow R\modl$
induces equivalences of conventional derived categories
$$
 \sD^\star(R\modl^\cta)\simeq\sD^\star(R\modl)
$$
for all symbols $\star=\bb$, $+$, $-$, and~$\varnothing$
by Corollary~\ref{contraadjusted-in-all-coresolving-derived}.
 The same holds for the absolute derived categories, $\star=\abs$.

 The same inclusion of exact categories also induces equivalences
of the contraderived categories in the sense of Positselski and Becker,
\begin{align*}
 \sD^\pctr(R\modl^\cta) &\simeq\sD^\pctr(R\modl), \\
 \sD^\bctr(R\modl^\cta) &\simeq\sD^\bctr(R\modl)
\end{align*}
by Corollaries~\ref{contraadjusted-in-all-coresolving-pos-contrader-cor}
and~\ref{contraadjusted-in-all-coresolving-beck-contrader-cor} (see also
Theorem~\ref{contraadjusted-in-all-coresolving-beck-contrader-thm}).

 The inclusion of exact categories $U\ctrh^\lct=R\modl^\cot\rarrow
R\modl$ induces equivalences of conventional derived categories
$$
 \sD^\star(R\modl^\cot)\simeq\sD^\star(R\modl)
$$
for the symbols $\star=+$, and much more importantly and nontrivially,
$\star=\varnothing$ (i.~e., for the unbounded derived categories).
 The assertion for $\star=+$ is a particular case of the very general
Proposition~\ref{one-sided-bounded-resolving-subcategory-prop}(a).
 The assertion for $\star=\varnothing$ is
Theorem~\ref{cotorsion-in-all-coresolving-derived}.

 The same inclusion of exact categories also induces an equivalence of
the Becker-contraderived categories
$$
 \sD^\bctr(R\modl^\cot)\simeq\sD^\bctr(R\modl)
$$
by Corollary~\ref{cotorsion-in-all-coresolving-contraderived-cor}
(see also Theorem~\ref{cotorsion-in-all-coresolving-contraderived-thm}).

 The inclusion of exact categories $U\ctrh^\lin=R\modl^\inj\rarrow
R\modl$ induces equivalences of conventional derived categories
$$
 \Hot^+(R\modl^\inj)=\sD^+(R\modl^\inj)\simeq\sD^+(R\modl).
$$
 The triangulated functors
$$
 \Hot^\star(R\modl^\inj)=\sD^\star(R\modl^\inj)
 \lrarrow\sD^\star(R\modl)
$$
induced by the same inclusion of exact categories are usually
\emph{not} equivalences for $\star=\bb$ or $\star=\varnothing$
(see Example~\ref{thematic-counterex}).

 For $\star=-$, this triangulated functor is also \emph{not}
an equivalence for associative rings $R$ in
general~\cite[Proposition~1.2 together with Corollary~2.3, 4.1, 7.2,
or~7.3, or together with Example~8.4]{Pab}.
 It would be interesting to have a counterexample for a commutative
ring~$R$.
 
 The triangulated functors
\begin{align*}
 \Hot(R\modl^\inj)=\sD^\pctr(R\modl^\inj) &\lrarrow\sD^\pctr(R\modl), \\
 \Hot(R\modl^\inj)=\sD^\bctr(R\modl^\inj) &\lrarrow\sD^\bctr(R\modl)
\end{align*}
are also \emph{not} expected to be equivalences in general.
 The same counterexamples from the paper~\cite{Pab} work, providing
complexes of injective modules over certain associative rings that
are bounded above and acyclic, hence contraacyclic in the sense of
Positselski~\cite[Lemma~4.1]{Psemi}, \cite[Lemma~5.3(b)]{PSch}
(hence also contraacyclic in the sense of Becker) in the abelian
category $R\modl$, but not contractible.

 Let us offer some comment or explanation.
 Morally, the locally contraadjusted and locally cotorsion contraherent
cosheaves are two contraherent versions or ``dual analogues'' of 
\emph{arbitrary} quasi-coherent sheaves; while the locally injective
contraherent cosheaves are the dual analogues of \emph{flat}
quasi-coherent sheaves.
 So the unbounded derived categories $\sD(X\ctrh)$ and
$\sD(X\ctrh^\lct)$ \emph{are} expected to be directly comparable to
the derived category of quasi-coherent sheaves $\sD(X\qcoh)$ for
a good enough scheme $X$; while the derived category $\sD(X\ctrh^\lin)$
is \emph{not} expected to be compared to $\sD(X\qcoh)$.
 See~\cite[Theorem~4.6.6]{Pcosh}, cf.~\cite[Theorem~5.7.1]{Pcosh}.

\subsection{Summary~II}
 The thesis \emph{flat modules are not that much different from
projective modules} in the title of Section~\ref{flat-modules-secn}
is explained/confirmed by
Theorems~\ref{projectives-in-flats-resolving-derived},
\ref{contraderived-category-of-flats},
and~\ref{coderived-category-of-flats}.
 To elaborate,
\begin{itemize}
\item the inclusion of exact categories $R\modl_\proj\rarrow
R\modl_\fl$ induces an equivalence of the conventional unbounded
derived categories
$$
 \Hot(R\modl_\proj)=\sD(R\modl_\proj)\simeq\sD(R\modl_\fl);
$$
\item the inclusion of exact categories $R\modl_\proj\rarrow
R\modl_\fl$ induces an equivalence of the Becker-coderived
categories
$$
 \Hot(R\modl_\proj)=\sD^\bco(R\modl_\proj)\simeq\sD^\bco(R\modl_\fl);
$$
\item the classes of acyclic, Becker-coacyclic, and Becker-contraacyclic
complexes in the exact category $R\modl_\fl$ coincide,
$$
 \Ac^\bco(R\modl_\fl)=\Ac(R\modl_\fl)=\Ac^\bctr(R\modl_\fl).
$$
\end{itemize}

 These are instances of behaviour typical for resolving subcategories
of finite resolution dimension or exact categories of finite
homological dimension.
 Cf.\ Corollaries~\ref{projectives-in-very-flats-resolving-derived},
\ref{contraderived-category-of-very-flats},
and~\ref{coderived-category-of-very-flats}, which are based on
Propositions~\ref{enough-inj-proj-finite-homol-dim-prop}
and~\ref{finite-co-resol-dim-conventional-derived}(b).

 The same thesis \emph{flat modules are not that much different from
projective modules} is also explained/confirmed
by Theorem~\ref{cotorsion-in-all-coresolving-derived}
and Corollary~\ref{cotorsion-in-all-coresolving-contraderived-cor}.
 To repeat,
\begin{itemize}
\item the inclusion of exact categories $R\modl^\cot\rarrow
R\modl$ induces an equivalence of the conventional unbounded derived
categories
$$
 \sD(R\modl^\cot)\simeq\sD(R\modl);
$$
\item the inclusion of exact categories $R\modl^\cot\rarrow
R\modl$ induces an equivalence of the Becker-contraderived categories
$$
 \sD^\bctr(R\modl^\cot)\simeq\sD^\bctr(R\modl).
$$
\end{itemize}

 These are instances of behaviour typical for coresolving subcategories
of finite coresolution dimension.
 Cf.\ Corollaries~\ref{contraadjusted-in-all-coresolving-derived}
and~\ref{contraadjusted-in-all-coresolving-beck-contrader-cor},
which are based on
Proposition~\ref{enough-inj-proj-finite-homol-dim-prop},
Proposition~\ref{finite-co-resol-dim-conventional-derived}(a),
and Theorem~\ref{very-flat-cotorsion-pair}(c).
 See Corollary~\ref{projdims-of-flats-and-cotorsion-dimensions} for
the connection between the projective dimensions of flat modules and
the cotorsion dimensions of arbitrary modules.

\subsection{Conclusion}
 For the purposes of the derived, coderived, and contraderived
categories, finite homological dimension is ``harmless''.
 This statement actually refers two two assertions: firstly,
the passage to an exact subcategory of finite (co)resolution
dimension does not change the derived/coderived/contraderived
category; and secondly, for an exact category of finite homological
dimension, the derived and coderived/contraderived category agree
whenever they are well-defined.

 The discussion in this section illustrates the thesis that
the passage to the direct limit closure of an exact (sub)category is
almost as ``harmless'' as finite homological dimension
(see Section~\ref{periodicity-II-subsecn}).
 Thus, even though projective dimensions of flat modules may be
infinite, flat modules are ``not that much different'' from
projective ones.

 The typical covers in algebraic geometry are related to commutative
ring homomorphisms $R\rarrow S$ making $S$ a flat, but not a projective
$R$\+module.
 The main technical problem of the theory of contraherent cosheaves
consists in this fact, and various specific technical problems tend
to arise from this main one.
 The discussion in this section suggests that such problems are
surmountable.

\Section{Contraherent Proof of Quasi-Coherent Cotorsion Periodicity}
\label{cotorsion-periodicity-secn}

 The proof of the quasi-coherent cotorsion periodicity
theorem~\cite[Theorem~9.2]{PS6} given in the paper~\cite{PS6}
(mentioned in Sections~\ref{flats-as-direct-limits-subsecn}\+-%
\ref{periodicity-II-subsecn} above) is essentially global, based
as it is on applying the category-theoretic
result of~\cite[Corollary~8.2]{PS6} to the whole category of
quasi-coherent sheaves $X\qcoh$.
 The aim of this section is to suggest a \emph{local} proof of
quasi-coherent cotorsion periodicity, using reduction to the affine
case established in~\cite[Theorem~1.2(2), Proposition~4.8(2), or
Theorem~5.1(2)]{BCE}.
 The concept of a contraherent cosheaf is used in order to
accomplish such a reduction.

\subsection{Preliminary discussion}
 The interaction of local and global phenomena is a fundamental
aspect of geometry.
 Many properties of modules over commutative rings are \emph{local}
in the sense that they can be checked on an affine open
covering $U=\bigcup_{\alpha=1}^dU_\alpha$ of an affine scheme $U$
using the localization functors to pass from $\cO(U)$\+modules
to modules over the rings $\cO(U_\alpha)$.
 For example, the projectivity~\cite[\S\,II.3.1]{RG}, \cite{Pe},
flatness, and very flatness~\cite[Lemma~1.2.6(a)]{Pcosh},
\cite[Example~2.5]{Pal} are local properties of modules over
commutative rings.

 The property of a module over commutative ring to be cotorsion is
\emph{not} local, however~\cite[Example~2.7]{Pal}.
 This fact stands in the way of a direct attempt to reduce
the cotorsion periodicity theorem for quasi-coherent sheaves to
the similar theorem for modules over commutative rings.

 As pointed out in~\cite[Example~3.8]{Pal}, the class of cotorsion
modules is \emph{colocal} (if the contraadjustedness is presumed).
 Contraherent cosheaves are defined by gluing modules over commutative
rings using the colocalization functors.
 Hence the idea to use contraherent cosheaves in order to obtain
a local proof of the quasi-coherent cotorsion periodicity.

 The exposition in this section is based on the discussion of
local, colocal, and antilocal properties in the paper~\cite{Pal},
as well as on certain constructions with contraherent cosheaves
worked out in~\cite[Sections~2.5\+-2.6 and~4.6]{Pcosh}.
 The class of cotorsion modules over commutative rings is
\emph{both colocal and antilocal}~\cite[Examples~3.8, 6.2,
and~7.2]{Pal}; these facts are important for our argument.

\subsection{Contraadjusted and cotorsion quasi-coherent sheaves}
 Throughout the present Section~\ref{cotorsion-periodicity-secn},
\,$X$~denotes a quasi-compact semi-separated scheme.
 The notation $\Ext_X^*({-},{-})=\Ext_{X\qcoh}^*({-},{-})$ stands for
the Ext functors computed in the abelian category $X\qcoh$.
 The definition of a \emph{very flat quasi-coherent sheaf} was given
in Section~\ref{very-flat-sheaves-subsecn}.

 A quasi-coherent sheaf $\C$ on $X$ is said to be \emph{cotorsion}
if $\Ext^1_X(\F,\C)=0$ for every flat quasi-coherent sheaf $\F$ on~$X$.
 A quasi-coherent sheaf $\C$ on $X$ is said to be \emph{contraadjusted}
if $\Ext^1_X(\F,\C)=0$ for all very flat quasi-coherent sheaves $\F$
on~$X$.
 Similarly to the notation for modules, we will denoted the classes of
flat, very flat, cotorsion, and contraadjusted quasi-coherent sheaves
on $X$ by $X\qcoh_\fl$, \,$X\qcoh_\vfl$, \,$X\qcoh^\cot$, and
$X\qcoh^\cta$, respectively.

\begin{thm} \label{qcomp-qsep-flat-cotorsion-pair}
 For any quasi-compact semi-separated scheme $X$, the pair of classes of
objects (flat quasi-coherent sheaves, cotorsion quasi-coherent sheaves)
is a hereditary complete cotorsion pair in the abelian category
$X\qcoh$.
\end{thm}

 The cotorsion pair from Theorem~\ref{qcomp-qsep-flat-cotorsion-pair}
is called the \emph{flat cotorsion pair} in the category of
quasi-coherent sheaves.

\begin{proof}
 First of all, it is important that the class of flat quasi-coherent
sheaves is generating in $X\qcoh$ \,\cite[Section~2.4]{M-n},
\cite[Lemma~A.1]{EP}.
 This is where the assumption that the scheme $X$ is quasi-compact and
semi-separated is needed (cf.~\cite[Theorem~2.2]{SS}).
 Otherwise, existence of flat covers and cotorsion envelopes in
the category of quasi-coherent sheaves is known for any
scheme~\cite[Theorem~4.2]{EE}.
 Combining these two results, one comes to the conclusion that
the flat cotorsion pair in $X\qcoh$ is complete.
 To show that it is hereditary, one can notice that condition~(1)
from Lemma~\ref{garcia-rozas} holds for flat quasi-coherent sheaves,
because it holds for flat modules.

 Alternatively, the results of~\cite[Lemma~4.1.8, Corollary~4.1.9, 
Lemma~4.1.10, and Corollary~4.1.11]{Pcosh} are applicable and tell
more than the theorem claims.
 This argument uses completeness of the flat cotorsion pair in
the categories of modules over commutative rings as a black box
(see Theorem~\ref{flat-cotorsion-pair}); then the rest of
the construction is more elementary, and no further use of
set-theoretical methods is needed.
 The construction proceeds from any fixed finite affine open
covering of the scheme~$X$.
 In the context of a finite principal affine open covering of
an affine scheme, one can also find this argument spelled out
in~\cite[proof of Proposition~4.3]{Pal}. 
\end{proof}

\begin{thm} \label{qcomp-qsep-very-flat-cotorsion-pair}
 For any quasi-compact semi-separated scheme $X$, the pair of classes of
objects (very flat quasi-coherent sheaves, contraadjusted
quasi-coherent sheaves) is a hereditary complete cotorsion pair in
the abelian category $X\qcoh$.
\end{thm}

 The cotorsion pair from
Theorem~\ref{qcomp-qsep-very-flat-cotorsion-pair} is called
the \emph{very flat cotorsion pair} in the category of
quasi-coherent sheaves.

\begin{proof}
 As in the previous theorem, first of all one needs to show that
the class of very flat quasi-coherent sheaves is generating in $X\qcoh$
\,\cite[Lemma~4.1.1]{Pcosh}.
 Then completeness of the very flat cotorsion pair in $X\qcoh$ is
provable by applying a suitable version of the Eklof--Trlifaj theorem
(Theorem~\ref{eklof-trlifaj}) in the abelian category $X\qcoh$.
 To show that this cotorsion pair is hereditary, one deduces
property~(1) of Lemma~\ref{garcia-rozas} for very flat quasi-coherent
sheaves from the similar property of very flat modules
(as per Theorem~\ref{very-flat-cotorsion-pair}(a)).

 Alternatively, the results of~\cite[Lemma~4.1.1, Corollary~4.1.2,
Lemma~4.1.3, and Corollary~4.1.4]{Pcosh} are available and provide
more information than stated in the theorem.
 This argument uses completeness of the very flat cotorsion pair in
the categories of modules over commutative rings as a black box
(see Theorem~\ref{very-flat-cotorsion-pair}(a)), similarly to
the alternative proof of the previous
Theorem~\ref{qcomp-qsep-flat-cotorsion-pair}.
\end{proof}

 The additional information about the flat and very flat cotorsion
pairs in $X\qcoh$ that one gets from the proofs
in~\cite[Section~4.1]{Pcosh} is the \emph{antilocality} of
the classes of contraadjusted and cotorsion quasi-coherent sheaves
(in the terminology of~\cite{Pal}).

\subsection{The contraadjusted dimension is finite}
\label{contraadjusted-dimension-subsecn}
 The aim of this Section~\ref{cotorsion-periodicity-secn} is to
present a contraherent/colocal proof of the following theorem.
 Another proof can be found in~\cite[Corollary~9.4]{PS6}.

\begin{thm} \label{quasi-coherent-cotorsion-periodicity}
 Let $X$ be a quasi-compact semi-separated scheme, and let\/ $\C^\bu$
be an acyclic complex of cotorsion quasi-coherent sheaves on~$X$.
 Then the sheaves of cocycles of the complex\/ $\C^\bu$ are also
cotorsion.
\end{thm}

 Im the present Section~\ref{contraadjusted-dimension-subsecn}
we prove the following proposition.

\begin{prop} \label{quasi-coherent-contraadjusted-periodicity}
 Let $X$ be a quasi-compact semi-separated scheme, and let\/ $\C^\bu$
be an acyclic complex of contraadjusted quasi-coherent sheaves on~$X$.
 Then the sheaves of cocycles of the complex\/ $\C^\bu$ are also
contraadjusted.
\end{prop}

 The \emph{contraadjusted dimension} of a quasi-coherent sheaf $\M$
on the scheme $X$ is defined as its coresolution dimension with
respect to the coresolving subcategory of contraadjusted quasi-coherent
sheaves $X\qcoh^\cta\subset X\qcoh$.
 (See the beginning of
Section~\ref{finite-coresolution-dimension-subsecn}
for the definition of the coresolution dimension.)

 The following lemma can be found in~\cite[Lemma~4.6.1(a)]{Pcosh}.

\begin{lem} \label{contraadjusted-dimension-finite}
 Let $X$ be a quasi-compact semi-separated scheme, and let
$X=\bigcup_{\alpha=1}^d U_\alpha$ be a finite affine open covering
of~$X$.
 Then the contraadjusted dimensions of quasi-coherent sheaves on $X$
do not exceed~$d$.
\end{lem}

\begin{proof}
 By Theorem~\ref{qcomp-qsep-very-flat-cotorsion-pair}, the pair of
full subcategories $(X\qcoh_\vfl$, $X\qcoh^\cta)$ is a hereditary
complete cotorsion pair in $X\qcoh$.
 By Theorem~\ref{very-flat-cotorsion-pair}(d), the projective
dimensions of very flat modules over commutative rings
do not exceed~$1$.
 By~\cite[Theorem~5.3(b)]{PS6}, it follows that the projective
dimensions of very flat quasi-coherent sheaves on $X$ do not
exceed~$d$.
 Now it remains to apply
Lemma~\ref{projective-and-coresolution-dimensions} to the cotorsion
pair $(X\qcoh_\vfl$, $X\qcoh^\cta)$ in $\sE=X\qcoh$.
\end{proof}

\begin{proof}[Proof of
Proposition~\ref{quasi-coherent-contraadjusted-periodicity}]
 The assertion follows immediately from
Lemma~\ref{contraadjusted-dimension-finite}.
 One has to use the fact that the coresolution dimension does not
depend on the choice of a coresolution; see the discussion and
references in the beginning of
Section~\ref{finite-coresolution-dimension-subsecn}.
\end{proof}

\subsection{Exact categories of contraherent cosheaves}
\label{exact-categories-of-contraherent-subsecn}
 We refer to Section~\ref{locally-contraadjusted-subsecn} for
the definition of the category of locally contraadjusted contraherent
cosheaves $X\ctrh$.
 The definitions of the full subcategories of locally cotorsion
contraherent cosheaves $X\ctrh^\lct$ and locally injective
contraherent cosheaves $X\ctrh^\lin$,
$$
 X\ctrh^\lin\subset X\ctrh^\lct\subset X\ctrh,
$$
can be found in Section~\ref{locally-cotorsion-subsecn}.

 It was mentioned in Remark~\ref{nonlocality-remark} and
Section~\ref{locally-cotorsion-subsecn} that the categories $X\ctrh$,
\,$X\ctrh^\lct$, and $X\ctrh^\lin$ have natural exact structures.
 Let us define these exact category structures now.

 A short sequence of contraherent cosheaves $0\rarrow\P\rarrow\Q
\rarrow\fR\rarrow0$ on $X$ is said to be admissible exact in $X\ctrh$
if, for every affine open subscheme $U\subset X$, the short sequence
of (contraadjusted) $\cO(U)$\+modules $0\rarrow\P[U]\rarrow\Q[U]
\rarrow\fR[U]\rarrow0$ is exact in $\cO(U)\modl$.
 The admissible short exact sequences in $X\ctrh^\lct$ and $X\ctrh^\lin$
are defined similarly.

 It is clear from the definition that the full subcategories
$X\ctrh^\lct$ and $X\ctrh^\lin$ are closed under extensions and
cokernels of admissible monomorphisms in $X\ctrh$ (because
the full subcategories of cotorsion and injective $\cO(U)$\+modules
are closed under extensions and cokernels of monomorphisms in
$\cO(U)\modl$).
 A more nontrivial assertion is that, on a quasi-compact semi-separated
scheme $X$, the full subcategories $X\ctrh^\lct$ and $X\ctrh^\lin$ are
coresolving.
 In other words, every contraherent cosheaf on $X$ is an admissible
subobject of a locally injective contraherent cosheaf.
 This is the result of~\cite[Lemma~4.2.2(a)]{Pcosh}; see
Theorem~\ref{locally-injective-cotorsion-pair} below for a stronger
assertion.

 It follows that, on a quasi-compact semi-separated scheme $X$,
the functors $\Ext$ computed in the exact categories $X\ctrh$,
\,$X\ctrh^\lct$, and $X\ctrh^\lin$ agree with each other.
 We will denote this functor by $\Ext^{X,*}({-},{-})=
\Ext_{X\ctrh}^*({-},{-})$.

\subsection{Antilocal contraherent cosheaves}
\label{antilocal-contraherent-subsecn}
 A contraherent cosheaf $\P$ on the scheme $X$ is called
\emph{antilocal} (``colocally projective'' in the terminology
of~\cite[Section~4.2]{Pcosh}) if $\Ext^{X,1}(\P,\fJ)=0$ for every
locally injective contraherent cosheaf $\fJ$ on $X$
(cf.~\cite[Corollary~4.2.3(a)]{Pcosh}).
 See Section~\ref{terminological-remark-subsecn} for
the terminological discussion.

 We will denote the full subcategory of antilocal contraherent cosheaves
by $X\ctrh_\al\subset X\ctrh$.
 (The notation $X\ctrh_\clp$ is used instead in
the preprint~\cite{Pcosh}.)
 The full subcategory of antilocal locally cotorsion contraherent
cosheaves will be denoted by $X\ctrh^\lct_\al=X\ctrh_\al\cap
X\ctrh^\lct\subset X\ctrh$.

 It follows from the next theorem that $X\ctrh_\al$ is a resolving
subcategory in $X\ctrh$ and $X\ctrh^\lct_\al$ is a resolving
subcategory in $X\ctrh^\lct$.

\begin{thm} \label{locally-injective-cotorsion-pair}
 For any quasi-compact semi-separated scheme $X$, the pair of classes of
objects (antilocal contraherent cosheaves, locally injective
contraherent cosheaves) is a hereditary complete cotorsion pair in
the exact category $X\ctrh$.
\end{thm}

 The cotorsion pair from
Theorem~\ref{locally-injective-cotorsion-pair} is called
the \emph{locally injective cotorsion pair} in the category of
contraherent cosheaves.

\begin{proof}
 This is~\cite[Corollary~4.2.5(a\+-b)]{Pcosh}.
 Part~(c) of the same corollary provides additional information, viz.,
a description of antilocal contraherent cosheaves on~$X$.
 Furthermore, \cite[Corollary~4.2.6(c)]{Pcosh} provides a description
of antilocal locally cotorsion contraherent cosheaves on~$X$.
\end{proof}

 For completeness of the exposition, let us also give the following
definition.
 A contraherent cosheaf $\fF$ on the scheme $X$ is called
\emph{antilocally flat} (``colocally flat'' in the terminology
of~\cite[Section~4.3]{Pcosh}) if $\Ext^{X,1}(\fF,\Q)=0$ for every
locally cotorsion contraherent cosheaf $\Q$ on $X$
(cf.~\cite[Corollary~4.3.2(a)]{Pcosh}).

\begin{thm} \label{antilocally-flat-cotorsion-pair}
 For any quasi-compact semi-separated scheme $X$, the pair of classes of
objects (antilocally flat contraherent cosheaves, locally cotorsion
contraherent cosheaves) is a hereditary complete cotorsion pair in
the exact category $X\ctrh$.
\end{thm}

 The cotorsion pair from
Theorem~\ref{antilocally-flat-cotorsion-pair} is called
the \emph{antilocally flat cotorsion pair} in the category of
contraherent cosheaves.

\begin{proof}
 This is~\cite[Corollary~4.3.4(a\+-b)]{Pcosh}.
 Part~(c) of the same corollary provides additional information, viz.,
a description of antilocally flat contraherent cosheaves on~$X$.
\end{proof}

 The descriptions of the classes of antilocal (locally contraadjusted) 
contraherent cosheaves, antilocal locally cotorsion contraherent
cosheaves, and antilocally flat contraherent cosheaves mentioned in
the proofs of Theorems~\ref{locally-injective-cotorsion-pair}\+-%
\ref{antilocally-flat-cotorsion-pair} can be expressed by saying
that these classes of contraherent cosheaves are \emph{antilocal}
(just as their names suggest) in the terminology of~\cite{Pal}.

\subsection{Underived na\"\i ve co-contra correspondence}
 A \emph{co-contra correspondence} is an equivalence between
the coderived category of a comodule-like category and
the contraderived category of a contramodule-like category.
 A \emph{na\"\i ve co-contra correspondence} is an equivalence
between the conventional derived category of a comodule-like
category and the conventional derived category of a contramodule-like
category.
 We refer to the introductions to the papers~\cite{Pmgm,Pps} for
a detailed discussion of the philosophy of co-contra correspondence.

 In the context of quasi-coherent sheaves and contraherent cosheaves
over a quasi-compact semi-separated scheme $X$,
the \emph{underived na\"\i ve co-contra correspondence} is
an equivalence of exact categories~\cite[Lemma~4.6.7]{Pcosh}
\begin{equation} \label{underived-co-contra}
 X\qcoh^\cta\simeq X\ctrh_\al.
\end{equation}
 Furthermore, it is important for us that the category
equivalence~\eqref{underived-co-contra} restricts to an equivalence
between the full subcategories $X\qcoh^\cot\subset X\qcoh^\cta$
and $X\ctrh^\lct_\al\subset X\ctrh_\al$, providing an equivalence
of exact categories~\cite[proof of Corollary~4.6.8(a)]{Pcosh}
\begin{equation} \label{underived-co-contra-cotorsion}
 X\qcoh^\cot\simeq X\ctrh^\lct_\al.
\end{equation}

 The equivalence of categories~\eqref{underived-co-contra} is given
by the respective restrictions of the partially adjoint functors
$$
 \fHom_X(\cO_X,{-})\:X\qcoh^\cta\lrarrow X\ctrh
$$
and
$$
 \cO_X\ocn_X{-}\,\:X\ctrh\lrarrow X\qcoh
$$
defined in~\cite[Sections~2.5 and~2.6]{Pcosh}.
 It is instructive to look into the constructions of the co-contra
correspondence functors for corings over associative
rings~\cite[Sections~0.2.6 and~5.1]{Psemi}, \cite[Section~3.4]{Prev},
which served as an inspiration for the more complicated constructions
for sheaves and cosheaves over schemes.

\subsection{Details of the constructions}
 Let us briefly sketch the constructions of the functors $\fHom_X$
and~$\ocn_X$ here.

 The functor of contraherent $\fHom$ from a quasi-coherent sheaf $\M$
to a quasi-coherent sheaf $\cP$ on $X$ is defined by the rule
$$
 \fHom_X(\M,\cP)[U]=\Hom_X(j_U{}_*j_U^*\M,\cP)
$$
for all affine open subschemes $U\subset X$.
 Here $\Hom_X({-},{-})=\Hom_{X\qcoh}({-},{-})$ is the notation for
the groups of morphisms in the category of quasi-coherent sheaves
$X\qcoh$, while $j_U\:U\rarrow X$ is the identity open immersion
morphism.
 Given two embedded affine open subschemes $V\subset U\subset X$,
there is a natural morphism $j_U{}_*j_U^*\M\rarrow
j_V{}_*j_V^*\M$ of quasi-coherent sheaves on $X$, inducing
a homomorphism of $\cO(U)$\+modules $\fHom_X(\M,\cP)[V]\rarrow
\fHom_X(\M,\cP)[U]$.

 Similarly to the construction of the contraherent cosheaf
$\Cohom_X(\M,\P)$ in Section~\ref{cohom-subsecn}, the contraherence
condition~(i) from Section~\ref{locally-contraadjusted-subsecn}
holds for $\fHom_X(\M,\cP)$ for any pair of quasi-coherent sheaves
$\M$ and $\cP$ on~$X$.
 The contraadjustedness condition~(ii) for $\fHom_X(\M,\cP)$ holds under
certain assumptions on $\M$ and $\cP$, as per~\cite[Section~2.5]{Pcosh}.
 As a result, the contraherent cosheaf $\fHom_X(\M,\cP)$ is well-defined
in the following cases:
\begin{enumerate}
\item if $\F$ is a very flat quasi-coherent sheaf and $\cP$ is
a contraadjusted quasi-coherent sheaf on $X$, then
$\fHom_X(\F,\cP)$ is a locally contraadjusted contraherent cosheaf;
\item if $\F$ is a flat quasi-coherent sheaf and $\cP$ is
a cotorsion quasi-coherent sheaf, then $\fHom_X(\F,\cP)$ is a locally
cotorsion contraherent cosheaf;
\item if $\M$ is an arbitrary quasi-coherent sheaf and $\J$ is
an injective quasi-coherent sheaf, then $\fHom_X(\M,\J)$
is a locally cotorsion contraherent cosheaf;
\item if $\F$ is a flat quasi-coherent sheaf and $\J$ is
an injective quasi-coherent sheaf, then $\fHom_X(\F,\J)$
is a locally injective contraherent cosheaf.
\end{enumerate}
 So $\fHom_X\:X\qcoh^\sop\times X\qcoh\dasharrow X\ctrh$ is
a partially defined functor of two arguments
(cf.\ the discussion of partially defined functors taking values in
exact categories in Section~\ref{main-definition-discussion-subsecn}).

 For the purposes of the na\"\i ve co-contra correspondence, we
are interested in cases~(1\+-2) with the structure sheaf $\F=\cO_X$.
 Then one has
\begin{alignat}{2}
 &\fHom_X(\cO_X,{-})\:X\qcoh^\cta &&\lrarrow X\ctrh
\label{contraadjusted-to-locally-contraadjusted} \\
 &\fHom_X(\cO_X,{-})\:X\qcoh^\cot &&\lrarrow X\ctrh^\lct.
\label{cotorsion-to-locally-cotorsion}
\end{alignat}
  Both these functors are exact by construction (because
the quasi-coherent sheaves $j_U{}_*j_U^*\cO_X$ are very flat on~$X$).

 For any quasi-coherent sheaf $\M$ and contraherent cosheaf $\P$
on $X$, the \emph{contratensor product} $\M\ocn_X\P$ is constructed
as the (nondirected) colimit of the following diagram of quasi-coherent
sheaves on $X$, indexed by the poset of all affine open subschemes
$U\subset X$ ordered by inclusion.
 To every affine open subscheme $U\subset X$, the quasi-coherent sheaf
$j_U{}_*j_U^*\M\ot_{\cO_X(U)}\P[U]$ is assigned.
 Here the tensor product sign $\ot_{\cO_X(U)}$ means the tensor product
of an object $j_U{}_*j_U^*\M\in X\qcoh$ endowed with an action of
the ring $\cO_X(U)$ and a usual $\cO_X(U)$\+module~$\P[U]$.
 For any pair of embedded affine open subschemes $V\subset U\subset X$,
there is a natural isomorphism $j_V{}_*j_V^*\M\simeq
j_U{}_*j_U^*\M\ot_{\cO_X(U)}\cO_X(V)$ of quasi-coherent sheaves on $X$,
which allows to construct the morphisms in
the diagram~\cite[Section~2.6]{Pcosh}.

 The adjunction isomorphism~\cite[formula~(20) in Section~2.6]{Pcosh}
$$
 \Hom_X(\M\ocn_X\P,\>\J)\simeq\Hom^X(\P,\fHom_X(\M,\J))
$$
holds for all quasi-coherent sheaves $\M$ and $\J$ and all
contraherent cosheaves $\P$ on $X$ for which the contraherent
cosheaf $\fHom_X(\M,\J)$ is well-defined (as per items~(1\+-4) above).
 Here $\Hom^X({-},{-})=\Hom_{X\ctrh}({-},{-})$ is the notation for
the groups of morphisms in the category of contraherent cosheaves
$X\ctrh$.
 Taking $\M=\cO_X$ and making $\J$ range over the injective
quasi-coherent sheaves on $X$, one concludes from item~(4) that
the restriction of the functor $\cO_X\ocn_X{-}$ to the full subcategory
of antilocal contraherent cosheaves is an exact functor
\begin{equation} \label{antilocal-to-quasi-coherent}
 \cO_X\ocn_X{-}\,\:X\ctrh_\al\lrarrow X\qcoh
\end{equation}
\cite[proof of Lemma~4.6.7]{Pcosh}.

\subsection{Details of the arguments: antilocality}
 In order to finish the proofs of the exact category
equivalences~\eqref{underived-co-contra}
and~\eqref{underived-co-contra-cotorsion}, it still needs to be
explained that the essential image of
the functor~\eqref{contraadjusted-to-locally-contraadjusted}
is contained in the full subcategory $X\ctrh_\al\subset X\ctrh$,
the essential image of the functor~\eqref{antilocal-to-quasi-coherent}
is contained in the full subcategory $X\qcoh^\cta\subset X\qcoh$,
and the functor~\eqref{antilocal-to-quasi-coherent} takes the full
subcategory $X\ctrh^\lct_\al\subset X\ctrh_\al$ into
the full subcategory $X\qcoh^\cot\subset X\qcoh$.
 It also needs to be explained why the resulting pair of adjoint
functors between the categories $X\qcoh^\cta$ and $X\ctrh_\al$
is a category equivalence.

 The most important property of the classes of objects involved
on which these proofs are based in the \emph{antilocality}, and
the most important property of the functors involved is
the \emph{compatibility with the direct images}.

 First of all, for any affine morphism of quasi-compact semi-separated
schemes $f\:Y\rarrow X$, there is the exact functor of direct image of
quasi-coherent sheaves $f_*\:Y\qcoh\rarrow X\qcoh$ taking $Y\qcoh^\cta$
into $X\qcoh^\cta$ and $Y\qcoh^\cot$ into $X\qcoh^\cot$
\,\cite[Section~2.5]{Pcosh}.
 Dual-analogously, there is the exact functor of direct image of
contraherent cosheaves $f_!\:Y\ctrh\rarrow X\ctrh$
\,\cite[Sections~2.3 and~3.3]{Pcosh} taking $Y\ctrh^\lct$ into
$X\ctrh^\lct$ and $Y\ctrh_\al$ into $X\ctrh_\al$
\,\cite[Section~4.2]{Pcosh}.

 The direct image functors $f_*$ and~$f_*$ form a commutative square
diagram with the functors $\fHom_Y(\cO_Y,{-})$ and $\fHom_X(\cO_X,{-})$
\,\cite[formula~(44) in Section~3.8]{Pcosh}, as well as with
the functors $\cO_Y\ocn_Y{-}$ and $\cO_X\ocn_X{-}$
\,\cite[formula~(46) in Section~3.8]{Pcosh}.

 The antilocality property of classes of quasi-coherent sheaves and
contraherent cosheaves applies in the following setting.
 Let $X=\bigcup_{\alpha=1}^d U_\alpha$ be a finite affine open
covering of a quasi-compact semi-separated scheme~$X$.
 Then the contraadjusted quasi-coherent sheaves on $X$ are precisely
the direct summands of finitely iterated extensions of the direct images
of contraadjusted quasi-coherent sheaves with respect to the open
immersions $U_\alpha\rarrow X$ \,\cite[Corollary~4.1.4(c)]{Pcosh}.
 Similarly, the cotorsion quasi-coherent sheaves on $X$ are precisely
the direct summands of finitely iterated extensions of the direct
images of cotorsion quasi-coherent sheaves from the open subschemes
$U_\alpha\subset X$ \,\cite[Corollary~4.1.11(c)]{Pcosh}.

 Furthermore, the antilocal (locally contraadjusted) contraherent
cosheaves on $X$ are precisely the direct summands of finitely iterated
extensions of the direct images of contraherent cosheaves
from~$U_\alpha$ \,\cite[Corollary~4.2.5(c)]{Pcosh}.
 Similarly, the antilocal locally cotorsion contraherent cosheaves on
$X$ are precisely the direct summands of finitely iterated extensions
of the direct images of locally cotorsion contraherent cosheaves
from~$U_\alpha$ \,\cite[Corollary~4.2.6(c)]{Pcosh}.

 The proofs of the exact category
equivalences~\eqref{underived-co-contra}
and~\eqref{underived-co-contra-cotorsion} are based on these
observations.
 See~\cite[Lemma~4.6.7 and proof of Corollary~4.6.8(a)]{Pcosh}
for further details.

\subsection{Proof of quasi-coherent cotorsion periodicity}
 Now we can present our proof of the quasi-coherent cotorsion
periodicity theorem.

\begin{proof}[Proof of
Theorem~\ref{quasi-coherent-cotorsion-periodicity}]
 We are given an acyclic complex of cotorsion quasi-coherent sheaves
$\C^\bu$ on a quasi-compact semi-separated scheme~$X$.
 Obviously, any cotorsion quasi-coherent sheaf is contraadjusted.
 By Proposition~\ref{quasi-coherent-contraadjusted-periodicity},
the sheaves of cocycles of the complex $\C^\bu$ are, at least,
contraadjusted.

 Now the equivalence of exact
categories~\eqref{underived-co-contra} implies that the complex
of antilocal contraherent cosheaves $\fHom_X(\cO_X,\C^\bu)$ is
acyclic in the exact category $X\ctrh_\al$.
 Furthermore, by formula~\eqref{underived-co-contra-cotorsion},
the antilocal contraherent cosheaf $\fHom_X(\cO_X,\C^n)$ is locally
cotorsion for every $n\in\boZ$.

 For every affine open subscheme $U\subset X$, we have a complex
of $\cO(U)$\+modules $\fHom_X(\cO_X,\C^\bu)[U]$.
 By the definition of the exact structure on $X\ctrh$, which is
inherited by the resolving subcategory $X\ctrh_\al\subset X\ctrh$,
the complex of $\cO(U)$\+modules $\fHom_X(\cO_X,\C^\bu)[U]$ is acyclic.

 By the cotorsion periodicity theorem for modules over
rings~\cite[Theorem~5.1(2)]{BCE}, any acyclic complex
of cotorsion $\cO(U)$\+modules has cotorsion $\cO(U)$\+modules
of cocycles; so the $\cO(U)$\+modules of cocycles of the complex
$\fHom_X(\cO_X,\C^\bu)[U]$ are cotorsion.
 Now it is clear that the complex $\fHom_X(\cO_X,\C^\bu)[U]$ is
acyclic in the exact category of cotorsion $\cO(U)$\+modules
$\cO(U)\modl^\cot$.

 As this holds for every affine open subscheme $U\subset X$, we
can conclude that the complex of contraherent cosheaves
$\fHom_X(\cO_X,\C^\bu)$ is acyclic not only in the exact category
of antilocal locally contraadjusted contraherent cosheaves
$X\ctrh_\al$, but also in the exact category of antilocal locally
cotorsion contraherent cosheaves $X\ctrh_\al^\lct$.
 Hence the formula~\eqref{underived-co-contra-cotorsion} tells
that the complex $\C^\bu$ is acyclic in the exact category of
cotorsion quasi-coherent sheaves $X\qcoh^\cot$.
 The latter assertion means that the quasi-coherent sheaves of
cocycles of the complex $\C^\bu$ are cotorsion, as desired.
\end{proof}

\subsection{Conclusion}
 One can say, and indeed we argued in
Sections~\ref{comodules-and-contramodules-secn}\+-%
\ref{coderived-and-contraderived-secn}, that contraherent cosheaves
are an important device for globalizing contramodules and
contraderived categories.
 Then one needs contraadjusted and/or cotorsion modules over
commutative rings in order to define the contraherent cosheaves
and work with them.

 Alternatively, one may be interested in cotorsion modules for
some other, unrelated reasons.
 Then one may run into the problem that the class of cotorsion
modules is not preserved by the localization functors in commutative
algebra, and the class of cotorsion quasi-coherent sheaves on
schemes is not local.

 If one still wants to develop a local point of view on cotorsion
modules, one needs to use the colocalization functors instead of
the localizations.
 In this context, the contraherent cosheaves are indispensable.

\end{document}